\documentclass[12pt]{article}

\usepackage{latexsym}
\usepackage{amsthm}\usepackage{amsfonts}\usepackage{latexsym}

\usepackage[utf8]{inputenc}
\usepackage{amsfonts}
\usepackage{amsthm}
\usepackage{amsmath}
\usepackage{amscd}
\usepackage[mathscr]{eucal}
\usepackage{graphicx}
\usepackage{graphics}
\usepackage{pict2e}
\usepackage{epic}
\usepackage[margin=4cm]{geometry}
\usepackage{epstopdf} 
\usepackage{mathtools}
\usepackage{graphicx}
\usepackage{braket}
\usepackage{float}
\numberwithin{equation}{section}

\newtheorem{Th}{Theorem}[section]
\newtheorem{Lemma}[Th]{Lemma}
\newtheorem{Cor}[Th]{Corollary}
\newtheorem{Prop}[Th]{Proposition}

\theoremstyle{definition}
\newtheorem{Def}[Th]{Definition}

\newtheorem{Rem}[Th]{Remark}
\newtheorem{?}[Th]{Problem}

\begin{document}
	\title{\textbf{Scaling of Components \\
		in	Critical Geometric Random Graphs on 2-dim Torus}}
	\author{Vasilii Goriachkin and Tatyana Turova}
	\date{}
	\maketitle 
	\noindent 
	\small{Centre for Mathematical Sciences, Faculty of Science, Lund University, Solvegatan 18, 22100, Lund, Sweden.}
	\normalsize
	
	\begin{abstract}
		We consider random graphs on 
		the set  of $N^2$ vertices placed on the 
		discrete $2$-dimensional torus. The edges between pairs of vertices 
		are independent, and
		their probabilities decay 
		with the distance $\rho$ between these vertices
		as
		$(N\rho)^{-1}$. This is an example of an inhomogeneous random graph which is not of rank 1.
		The reported previously results on 
		the sub- and super-critical cases of this model exhibit great similarity to the classical Erd\H{o}s-R\'{e}nyi
		graphs.
		
		Here we study 
		the critical phase. 
		A diffusion approximation for the size of the largest connected component rescaled with $(N^2)^{-2/3}$
		 is derived. This completes the proof that in all regimes the model is within the same class as 
		Erd\H{o}s-R\'{e}nyi
		graph with respect to scaling of the largest component. 
	\end{abstract}

\noindent
{\it MSC-class: 	05C80, 60G42, 60G50.}

\section{Introduction}
\subsection{The Model and Main Result}\label{SectionModel}
Let $N \in \mathbb{N}$ and consider in dimension $d\geq 2$
the discretized torus $\mathbb{T}_N^d = (\mathbb{Z} / N \mathbb{Z})^d$.
Denote the set of vertices on this torus  by $$V_N = \{1, \ldots, N\}^d.$$ Hence, the number of vertices in $V_N$ is
\[|V_N|=N^d=:n(N).\]
For any two vertices $u = (u_1 , \ldots , u_d ), v= (v_1 , \ldots , v_d ) \in V_N$ define the torus distance $\rho(u, v)$  between them by
\begin{equation*}\label{distance}
\rho(u, v) =\sum_{k=1}^d\rho_N(u_k - v_k) ,
\end{equation*}
 where for any $1\leq i \leq N$
\begin{equation*}
\rho_N(i) = 
	\begin{cases}
		i, \quad & 0 \leq i \leq N/2, \\
		N - i, \quad & N/2 < i < N.
	\end{cases}
      \end{equation*}

      Consider  a  random simple graph on the set of vertices $V_N$ defined as follows.
      Let   $c>0$ and $0\leq  \alpha <d$  be fixed arbitrarily, and let $G_{N,d}^{\alpha , c}$  denote a random graph on $V_N$ where
      any two different vertices $u, v \in V_N$ are connected independently of the rest
      with probability
\begin{equation}\label{puv}
p(u, v) = \min \left \{\dfrac{c}{N^{ d-\alpha}\rho^{ \alpha}(u, v)}, 1 \right \}.
\end{equation}
We shall also write $p(u, u)=0$ for any $u\in V_N$, meaning that there are no loops in the graph.

This model was introduced in \cite{NT} as a generalization of the original model  from  \cite{JK}, which is a particular case of (\ref{puv}) when $d=2$ and $\alpha=1$.

For the range of parameters $0\leq \alpha <d$, which includes our case here, the random distance graphs can be viewed as a subclass of 
the general inhomogeneous random graph model introduced in \cite{BolJanRio2}. The relation of the models of

Observe, that when $\alpha=0$ probabilities (\ref{puv}) for all large $N$ are uniformly equal to $$\frac{c}{|V_N|}$$
for all $u, v \in V_N$, and hence
graph $G_{N,d}^{0, c}$ becomes a classical Erd\H{o}s - R\'{e}nyi random graph on $V_N$.
Otherwise, when $\alpha>0$  the probability of connection  (\ref{puv})  depends on the distance between the vertices, and therefore $G_{N,d}^{\alpha , c}$  
is called a random geometric graph (or a random distance graph).  Observe that in terms of theory of inhomogeneous random graphs \cite{BolJanRio2} probability 
(\ref{puv}) is not of rank-1 (i.e., it is not represented as a product, say $\phi(u)\phi(v)$), which usually complicates its analysis. The close relation between 
random geometric graphs \cite{P} and general inhomogeneous random graphs was disclosed already in \cite{BolJanRio1}.

Utilizing approach via inhomogeneous random graphs it was proved in 
\cite{ANT} for the case $d=2, \alpha=1$, and later in \cite{NT} for all $d>1$ and  $0< \alpha <d$,  that random geometric graphs 
$G_{N,d}^{\alpha , c}$ for all $0\leq  \alpha <d$
exhibit phase transition similar to the one in $G_{n,p}$. More precisely,  there is a critical value $c^{cr}= c^{cr}(d, \alpha) $ (see the details in \cite{NT}) such that for all $c< c^{cr}(d, \alpha)$ the largest connected component of $G_{N,d}^{\alpha , c}$ scaled to $\log |V_N|$ converges in probability
to a positive constant, while for all $c> c^{cr}(d,\alpha)$ the largest connected component of $G_{N,d}^{\alpha , c}$ scaled to $|V_N|$ converges in probability
to a positive constant.

To begin study of the critical phase of model $G_{N,d}^{\alpha , c}$, here we start with the case 
$d=2, \alpha=1$. 

From now on we fix  $d=2, \alpha=1$, and to simplify notation we shall write in  this case 
\[G_{N}^{c}= G_{N,2}^{1 , c}.\]
The critical value for this model is known from \cite{ANT}
(or \cite{JK}), it is 
\[c^{cr}(d, \alpha)\mid _{d=2, \alpha=1} = \frac{1}{4 \log 2}.\]
Note that \cite{ANT} studies the largest component in $G_{N}^{ c}$ when $c\neq c^{cr}$, while work \cite{JK}
focuses on bootstrap percolation processes on $G_{N}^{ c}$.

To complete a picture of phase transitions in $G_{N}^{ c}$  we prove that at the critical value of $c$ graph $G_{N}^{ c}$ has exactly same scaling limit of the largest component 
as in the classical random  Erd{\"o}s-R{\'e}nyi graph 
derived first by Aldous \cite{Aldous}.

\begin{Th}\label{T}
  Let $G_{N}^c $ be a random graph on the set of vertices of $\mathbb{T}_N^2$ with probabilities of connections
\begin{equation*}\label{puvcr}
p(u, v) = \min \left \{\dfrac{c}{N\rho(u, v)}, 1 \right \},
\end{equation*}
where 
$$c=c^{cr}= \frac{1}{4 \log 2}$$
is  the critical value.

Let ${\cal C}_1, {\cal C}_2, \ldots$ denote the ordered sizes of the connected components in $G_{N}^c$ with ${\cal C}_1$ being the largest one. Let $\gamma_1, \gamma_2, \ldots$ denote the ordered lengths of the excursions of the process
\[B(s)=\widetilde{W}(s)-\min_{0<t<s}\widetilde{W}(t),\]
where
\begin{equation}\label{W}
\widetilde{W}(s)=W(s)-\frac{1}{2}s^2
\end{equation}
with $W$ being  the standard Brownian motion.
 Then 
 $$ \frac{{1}}{|V_N|^{2/3}} \left ( {\cal C}_1, {\cal C}_2, \ldots \right )\stackrel{d}{\rightarrow} \left ( \gamma_1, \gamma_2, \ldots \right ), \ \mbox{ as } \ N \rightarrow \infty,$$
 with respect to $l^2$ topology on the set of infinite non-decreasing sequences $x = (x_1, x_2, \ldots)$
of non-negative values 
with  metric $d(x, y) = \left (\sum\limits_{i} (x_i - y_i)^2 \right )^{1/2}$.
\end{Th}

\subsection{Related works and open questions}\label{SectionRelatedWorks}

After an inspiring work of Aldous \cite{Aldous} who 
introduced the method of diffusion approximation for the study of (homogeneous) random graphs, 
a number of inhomogeneous models, see \cite{BHL}, \cite{Turova}, \cite{CG}, \cite{DLV}, \cite{GS}, were proved  as well to be  in the same universality class as the Erd{\"o}s-R{\'e}nyi graph
in \cite{Aldous}. Also novel scalings were  discovered in 
 \cite{BHL2} and \cite{DHL} for
the configuration model.

It must be noted that although all cited papers follow in general the idea of Aldous \cite{Aldous}, each model requires authentic tools to establish the 
diffusion approximation and to find the coefficients in
the defining equation (as (\ref{W}) above).
Furthermore, all but one (namely, \cite{DLV}) of these studies concern graphs of rank one only. In \cite{DLV}, which studies a quantum graph constructed from copies of 
Erd{\"o}s-R{\'e}nyi graphs, the authors point out some particular obstacles specific for their non- rank one model.
In general, such models have, roughly speaking, more dependencies in the structure of the graph when compared with the rank one case.

Geometric random graphs $G_{N,d}^{\alpha , c}$  with probabilities of connections (\ref{puv}) offer perhaps the most versatile model which is not of rank one. 
The model $G_{N,d}^{\alpha , c}$
provides a common framework for different classes of
graphs, because depending on the parameters it interpolates between models of long-range percolation (intensively studied as well, e.g., \cite{CS}) and classical random graphs, including various random distance graphs as e.g., in 
\cite{DHH}, \cite{JK}, \cite{ANT}.

Our proof of Theorem \ref{T} explores properties of random walk associated with the graph $G_{N}^{c}$. In particular, the number of small cycles in the graph plays an essential role. Observe, that in the  case $d=2, \alpha=1$ considered here, the expectation of the number of triangles is bounded uniformly in $N$, as it is in the classical random graph. Therefore 
in a view of this remark the result of 
Theorem \ref{T} is perhaps expected.

 However, the number of triangles in the general model
$G_{N,d}^{\alpha , c}$ undergoes itself phase transition along the parameters as reported in \cite{NT}. Therefore our
proof 
does not cover the range of all parameters of general model (\ref{puv}). A heuristic argument suggests that a similar proof should yield same results in the range $\alpha<\frac{2}{3}d$  for all $d\geq 2,$ but for $\alpha$ close to $d$ there might be extra terms in (\ref{W}). We may only conclude
 that it remains a challenge to treat the general model (\ref{puv}) at criticality.

\section {Proofs}
\subsection {Plan of the proof.}\label{SectionPlanProof}

The proof of Theorem \ref{T} follows the main idea for diffusion approximation as in \cite{Aldous}. To execute this program, after deriving in Section \ref{SectionBasicPropertiesG} some basic properties of the graph $G_N^c$, we construct in Sections \ref{SectionBFW} and \ref{TT} an exploration process, so-called breadth first walk. For that we use approach from \cite{Turova} (inspired in turn by \cite{MartinLof}). 

Sections \ref{SectionMartingale} and \ref{SectionDrift} closely follow steps from \cite{Aldous} on construction of a martingale sequence associated with our breadth first walk.
In Section \ref{SectionDrift} we state Lemma \ref{dr}
and Lemma \ref{ma} , which yield the result of Theorem \ref{T}. 

Sections \ref{SectionMixing}-\ref{SectionProof29} is a preparation for the proof of 
Lemma \ref{dr}
and Lemma \ref{ma} as follows.
Sections \ref{SectionMixing}, \ref{SectionBasicBFW}
describe properties  of the breadth first walk. 
In our case graphs are in a metric space therefore 
the exploration process here is described in terms of random walks on discrete torus. Sections 
\ref{SectionMC}, \ref{SectionRWT} elaborate on the properties of random walks on torus, which help us to derive in Sections \ref{SectionProof29} the crucial properties of our breadth first algorithm. 

Section \ref{SectionAsymPropDrift} proves Lemma \ref{dr}.

Section \ref{SectionProof27} proves Lemma \ref{ma}.

\subsection {Basic properties of  $G_N^c$.}\label{SectionBasicPropertiesG}

Let us set up notation to be used throughout the proof. 

For each vertex $v \in V_N$ let $N_r(v)$ denote the set  of vertices at distance $r$ from $v$:
$$ N_r(v) =  \left \{ u \in \mathbb{T}_N^2 : \rho(u, v) = r  \right \} . $$
Observe, that for all $v \in \mathbb{T}_N^2$ we have 
\begin{equation*}
 \left |   N_r(v)\right |=: N_r, \label{TS1}
\end{equation*}
 which 
is: for odd $N$
\begin{equation*}
N_r = 
	\begin{cases}
		4 r, \quad & 1 \leq r < N/2, \\
		4(N - r), \quad & N/2 < r < N  ,
	\end{cases}
\end{equation*}
and for even $N$
\begin{equation*}
N_r = 
	\begin{cases}
		4 r ,\quad & 1 \leq r < N/2, \\
		4 r - 2 ,\quad & r = N/2, \\
		4(N - r) ,\quad & N/2 < r < N,  \\
		1, \quad & r = N.
	\end{cases}
      \end{equation*}
      
Correspondingly, for any 
vertex $v \in V_N$ let $\mathcal{N}(v)$ denote a random set  of vertices connected to  $v$
in graph $G_N^c $:
\begin{equation}\label{Nvind}
\mathcal{N}(v):=    \left \{ u \in V_N: u \sim v \; \text{in} \; G_N^c  \right \}.
\end{equation}
Observe that by the symmetry of the model and by the independence of edges the variables
$|\mathcal{N}(v)|$, $v \in V_N$, are identically distributed random variables.
Setting also 
\[ \mathcal{N}_r(v)=    \mathcal{N}(v) \cap N_r(v),\]
to be the set of vertices at distance $r$ from $v$ which are connected to  $v$, 
 we get a representation
\begin{equation*}\label{NNr}
  \mathcal{N}(v)=\bigcup_{r\geq 1}\mathcal{N}_r(v).
      \end{equation*}

      The following Lemma is straightforward, it slightly improves the asymptotic used already in \cite{ANT} or \cite{JK}. 
      
\begin{Lemma}\label{Lm_expectation_n_vertices} 
	For any vertex $v$ in $G_N^c$ for odd $N$ 
	\begin{equation*}
		\mathbb{E} \left (  \left | \mathcal{N}(v)  \right | \right )= \sum_{u\in V_N} p(v, u)   = \sum_{r=1}^Np_r N_r = 4 c \log 2 - \dfrac{2c}{N} - \dfrac{c}{N^2} + O \left (\dfrac{1}{N^4} \right ),
	\end{equation*}
	and for even $N$
	\begin{equation*}
		\mathbb{E} \left (  \left | \mathcal{N}(v)  \right | \right ) = 4 c \log 2 - \dfrac{2c}{N} - \dfrac{2c}{N^2} + O \left (\dfrac{1}{N^4} \right ).
	\end{equation*}
      \end{Lemma}
      \hfill$\Box$

      This    result implies that in the critical case, i.e., when $$c = c^{cr}=\dfrac{1}{4 \log 2},$$  we have
      \begin{equation}\label{Ecr}
		\mathbb{E}   \left | \mathcal{N}(v)  \right |  = \sum_{u \in V_N}p(v, u)=1- \dfrac{2c}{N} -  O \left (\dfrac{1}{N^2} \right ).
	\end{equation}

        For further reference let us  collect here similar results of straightforward computations.

        \begin{Prop}\label{Prop1}
          For all $c>0$
	\begin{equation}\label{Pr1a2}
         \sum_{u\in V_N} p^2(v, u)   = \sum_{r = 1}^{N} N_r p_r^2 =
          \dfrac{4 c^2 \log N}{N^2} + O \left (\dfrac{1}{N^2} \right ).
              \end{equation}
              For any $u \in V_N $ and any subset $A \subseteq V_N$
	\begin{equation}\label{Pr1a1}
		\dfrac{c |A|}{N^2} \leq \sum\limits_{u \in A} p(v, u) \leq \dfrac{4 c \sqrt{|A|}}{N}.
	\end{equation}     
      \end{Prop}
      \hfill$\Box$
      
      We shall also use result of Lemma 1 from \cite{ANT}, stating that for any 
      $A \subseteq V_N$ with $|A|=o(N^2)$ it holds that for any $v \in V_N $ 
      	\begin{equation}\label{L1ANT}
	\sum\limits_{u \in A} p(v, u) =o(1).
	\end{equation}  

\subsection{Breadth first walk}\label{SectionBFW}
Given a graph $G_N^c$  let us describe an algorithm of revealing its connected components.

At the first  step $k=1$  choose vertex $v_1 \in V_N$ uniformly. 
Then we identify  the set of vertices  which are connected to $v_1$: 
\begin{equation*}\label{v1}
\mathcal{N}(v_1, 1) = \left \{ u : u \sim v_1 \; \text{in} \; G_N \right \}
=\mathcal{N}(v_1),
\end{equation*}
and declare vertex $v_1$ to be ``saturated''. 
This ends step 1. 

At the end of each step $k$ we shall partition the set of vertices of the graph into  the set of  unexplored vertices
and the set of used vertices, using the following notation:

$U_k$ is the set of unexplored ({\it available}) vertices,
	
$I_k$ is the set of used vertices.

\noindent
Furthermore, all the vertices within the set of used vertices $I_k$ which are not saturated 
we shall call "active",  denoting $\mathcal{A}_{k}$ their set.
Hence, at the end of the first step we set
 \begin{equation*}\label{v12}
I_1 = \{v_1\} \cup \mathcal{N}(v_1, 1),
\end{equation*}
\begin{equation*}\label{U1}
U_1 = V_N \setminus I_1, 
\end{equation*}
 \begin{equation*}\label{A1}
\mathcal{A}_1 = I_1 \setminus \{v_1\}.
\end{equation*}

Suppose at the end of $k$th step we have  defined $v_{k}, \mathcal{A}_k, U_k, I_k$. Then at the step $k+1$ we define 
$v_{k+1}, \mathcal{A}_{k+1}, U_{k+1}, I_{k+1}$ as follows.

\medskip

\noindent
Case I. If $\mathcal{A}_k \neq \emptyset$ we choose vertex $v_{k+1} $ {\it uniformly} among the vertices of 
$\mathcal{A}_k$  which  are at the {\it minimum} graph-distance from the root of the  current tree. (In other words, considering the revealed
tree-graph as a branching process,  we select vertices by generations, firstly using all vertices of the first generation, then all vertices of the second generation, and so on.)
Then we identify the set of vertices connected to $v_{k+1}$ in $G_N^c$ among the {\it available} vertices:
	\begin{equation}\label{Tse8}
	\mathcal{N}(v_{k+1}, k+1) = \left \{ u : u \sim v_{k+1} \right \} \cap U_k = \mathcal{N}(v_{k+1})\cap U_k.
\end{equation}
Then we define
\begin{equation*}\label{v13}
I_{k+1} = I_k \cup \mathcal{N}(v_{k+1}, k+1),
\end{equation*}
\begin{equation*}\label{U12}
U_{k+1}= V_N \setminus I_{k+1}, 
\end{equation*}
and
 \begin{equation*}\label{A12}
\mathcal{A}_{k+1} = I_{k+1} \setminus \{v_1, \ldots, v_{k+1}\} = \mathcal{A}_{k}  \setminus \{v_{k+1}\}\cup  \mathcal{N}(v_{k+1}, k+1).
\end{equation*}

\noindent
Case II. (a) If $\mathcal{A}_k = \emptyset$ and $U_k=\emptyset$ then we stop the algorithm.

(b) If $\mathcal{A}_k = \emptyset$ while $U_k  \neq \emptyset$ we first choose vertex $v_{k+1}$ {\it uniformly} in $U_k$, and then identify the set of vertices connected to $v_{k+1}$ in graph $G_N^c$ among the available vertices:
	\begin{equation*}\label{Tse7}
	\mathcal{N}(v_{k+1}, k+1) =  \mathcal{N}(v_{k+1})  \cap U_k .
\end{equation*}
Observe that here $v_{k+1}\in U_k$ unlike in (\ref{Tse8}). 

Finally, we define
\begin{equation*}\label{v14}
I_{k+1} = I_k \cup \{v_{k+1}\}  \cup \mathcal{N}(v_{k+1}, k+1),
\end{equation*}
\begin{equation*}\label{U13}
U_{k+1}= V_N \setminus I_{k+1}, 
\end{equation*}
and
 \begin{equation*}\label{A13}
\mathcal{A}_{k+1} = I_{k+1} \setminus \{v_1, \ldots, v_{k+1}\}=
\mathcal{N}(v_{k+1}, k+1).
\end{equation*}

This finishes $(k+1)$-st step of the algorithm.

\medskip
We shall also keep record of the vertices which are the roots of all trees revealed by our algorithm at the end of the $k$-th  step. More precisely, define for each $k\geq 1$ set $\mathcal{R}_k $ as follows. Set  $\mathcal{R}_1=\{v_1\} .$
Then 
 \begin{equation*}\label{Tm4}
\mathcal{R}_{k+1} = 
\left\{
\begin{array}{ll}
\mathcal{R}_{k} \cup \{v_{k+1}\} , & \mbox{ if } \mathcal{A}_k = \emptyset, \\
\mathcal{R}_{k}, & \mbox{ otherwise.} 
\end{array}
\right.
\end{equation*}
This gives us another representation for the set of used vertices:
\begin{equation}\label{Tm5}
I_{k} = \bigcup_{i=1}^k \mathcal{N}(v_i) \cup \mathcal{R}_{k}= \bigcup_{i=1}^k \mathcal{N}(v_i, i) \cup \mathcal{R}_{k},
\end{equation}
where the last is a decomposition into disjoint sets.

As long as consecutively $\mathcal{A}_{1} \not =\emptyset, \ldots \mathcal{A}_{k} \not =\emptyset$,
at the end of each step $k\geq 1$ we have a tree with the set of vertices $I_k$, and the edges from $v_i$ to  $\mathcal{N}(v_{i}, i)$, $i\leq k$,  inherited from the graph $G_N^c$.
By this construction we find all the vertices of the connected component of graph $G_N^c$ which contains $v_1=\mathcal{R}_{1}$,   and the 
 size of this component, denote it $C_1$, is equal to
\begin{equation}\label{C1}
C_1=\min \{k: \mathcal{A}_k = \emptyset\} = \min \{k: \left | \mathcal{A}_k \right | = 0\}= \min \{k: \left | I_k \right | = k\}.
\end{equation}

To count the number of vertices in the consecutively revealed connected components of $G_N$ we introduce the following process:
\[z(1)=0,\]
\begin{equation}\label{z}
z(k+1) = z(k) -1+ \left | \mathcal{N}(v_k, k) \right | =-k + \sum\limits_{i=1}^k \left | \mathcal{N}(v_i, i) \right | , \ \ 1\leq k<n.
\end{equation}
Observe that for all $k\geq 1$ as long as $z(k)\geq 0$ we have $$z(k)=|\mathcal{A}(k-1)|-1.$$ Hence, by (\ref{C1}) we have 
\begin{equation*}\label{C11}
C_1=\min\{k>0: z(k)=-1\}-1.
\end{equation*}
Furthermore, setting $\tau_0=1$ and for all $m\geq 1$
\begin{equation*}\label{C112}
\tau_m=\min\{k>0: z(k)=-m\},
\end{equation*}
 we get the sizes $C_m$ of consecutively revealed components
by the formula
\begin{equation*}\label{Cn}
C_m=\tau_{m}-\tau_{m-1}.
\end{equation*}

\subsection{Tree terminology.}\label{TT}

As we described above at the $k$-th step of the algorithm we have a collection of connected components of $G_N^c$, each is represented by a tree
 with root in the set $\mathcal{R}_k$. Hence, we can represent the set of revealed vertices
 $I_k$ as vertices of the trees: 
\begin{equation}\label{Tree}
T_k=\bigcup_{v_i\in \mathcal{R}_k} T_k(v_i),
\end{equation}
where $T_k(v_i)$ denotes a tree at step $k$ with root at $v_i$.
We shall also set directions for the edges of 
$T_k(v_i)$ from the root to the first generation and so on further to the leaves.

\begin{Def}\label{Conn}
  We say that two vertices $u, v \in T_k$ are directly connected if there is a (unique) path of directed edges 
  between them.
\end{Def}

If there is a path from $u$ to $v$ we denote it  $(u\rightarrow v)_{T_k}$.

\begin{Def}\label{Detpath}
The number of the edges in
path $(u\rightarrow v)_{T_k}$ 
we call the size of the path, and denote it
  $\left\|(u\rightarrow v)_{T_k}\right\|$.
  \end{Def}

Let us define the $T_k$ - distance between vertices.
  
  First we note that 
  if a pair of vertices $u, v \in T_k$ is
  in the same component of $T_k$, then there is a set of vertices, which 
  are directly connected to both: denote this set
  \[A_k(u,v)=\{a\in T_k: (a \rightarrow v)_{T_k}, (a \rightarrow u)_{T_k}\}\]
  and call it ancestry of $u, v$.

\begin{Def}\label{DetpathSize}
  The $T_k$ - distance between two vertices $u, v $ which  belong  the same component of $T_k$  is the number
  \[|u,v|_{T_k}:= \min_{a\in A_k(u,v)} \left(
    \left\|(a\rightarrow u)_{T_k}\right\|+\left\|(a\rightarrow v)_{T_k}\right\|\right).
  \]
  
  Otherwise, when  $u$ and $v$ belong to different connected components of $T_k$ we set
  $|u,v|_{T_k}=\infty$.
\end{Def}

Notice, that if $u$ is connected to $v$ by a path $(u\rightarrow v)_{T_k}$ then by Definition \ref{DetpathSize} the $T_k$ - distance between them is simply the size of this path.

Finally, for any $v_i$ we denote $R(v_i)$
the root of tree in the set (\ref{Tree}) that vertex $v_i$ belongs to. 

\subsection{Martingale associated with the breadth first walk.}\label{SectionMartingale}

Let $\{\mathcal{F}_k^z, \; k \geq 0 \}$ be the filtration  generated by the process $z$, where
\begin{equation*}\label{Fz}
 \mathcal{F}_k^z = \sigma\{ z(i), i=1, \ldots,  k\}.
\end{equation*}

Denote for any function $f:\mathbb{Z} \rightarrow \mathbb{R}$
\[\Delta f(i)=f(i+1)-f(i),  \ \ i\in \mathbb{Z}.\]
Next we define the martingale sequence $\mathcal{M}(k)$ associated with filtration $\mathcal{F}^z$: set $\mathcal{M}(0)=0$ and for all $k\geq 1$
\begin{equation}\label{Ms}
\Delta \mathcal{M}(k) = \Delta z(k)- \mathbb{E} \left\{ \Delta z(k)\left.  \right|  \mathcal{F}_{k}^z \right \}.
\end{equation}
This gives us the representation for the process $z(k)$, $k\geq 2$:
\begin{equation}\label{ms}
z(k) =  \sum_{i=1}^{k-1} \Delta z(i)=\mathcal{M}(k) + \sum_{i=1}^{k-1}\mathbb{E} \left\{ \Delta z(i)\left.  \right|  \mathcal{F}_{i}^z \right \}= : \mathcal{M}(k) + \mathcal{D}(k).
\end{equation}

Define the following rescaled processes on an arbitrarily fixed interval $[0,T]$
\[
\begin{array}{ll}
{\widetilde {\cal D}} (s)& = {n^{-1/3}} \ {\cal D} (1+ [  n^{2/3} s ] ),   \\ \\
{\widetilde {\cal M}} (s)&  = {n^{-1/3}}
 {\cal M}(1+ [ n^{2/3} s ] ), \ \ \ \ s\in [0,T],
\end{array}
\]
where $n=|V(N)|=N^2$.
We  study these processes separately in the following Lemmas.
\begin{Lemma}\label{dr}
  On an arbitrarily fixed interval $[0,T]$
\begin{equation}\label{M5}
{\widetilde {\cal D}} (s) \stackrel{P}{\rightarrow} -
\frac{1}{2}
s^2.
\end{equation}
uniformly in $s\in [0,T]$
 as $n\rightarrow \infty$.
\end{Lemma}

\begin{Lemma}\label{ma}
  On an arbitrarily fixed interval $[0,T]$
\begin{equation}\label{M4}
{\widetilde {\cal M}} (s) 
\stackrel{d}{\rightarrow}
\ W(s),
\end{equation}
uniformly in $s\in [0,T]$
 as $n\rightarrow \infty$.
\end{Lemma}

The results (\ref{M5}) and (\ref{M4})  imply 
\begin{equation}\label{Ts1}
{n^{-1/3}} \ z (1+ [ n^{2/3} s ])\stackrel{d}{\rightarrow} W(s)-
\frac{1}{2}
s^2,
\end{equation}
from where the statement of Theorem \ref{T} follows (as  in \cite{Aldous}).
Therefore we aim to prove the last two lemmas. 

\subsection{Drift $\mathcal{D}(k)$}\label{SectionDrift}

Let $k>2$. By the definition (\ref{z}) we have 
\begin{equation*}\label{Tf1}
\mathbb{E} \left\{ \Delta z(k)\left.  \right|  \mathcal{F}_k^z \right \}=-1 +  \mathbb{E} \left\{  \left | \mathcal{N}(v_k, k) \right | \big| \mathcal{F}_{k}^z  \right\} ,
\end{equation*}
and therefore using also  (\ref{ms}) we get 
\[
\mathcal{D}(k+1) =\sum\limits_{j=1}^{k}  
\mathbb{E} \left\{ \Delta z(j)\left.  \right|  \mathcal{F}_j^z \right \}
\]
\begin{equation}\label{Dk}
= -k +\mathbb{E}  | \mathcal{N}(v_1)  |    + \sum\limits_{j=2}^{k}  \mathbb{E} \left\{  \left | \mathcal{N}(v_j, j) \right | \big| \mathcal{F}_{j}^z  \right\} .
\end{equation}
Let us use notation
\[\mathcal{N}_r(v_j) = \mathcal{N}  (v_j) \cap  {N}_r(v_j) , \]
\[\mathcal{N}_r(v_j,j) = \mathcal{N}  (v_j,j) \cap  {N}_r(v_j) .\]

Observe that 
\begin{equation*}\label{FUI}
\mathcal{F}_j^z \subseteq \mathcal{F}_{j-1}:=\sigma\{ I_i, v_i, i=1, \ldots,  j-1\} \subseteq \mathcal{F}_{j},
\end{equation*}
and by the construction 
\begin{equation*}\label{Tse9}
\mathcal{N}_r(v_j,j) \in Bin (| U_{j-1} \cap {N}_r(v_j)|, p_r) .
\end{equation*}

Hence,   using decomposition 
\begin{equation}\label{TSe2}
\mathcal{N}(v_j, j) = \mathcal{N}(v_j) 
\cap U_{j-1} =  \bigcup_{r=1}^N\mathcal{N}_r(v_j) 
\setminus \left( \mathcal{N}_r(v_j) \cap I_{j-1} \right) 
\end{equation}
we get for all $j>1$
\[\mathbb{E} \left\{  \left | \mathcal{N}(v_j, j) \right | \big| \mathcal{F}_{j}^z  \right\} =
\mathbb{E} 
\left\{ \mathbb{E} \left\{ 
\left | \mathcal{N}(v_j) \cap U_{j-1}\right | \big| \mathcal{F}_{j-1}  
\right\} 
\big|\mathcal{F}_{j}^z  \right\} 
\]

\[= \sum_{r=1}^Np_r
\mathbb{E} 
\left\{ |N_r(v_j)|
\big|\mathcal{F}_{j}^z  
\right\} -\sum_{r=1}^Np_r
\mathbb{E} 
\left\{ |N_r(v_j)\cap I_{j-1}|
\big|\mathcal{F}_{j}^z  
\right\} \]
\begin{equation}\label{TSe4}
= \sum_{r=1}^Np_rN_r
-\sum_{r=1}^Np_r
\mathbb{E} 
\left\{ 
|N_r(v_j)\cap I_{j-1}|
\big|\mathcal{F}_{j}^z  
\right\}, 
\end{equation}
where the last equality is due to (\ref{TS1}).
With a help of  representation (\ref{Tm5})  we get for all $j>1$
\[
\mathbb{E} \left\{ \left| {N}_r(v_j) \cap I_{j-1} \right| \big| \mathcal{F}_{j}^z  \right\}
\]
\begin{equation}\label{Tm1}
= \sum_{i=1}^{j-1}
\mathbb{E} \left\{ \left|{N_r}(v_j) \cap \mathcal{N}(v_i,i)\right|  \big| \mathcal{F}_{j}^z  \right\}
\end{equation}
\[
+\mathbb{E} \left\{ \left|{N_r}(v_j) \cap \{v_1\} \right| \big| \mathcal{F}_{j}^z  \right\}
+\sum_{i=2}^{j-1} 
\mathbb{E} \left\{ {\bf I}_{{\cal A}_{i-1}=\emptyset} \left|{N_r}(v_j) \cap \{v_i\} \right| \big| \mathcal{F}_{j}^z  \right\},
\]
where again by (\ref{TSe2}) and representation (\ref{Tm5}) we have
\begin{equation}\label{TSe3}
\mathbb{E} \left\{ \left|{N_r}(v_j) \cap \mathcal{N}(v_i,i)\right|  \big| \mathcal{F}_{j}^z  \right\}
\end{equation}
\[= \mathbb{E} \left\{ 
\mathbb{E} \left\{ \left|{N_r}(v_j) \cap \mathcal{N}(v_i)  \right| \big| \mathcal{F}_{j-1}  \right\} \big| \mathcal{F}_{j}^z  \right\}
\]
\[- \mathbb{E} \left\{ 
\sum_{s=1}^N\sum_{l=1}^{i-1}\mathbb{E} \left\{ \left|{N_r}(v_j) \cap \mathcal{N}_s(v_i) \cap \mathcal{N}(v_l,l) 
\right| \big| \mathcal{F}_{j-1}  \right\} \big| \mathcal{F}_{j}^z  \right\}
\]
\[- \mathbb{E} \left\{ 
\sum_{s=1}^N\sum_{l=1}^{i-1}\mathbb{E} \left\{ {\bf I}_{ {\cal A}_{l-1}= \emptyset }\left|{N_r}(v_j) \cap \mathcal{N}_s(v_i) \cap \{v_l\}
\right| \big| \mathcal{F}_{j-1}  \right\} \big| \mathcal{F}_{j}^z  \right\}.
\]
\begin{Prop}\label{Po1}
	For all $1\leq i< j \leq n$ 
	\begin{equation*}
	\left\{ {\cal A}_{i}= \emptyset \right\} \in \mathcal{F}_{j}^z.
	\end{equation*}
	
	\begin{proof}
		\[
		\left\{ {\cal A}_{i}= \emptyset \right\} = \left \{ z(i+1) = \min\limits_{l \leq i} z(l+1) \right \} \in \mathcal{F}_{j}^z ,
		\]
		because $i+1 \leq j$.
	\end{proof}
\end{Prop}

Substituting (\ref{TSe3}),  (\ref{TSe4}), (\ref{Tm1}) into (\ref{Dk}), and making use of Proposition \ref{Po1} we get
\begin{equation}\label{TSe5}
\mathcal{D}(k+1) =-k +\mathbb{E}  | \mathcal{N}(v_1)  | + \sum\limits_{j=2}^{k}  
\sum_{r=1}^Np_r N_r
\end{equation}
\[- \sum\limits_{j=2}^{k}  \sum_{r=1}^N p_r\sum_{i=1}^{j-1} \mathbb{E} \left\{ 
\sum_{s=1}^N\mathbb{E} \left\{ \left|{N_r}(v_j) \cap \mathcal{N}_s(v_i)  \right| \big| \mathcal{F}_{j-1}  \right\} \big| \mathcal{F}_{j}^z  \right\}\]
\[-\sum\limits_{j=2}^{k}  \sum_{r=1}^N p_r\sum_{i=1}^{j-1} {\bf I}_{{\cal A}_{i-1}=\emptyset} 
\mathbb{P} \left\{ v_i \in {N_r}(v_j) \big| \mathcal{F}_{j}^z  \right\}\]
\[+\sum\limits_{j=2}^{k}  \sum_{r=1}^N p_r\sum_{i=1}^{j-1}
\mathbb{E} \left\{ 
\sum_{s=1}^N\sum_{l=1}^{i-1}\mathbb{E} \left\{ \left|{N_r}(v_j) \cap \mathcal{N}_s(v_i) \cap \mathcal{N}(v_l,l) 
\right| \big| \mathcal{F}_{j-1}  \right\} \big| \mathcal{F}_{j}^z  \right\}
\]
\[+\sum\limits_{j=2}^{k}  \sum_{r=1}^N p_r\sum_{i=1}^{j-1}\mathbb{E} \left\{ 
\sum_{s=1}^N\sum_{l=1}^{i-1}\mathbb{E} \left\{ {\bf I}_{ {\cal A}_{l-1}= \emptyset }\left|{N_r}(v_j) \cap \mathcal{N}_s(v_i) \cap \{v_l\}
\right| \big| \mathcal{F}_{j-1}  \right\} \big| \mathcal{F}_{j}^z  \right\}.
\]
Taking into account the result of the Lemma \ref{Lm_expectation_n_vertices} we derive from (\ref{TSe5})
\begin{eqnarray}\label{TSe10}
&& \mathcal{D}(k+1) =-k\left(\dfrac{2c}{N} +O\left(\dfrac{1}{N^2}\right) \right) \\ \nonumber
&& - \sum\limits_{j=2}^{k}  \sum_{r=1}^N p_r\sum_{i=1}^{j-1} \mathbb{E} \left\{ 
\sum_{s=1}^N\mathbb{E} \left\{ \left|{N_r}(v_j) \cap \mathcal{N}_s(v_i)  \right| \big| \mathcal{F}_{j-1}  \right\} \big| \mathcal{F}_{j}^z  \right\} \\ \nonumber
&& -\sum\limits_{j=2}^{k}  \sum_{r=1}^N p_r\sum_{i=1}^{j-1} {\bf I}_{{\cal A}_{i-1}=\emptyset} 
\mathbb{P} \left\{ v_i \in {N_r}(v_j) \big| \mathcal{F}_{j}^z  \right\} \\ \nonumber
&& +\sum\limits_{j=2}^{k}  \sum_{r=1}^N p_r\sum_{i=1}^{j-1}
\mathbb{E} \left\{ 
\sum_{s=1}^N\sum_{l=1}^{i-1}\mathbb{E} \left\{ \left|{N_r}(v_j) \cap \mathcal{N}_s(v_i) \cap \mathcal{N}(v_l,l) 
\right| \big| \mathcal{F}_{j-1}  \right\} \big| \mathcal{F}_{j}^z  \right\} \\ \nonumber
&& +\sum\limits_{j=2}^{k}  \sum_{r=1}^N p_r\sum_{i=1}^{j-1} 
\sum_{l=1}^{i-1} {\bf I}_{ {\cal A}_{l-1}= \emptyset }
\mathbb{P} \left\{ v_l \in {N_r}(v_j) \cap \mathcal{N}(v_i) 
\big| \mathcal{F}_{j}^z  \right\}.
\end{eqnarray}
\[=: -k\left(\dfrac{2c}{N} +O\left(\dfrac{1}{N^2}\right) \right)
- \mathcal{D}_1(k+1) - \mathcal{D}_2(k+1) +
\mathcal{D}_3(k+1) +\mathcal{D}_4(k+1) .
\]

It appears clear that in order to analyse further the drift (\ref{TSe10}) and to establish (\ref{M5}) we need a joint distribution of random vertices $v_i$.

 \subsection{Mixing properties of the breadth first walk.}\label{SectionMixing}
 
 The key property of the introduced breadth first walk is the asymptotic uniformity of the distribution of vertices $v_i$ on torus, which is established in the following theorem.
 Recall that the number of vertices in graph $G_{N}^{c}$ equals $|V_N| = N^2 = n$.

\begin{Th}\label{ThIndep}
  Let $T>0$ be fixed arbitrarily.
  Then for all $a \neq b \in V_N$ 
  it holds that
  \begin{equation}\label{TM28*}
	\mathbb{P} \left \{ v_j = b \big| v_i = a \right \} \leq \max \left ( \dfrac{4 c^2 \log N}{N^2}, p(a, b) \right ) \left ( 1 + o(1) \right )
      \end{equation}
      uniformly in 
      $3 \leq i< j \leq k\leq  T|V_N|^{2/3}$.
      
      If additionally
      \[
  j - i \geq |V_N|^{1/6} \]
then 
      \begin{equation}\label{TM28}
	\mathbb{P} \left \{ v_j = b \big| v_i = a \right \} = \dfrac{1}{|V_N|} + o\left (\dfrac{1}{|V_N|} \right )
      \end{equation}
      uniformly in 
      $ |V_N|^{1/6}<i+|V_N|^{1/6} \leq j \leq k\leq  T|V_N|^{2/3}$,
      and $a \neq b \in V_N$.
\end{Th}

In fact, precisely the last property (\ref{TM28}) of fast mixing is determining for the construction of our breadth-first walk, and it is fundamental for the remaining proof.

Usefulness of such a simple asymptotic formula as (\ref{TM28}) for the analysis of (\ref{TSe10}) leaves no doubts.
To establish (\ref{TM28}) we shall study  separately probability as in  (\ref{TM28}) but conditioning further on different events.

Before we outline the proof of Theorem \ref{ThIndep} let us establish 
some useful facts about breadth-first walk. 

\begin{Lemma}\label{LmEstI_k2}
	For any  $1 \leq k \leq |V_N|=N^2 $ and any constant  $C>e$
	\begin{equation}\label{LmEstI_k2_ineq}
		\mathbb{P} \left ( |I_k| \geq k C \right ) \leq e^{-k \left ( C - e \right )},
              \end{equation}
              at least for all large $N$.
            \end{Lemma}
        \begin{Rem}\label{TR}
        	$C$ in (\ref{LmEstI_k2_ineq}) is a constant, but can be chosen as a function of $N$, e.g., $C=N^{\varepsilon}.$
        \end{Rem}

\begin{proof}
  Observe that for all $k\geq 1$
	\begin{equation}\label{TM1}
		|I_k| \leq \sum\limits_{i = 1}^{k} |\mathcal{N} (v_i, i)| + k \leq \sum\limits_{i = 1}^{k} |\mathcal{N} (v_i)| + k.
              \end{equation}

              Recall (see (\ref{Nvind}))  that  $|\mathcal{N} (v_i)|, i\geq 1,$ are identically distributed random variables. Let
	\begin{equation}\label{eta}
              \eta_{uv} \in Be\Big(p(u,v)\Big),  \ \ \ \ \ \ u,v \in V_N,
            \end{equation}
            be independent Bernoulli random variables $Be(p(u,v))$, where $p(u,v)$ is the probability of the edge  $(u,v)$ in $G_N^c$. In other words, $\eta_{uv}$ can be regarded as an indicator of the edge $(u,v)$ in $G_N^c$.
              Then
              \[\mathcal{N} (v_i) = \sum_{v\neq v_i}\eta_{v_i v},\]
              and it is standard to show (see the details, e.g., in \cite{ANT}) that $|\mathcal{N} (v_i) |$ is stochastically dominated by
Poisson distributed random variable with mean value 
\begin{equation}\label{meanE}
\sum\limits_{r = 1}^{N} N_r | \log (1 - p_r)|\leq 1
 \end{equation}
(recall (\ref{Ecr})) for all large $N$. 
Hence, $|\mathcal{N} (v_i)|$ is stochastically dominated by Poisson distributed
variable with mean 1. 

Introducing now  independent copies  $\xi_1, \xi_2, ..., \xi_k$ of a Poisson random variable with mean 1, and making use of 
 (\ref{TM1}) we get
	\begin{eqnarray}
		&& \mathbb{P} \left ( |I_k| \geq k C \right ) \leq \mathbb{P} \left ( \sum\limits_{i = 1}^{k} |\mathcal{N} (v_i, i)| + k \geq k C \right ) \nonumber \\ \nonumber
		&& \leq \mathbb{P} \left ( \sum\limits_{i = 1}^{k} \xi_i + k \geq kC \right ) = \mathbb{P} \left ( \sum\limits_{i = 1}^{k} \xi_i \geq k \left ( C - 1\right ) \right ) \\ \nonumber
		&& \leq \dfrac{\mathbb{E} \left ( \exp \left (\sum\limits_{i = 1}^{k} \xi_i \right ) \right )}{\exp \left (k \left ( C - 1 \right ) \right )} = \dfrac{\exp (k \left ( e - 1 \right )) }{\exp \left (k \left ( C - 1 \right ) \right )} = e^{-k \left ( C - e \right )}.
	\end{eqnarray}
\end{proof}

We shall mostly use the following immediate corollary of the last lemma.

\begin{Cor}\label{TC1}
Let  $T>0$  be fixed arbitrarily. Then for any $C>2e T$  one has
\begin{equation*}\label{Ik}
\mathbb{P} \left ( |I_k| \geq |V_N|^{2/3} C\right ) \leq e^{-|V_N|^{2/3} C/2
	},
\end{equation*}
uniformly in 
$1 \leq k \leq T|V_N|^{2/3} $
for all large $N$.\hfill$\Box$
\end{Cor}

Note that  the uniform distribution of $v_1$ on $V_N$, yields as well the uniform distribution  of $v_i$, for any fixed $i$. Hence, we  derive immediately from
Corollary \ref{TC1}
a conditional version as follows.

\begin{Cor}\label{TCo2}
  Under assumption of Corollary \ref{TC1}
\begin{equation*}\label{Ik2}
\mathbb{P} \left ( |I_k| \geq |V_N|^{2/3} C \Big| v_i=a \right ) \leq e^{-|V_N|^{2/3} C/2
	},
\end{equation*}
uniformly in $a\in V_N$ and
$1 \leq i \leq k \leq T|V_N|^{2/3} $
for all large $N$. \hfill$\Box$
\end{Cor}

Return now to Theorem \ref{ThIndep}. With a help of Corollary \ref{TCo2} we get first
\begin{equation}\label{TM28nn}
\mathbb{P} \left \{ v_j = b \big| v_i = a \right \} 
\end{equation}
\[
=\mathbb{P} 
\left\{ v_j = b \big| v_i = a,  |I_k| \leq |V_N|^{2/3} C
\right\} 
+ O\left( e^{-|V_N|^{2/3}} \right), \]
where $C>2e T$ is fixed arbitrarily.
Then we split probability in (\ref{TM28nn})
distinguishing in particular, case when $v_j $ and $v_i$ belong to different components of $T_k$, (i.e.,  $|v_i,v_j|_{T_k}=\infty$, in terms defined in Section \ref{TT})
as follows
\begin{equation}\label{TM28n}
\mathbb{P} \left \{ v_j = b \big| v_i = a \right \} 
\end{equation}	
\[=\mathbb{P} \left \{ v_j = b \big| v_i = a, |v_i,v_j|_{T_k}=\infty ,
 |I_k| \leq |V_N|^{2/3} C\right \} 	\]
 \[\times
 \mathbb{P} \left \{  |v_i,v_j|_{T_k}=\infty \big| v_i = a,
 |I_k| \leq |V_N|^{2/3} C\right \}\]
 \[+ \sum_{L\geq 1}\mathbb{P} \left \{ v_j = b \big| v_i = a, |v_i,v_j|_{T_k}=L ,
 |I_k| \leq |V_N|^{2/3} C\right \} 	\]
 \[\times
 \mathbb{P} \left \{  |v_i,v_j|_{T_k}=L \big| v_i = a,
 |I_k| \leq |V_N|^{2/3} C\right \}
 + O\left( e^{-|V_N|^{2/3}} \right).
 	 \]
The proof of Theorem \ref{ThIndep} will be based 
on this decomposition and the following three lemmas, which study separately terms on the right in (\ref{TM28n}).

 \begin{Lemma}\label{LT1}
 	Let $k = |V_N|^{2/3} s, \; s \in [0, T]$. Then for all $a \neq b \in V_N$ and for all $1 \leq i< j \leq k$
\begin{equation}\label{TM29}
  \mathbb{P} \left \{ v_j = b \big| v_i = a,    |v_i,v_j|_{T_k}=\infty ,
 |I_k| \leq |V_N|^{2/3} C
  \right \} = \dfrac{1}{N^2} + o\left (\dfrac{1}{N^2} \right )
\end{equation}
uniformly in $k\leq T |V_N|^{2/3}$.
\end{Lemma}

Next lemma establishes a similar to (\ref{TM29}) asymptotic but under assumptions that $v_j $ and $v_i$ belong to the same component of $T_k$, and moreover that the distance between them is large.
\begin{Lemma}\label{LT2}
  Let $T>0$ be fixed arbitrarily. There is a positive $C_1$ such that for 
  any $$C_1\log N<L <k\leq T|V_N|^{2/3}, $$ 
  and for all $a \neq b$ in $V_N$ one has 
  \begin{equation}\label{TM30*}
     \mathbb{P} \left \{ v_j = b \big| v_i = a,    |v_i,v_j|_{T_k}=L,
         |I_k| \leq |V_N|^{2/3} C \right \}
\end{equation}
       \[ \leq \max \left ( \dfrac{4 c^2 \log N}{N^2}, p(a, b) \right ) \left ( 1 + o(1) \right )
	\]
	uniformly in $C_1\log N<k\leq T |V_N|^{2/3}$ and $1 \leq i< j \leq k$.
	
        If also
\begin{equation}\label{cMk}
        j-i \geq |V_N|^{1/6},
\end{equation}
        then 
          \begin{equation}\label{TM30}
         \mathbb{P} \left \{ v_j = b \big| v_i = a,    |v_i,v_j|_{T_k}=L,
         |I_k| \leq |V_N|^{2/3} C \right \} = \dfrac{1}{N^2} + o\left (\dfrac{1}{N^2} \right )
	\end{equation}
	uniformly in $1+|V_N|^{1/6} \leq i+|V_N|^{1/6} \leq j \leq k$ and all $a \neq b$ in $V_N$
\end{Lemma}

Finally, for the case when $v_j $ and $v_i$ belong to the same component of $T_k$ and moreover at short distance from each other, we get the following bound.

\begin{Lemma}\label{LT3}
Let $T>0$ and $C_1>0$ be fixed arbitrarily. 
Then for all 
	$$L\leq C_1\log N, $$
any	$i,j$ such that $$
1+|V_N|^{1/6} \leq i+|V_N|^{1/6} \leq j \leq k\leq |V_N|^{2/3},$$ $($i.e., $j-i \geq |V_N|^{1/6})$ and $a \in V_N,$ one has
\begin{equation*}\label{TJ32}
\mathbb{P} \left \{|v_i,v_j|_{T_k}=L
\big| v_i = a,    
|I_k| \leq C|V_N|^{2/3} 
\right \} \leq 2 {C_1} \dfrac{\log N}{  N^{1/3} }
\end{equation*}
uniformly for all large $N$.
\end{Lemma}

Proofs of the last two lemmas explore a natural relation of the breadth-first walk to a random walk on a random graphs. 
Therefore we devote next sections to the study of random walks associated with the breadth-first walk. 

\subsection{Basic properties of the breadth-first walk. }\label{SectionBasicBFW}

\begin{Lemma}\label{Lm_Est_ENvjjFjz}
  Let $T>0$  and  $s \in [0, T]$ be fixed arbitrarily.
Then uniformly in 
   $1 \leq  j \leq k = n^{2/3} s = N^{4/3} s$
   \begin{equation}\label{tm10}
   1 - \dfrac{4 c \sqrt{|I_{j-1}|}}{N} + O\left(\frac{1}{N} \right) \leq \mathbb{E} \left ( |\mathcal{N} (v_j, j)| \big| \mathcal{F}_j^z \right )  \leq 1 + O\left(\frac{1}{N} \right),
\end{equation}
   and
   \begin{equation}\label{tm11}
     2  - 4 \dfrac{4 c \sqrt{|I_{j-1}|}}{N} 
     +O\left(\frac{1}{N} \right)
     \leq \mathbb{E} \left ( |\mathcal{N} (v_j, j)|^2 \big| \mathcal{F}_j^z \right ) \leq 2 + O\left(\frac{1}{N} \right)
     \end{equation}
   as $N \rightarrow \infty.$

	\begin{proof}
	
                Consider with a help of  random variables introduced in (\ref{eta})
	\[
                  \mathbb{E}
                  \left\{|\mathcal{N} (v_j, j)|
                    \big| \mathcal{F}_j^z
                  \right \}     =
                     \mathbb{E}
                  \left\{\sum\limits_{u \in U_{j-1}} \eta_{v_j u}
                    \big| \mathcal{F}_j^z
                  \right \}
\]
		\begin{equation*}
                  =   \mathbb{E} \left\{
                    \left ( \sum\limits_{u \in V_N } \eta_{v_j u}
-\sum\limits_{u \in I_{j-1}} \eta_{v_j u}
                    \right )\big| \mathcal{F}_j^z
                  \right \}
                \end{equation*}
\[    =    \sum\limits_{u \in V_N } p(v, u)
-\mathbb{E} \left\{\sum\limits_{u \in I_{j-1}} p(v_j, u)
\big| \mathcal{F}_j^z
\right \},
\]
      where $v\in V_N$ is an arbitrarily  fixed vertex.
      Making use of  Lemma \ref{Lm_expectation_n_vertices} and bound (\ref{Pr1a1}) we derive from here
		\begin{equation*}\label{tm4}
       1 + O\left(\frac{1}{N} \right) \geq            \mathbb{E} \left ( |\mathcal{N} (v_j, j)| \big| \mathcal{F}_j^z \right ) =  1 + O\left(\frac{1}{N} \right) - \mathbb{E} \left ( \sum\limits_{u \in I_{j-1}} p(v_j, u) \big| \mathcal{F}_j^z \right )
     \end{equation*}             
		\[  \geq 1 + O\left(\frac{1}{N} \right) - \dfrac{4 c \sqrt{|I_{j-1}|}}{N},\]
		which yields (\ref{tm10}).
                     
 In the same manner  consider now
 \[
 \mathbb{E} \left\{
 |\mathcal{N} (v_j, j)|^2 \big| \mathcal{F}_j^z \right \}
 \]
		\begin{equation*}\label{tm5}
                   = \mathbb{E} \left \{ \left ( \sum\limits_{u \in U_{j-1}} p(v_j, u) + \sum\limits_{u,w \in U_{j-1}: \ u\neq w} p(v_j, u) p(v_j, w) \right)
                     \big| \mathcal{F}_j^z \right \}
          \end{equation*}
          \[          
                   =
                     \mathbb{E} \left \{ \left (
\sum\limits_{u \in V_{N}} p(v_j, u)-
                     \sum\limits_{u \in I_{j-1}} p(v_j, u)
                     \right)
                     \big| \mathcal{F}_j^z \right \}\]
                     \[
                  +
                     \sum\limits_{u,w \in V_N} \mathbb{E} \left \{ p(v_j, u) p(v_j, w) 
                     \big| \mathcal{F}_j^z \right \}
-
                     \sum\limits_{u \in V_N} \mathbb{E} \left \{ p^2(v_j, u) 
                     \big| \mathcal{F}_j^z \right \}\]
                     \[
-2\mathbb{E} \left \{ \sum\limits_{u \in V_N,w \in I_{j-1}: \ u\neq w} p(v_j, u) p(v_j, w)
                     \big| \mathcal{F}_j^z \right \}.
                \]
                Using the translation invariance of the probabilities of edges 
                we can rewrite the last formula as follows
                \[
                \mathbb{E} \left\{
                |\mathcal{N} (v_j, j)|^2 \big| \mathcal{F}_j^z \right \}
                \]
                \begin{equation} \label{tm6}        
               = \sum\limits_{u \in V_{N}} p(v, u)
                - \mathbb{E} \left\{
              \sum\limits_{u \in I_{j-1}} p(v_j, u)
                \big| \mathcal{F}_j^z \right \}
                \end{equation}
                \[+
                \left( 
                \sum\limits_{u \in V_N}   p(v, u) \right) ^2
                -
                \sum\limits_{u \in V_N}   p^2(v, u) 
                -2\mathbb{E} \left \{ \sum\limits_{u \in V_N,w \in I_{j-1}: \ u\neq w} p(v_j, u) p(v_j, w)
                \big| \mathcal{F}_j^z \right \},
                \]
                where $v\in V_N$ is fixed arbitrarily.
      Then with a help of {Lemma} \ref{Lm_expectation_n_vertices}  we immediately derive from here the uniform upper bound:
      \[\mathbb{E} \left\{
      |\mathcal{N} (v_j, j)|^2 \big| \mathcal{F}_j^z \right \}
      \]
      \begin{equation}\label{tm8}
      \leq \sum\limits_{u \in V_{N}} p(v, u)
      +\left( 
      \sum\limits_{u \in V_N}   p(v, u) \right) ^2=2+ O\left(\frac{1}{N}\right).
      \end{equation}
      On the other hand, taking also into account 
      bounds (\ref{Pr1a2}) and (\ref{Pr1a1}) we get from (\ref{tm6}) the lower bound as well:      
      \[
      \mathbb{E} \left\{
      |\mathcal{N} (v_j, j)|^2 \big| \mathcal{F}_j^z \right \}
      \]
      \begin{equation} \label{tm9}        
      \geq 2+ O\left(\frac{1}{N}\right)
      - \dfrac{4 c \sqrt{|I_{j-1}|}}{N}  \left(1+2\left(1+ O\left(\frac{1}{N}\right)\right)\right)-5c^2\frac{\log N}{N^2}.
      \end{equation}
      The last two bounds (\ref{tm9}) and (\ref{tm8}) confirm (\ref{tm11}) and finish the proof of the Lemma. 
	\end{proof}
\end{Lemma}

As a corollary of the previous lemma we get the following result.
\begin{Lemma}\label{Lm_Est_ENvjj}
 Let $T>0$   be fixed arbitrarily.
Then uniformly in 
$1 \leq  j \leq k \leq T |V_N|^{2/3}=TN^{4/3} $
	\[
		\mathbb{E} |\mathcal{N} (v_j, j)|  = 1 + O\left({N^{-1/4}}\right),\]
		and
		\[ \mathbb{E} |\mathcal{N} (v_j, j)|^2 = 2 + O\left({N^{-1/4}}\right) .\]
	\begin{proof}          
          Note that by Lemma \ref{LmEstI_k2} for an arbitrarily fixed $\varepsilon$ and any  $j\leq Tn^{2/3}$, where $n=|V_N|=N^2$
          \begin{equation}\label{tm12}
          \mathbb{E}\sqrt{|I_{j-1}|} \leq \sqrt{|Tn^{2/3+\varepsilon}|} +
          \sqrt{|V_N|}e^{-(n^{\varepsilon}-e)}=O\left(
          n^{1/3+\varepsilon /2}
          \right)=O\left(
          N^{2/3+\varepsilon}
          \right),
          \end{equation}
          hence
          \[\frac{\mathbb{E}\sqrt{|I_{j-1}|} }{N}=N^{-1/3+\varepsilon}.\]
          Combining this bound with result of Lemma \ref{Lm_Est_ENvjjFjz} we get
          the statements of the Lemma.
                	\end{proof}
\end{Lemma}

\subsection{Markov chains associated with random graph.}\label{SectionMC}

First we study an auxiliary Markov chain 
 ${\bf X}=\{X_k\}_{k = 1}^{\infty}$  on 
the state space $V_N$ with  transition probability matrix defined by the probabilities of edges in $G_N^c$ as follows
$$P: = (P(u, v))_{u, v \in V_N}$$
defined by  
\begin{equation}\label{TranProbMatr}
P(u, v)=
\left\{
\begin{array}{ll}
\dfrac{p(u, v)}{Z_{N}}, &  u\neq v, \\
0, & u=v,
\end{array}
\right.
\end{equation}
with
\begin{equation}\label{ZN}
Z_N =
\sum \limits_{ v \in V_N : v \neq u }
p(u, v) =1-\frac{2c}{N}+O\left(\frac{1}{N}\right),
\end{equation}
where the last equality is by (\ref{Ecr}).

In a similar way let us also define for any subset $A\subset V_N$ a 
Markov chain 
${\bf X}_A=\{X_{A,k}\}_{k = 1}^{\infty}$  on 
the state space $V_N\setminus A$ with  transition probability matrix defined as follows
$$P_A: = (P_A(u, v))_{u, v \in V_N\setminus A}$$
defined by  
\begin{equation}\label{XA}
P_A(u, v)=
\left\{
\begin{array}{ll}
 \dfrac{p(u, v)}{Z_{N,A}}, &  u\neq v, \\
  0, & u=v,
 \end{array}
 \right.
\end{equation}
with
\begin{equation}\label{ZNA}
Z_{N,A} =
\sum \limits_{ v \in V_N \setminus A: \ u\neq v }
p(u, v).
\end{equation}

We shall prove that the limiting distribution of (finite state) Markov chain $\{X_k\}$ is uniform on $V_N$. Let  $$\pi=\left( \pi(v)= \dfrac{1}{N^2}, v\in V_N\right)$$ denote this distribution.
 
 Recall some definitions.
\begin{Def}
	The total variation distance between two probability distributions $\mu$ and $\nu$ on a finite space $\mathcal{X}$ is defined by
	\begin{equation*}
		\| \mu - \nu \|_{TV} = \max\limits_{A \subseteq \mathcal{X}} |\mu(A) - \nu(A)|.
	\end{equation*}
      \end{Def}

      \begin{Prop}[\cite{LevinPeres} Proposition 4.2]
	Let $\mu$ and $\nu$ be two probability distributions on a finite space $\mathcal{X}$. Then
	\begin{equation*}
		\| \mu - \nu \|_{TV} = \dfrac{1}{2} \sum\limits_{x \in \mathcal{X}} |\mu(x) - \nu(x)|.
	\end{equation*}
      \end{Prop}
  
The last result immediately gives us a useful bound 
	\begin{equation*}\label{Cor_estimation_l_norm_and_TV}
		\|\mu - \nu\|_{l_{\infty}} = \max\limits_{x \in \mathcal{X}} |\mu(x) - \nu(x)| \leq \|\mu - \nu\|_{l_1} = 2 \| \mu - \nu \|_{TV}.
	\end{equation*}

        Now we can derive the stationary distribution of the introduced above Markov chain $X$.
        The following Lemma repeats arguments of \cite{LevinPeres}, but we provide the lines of the proof adopted for our case. 
        
\begin{Lemma}[\cite{LevinPeres} Theorem 4.9] \label{LmLevin_Peres}
  The stationary distribution of $\{X_k\}_{k = 1}^{\infty}$ for any initial state $X_1=a\in V_N$ is uniform on $V_N$. Furthermore, for all $k>1$
	\begin{equation*}\label{TV_bound}
		\max\limits_{u \in V_N} \| P^k(u, \cdot) - \pi ( \cdot)\|_{TV} \leq \left ( 1 - c_1 \right )^{k/2 - 1}
              \end{equation*}
              uniformly in $u\in V_N,$
              where $$c_1=\frac{c}{2}= \frac{c^{cr}}{2} = \frac{1}{8\log2}<1.$$
            \end{Lemma}

\begin{proof}
	The chain ${\bf X}$ is irreducible, positive recurrent and aperiodic, hence there exists a unique stationary distribution ${\widetilde {\pi}}$ to which  $X_k$ converges irrespectively to initial distribution. The stationary distribution ${\widetilde {\pi}}$ satisfies 
	\begin{equation*}
		{\widetilde {\pi}}P = {\widetilde {\pi}}.
	\end{equation*}
        Since $P$ is a symmetric stochastic matrix, ${\widetilde {\pi}}(u) = \dfrac{1}{N^2}= \pi(u) $ for any $u \in V_N$.

        By definition (\ref{TranProbMatr}) it follows that 
	\begin{equation}\label{t1}
          P^2(u, v) = \sum\limits_{x \in V_N} P(u, x) P(x, v)
=\dfrac{1}{Z_N^2}
\sum\limits_{\begin{smallmatrix} x \in V_N \\ x \neq u, v \end{smallmatrix}}
p(u, x) p(x, v).
\end{equation}
Taking into account (\ref{ZN}) 
we get from (\ref{t1})
	\begin{equation}\label{t2}
    P^2(u, v)  \geq \dfrac{1}{Z_N^2}\frac{c}{N^2}
  \sum\limits_{\begin{smallmatrix} x \in V_N \\ x \neq u, v \end{smallmatrix}}
  p(x, v)= \frac{c}{N^2}\left(1+O \left(\frac{1}{N}\right)\right)   \geq   
\frac{c}{2 }\pi(v)
\end{equation}
for all large $N$.
        This allows us to use the following construction from \cite{LevinPeres}. 
	Let $\Pi$ be the $N^2\times N^2$ matrix with each row equal  $\pi$, and set
        \[c_1=\frac{c}{2 }\]
Then the equation 
	\begin{equation*}\label{P2k_basis_induction}
		P^2 = c_1\Pi + (1 - c_1) Q,
              \end{equation*}
              which holds for by (\ref{t2})
	defines a stochastic matrix Q. One can prove by induction (consult \cite{LevinPeres})
        that 
	\[
		P^{2 k} = (1 - (1 - c_1)^k) \Pi + (1 - c_1)^k Q^k,
              \]
              which is equivalent to
              	\begin{equation*}\label{P2k_step_induction}
		P^{2 k} -  \Pi =  (1 - c_1)^k (Q^k - \Pi).
              \end{equation*}
This together with an observation that for a stochastic matrix $P$ it holds that $\Pi P=\Pi$, implies for any $j \geq 1$
	\begin{equation*}\label{P2kj}
		P^{2 k + j} - \Pi = (1 - c_1)^k (Q^k P^j - \Pi).
              \end{equation*}
              Hence,
              \[	\| P^{2 k + j}(u, \cdot) - \pi(\cdot) \|_{TV} \leq (1 - c_1)^k
\| Q^k P^j(u, \cdot) - \pi(\cdot) \|_{TV} \leq (1 - c_1)^k,
\]
where the last inequality is due to the assumption that both $P$ and $Q$ are stochastic matrices. This finishes the proof.
\end{proof}

The last lemma yields as well similar result 
when set $V_N$ is replaced 
by a smaller one.

\begin{Cor}\label{CTA}
	Let $A\subset V_N$ be that $$|A| =O\left( |V_N|^{2/3} \right).$$ 
	Then
	 the stationary distribution of ${\bf X}_A$ for any initial state $X_1=a\in V_N\setminus A$ is asymptotically  uniform on $V_N\setminus A$. Furthermore, for all $k>1$
	\begin{equation*}\label{TV_bound2}
	\max\limits_{u \in V_N\setminus A} \| P_A^k(u, \cdot) - \pi ( \cdot)\|_{TV} \leq \left ( 1 - c_1 \right )^{k/2 - 1}
	\end{equation*}
	uniformly in $u\in V_N\setminus A.$
\end{Cor}
\begin{proof}
		First we observe that by Lemma \ref{Lm_expectation_n_vertices} and bound
		(\ref{Pr1a1}) from Proposition \ref{Prop1}
		we have
		\begin{equation*}\label{TM31}
	Z_{N,A}	= \sum\limits_{u \in V_N \setminus A} p(u, v)\leq 1+O\left(\frac{1}{N}\right)
		-O\left({\frac{\sqrt{|A|}}{N}}\right)= 1 + O \left ( 1 / N^{1/3} \right ),
		\end{equation*}
		which also  yields
		\begin{equation}\label{PropEstsum2}
		\sum\limits_{u \in V_N \setminus A} P_{A}(u, v) = 1 + O \left ( 1 / N^{1/3} \right ).
		\end{equation}
		Hence, 
		\begin{equation}\label{TM32}
                  P_{A}^2(u, v) = \sum\limits_{x \in V_N \setminus A} P_{A}(u, x) P_{A}(x, v) \geq \frac{1}{Z_{N,A}^2}
                  \dfrac{c}{N^2} \left  ( 1 + o(1) \right ) \geq  \dfrac{c}{2} \pi(v)
		\end{equation}
                for all large $N$.
		
	Having (\ref{TM32}) allows us to repeat the proof of Lemma \ref{LmLevin_Peres} and to derive the statement of the Corollary.
\end{proof}

Let us also establish here a useful property of matrix $P$.
\begin{Prop}\label{Lm_ineq_Markov_chain}
	For any $u, v \in V_N$ and for any integer $k \geq 2$
	\begin{equation}\label{Lm_ineq_Markov_chain_ineq}
	P^{k}(u, v) \leq P^{2} \left (u, u \right ) = \dfrac{4 c^2 \log{N}}{N^2} + O(1 / N^2).
	\end{equation}
      \end{Prop}

\begin{proof}
	We prove this inequality by induction. For $k = 2$
	\begin{eqnarray}
	& &P^{2} \left (u, u \right ) - P^{2} \left (u, v \right ) = \sum\limits_{w \in V_N} P(u, w) P(w, u) - \sum\limits_{w \in V_N} P(u, w) P(w, v)  \nonumber \\ \nonumber
	&&= \dfrac{1}{2} \left (2 \sum\limits_{w \in V_N} (P(u, w))^2 - 2 \sum\limits_{w \in V_N} P(u, w) P(w, v) \right )  \\ \nonumber
	&&= \dfrac{1}{2} \left (\sum\limits_{w \in V_N} (P(u, w))^2 + \sum\limits_{w \in V_N} (P(w, v))^2 - 2 \sum\limits_{w \in V_N} P(u, w) P(w, v) \right )  \\ \nonumber
	&&= \dfrac{1}{2} \left (\sum\limits_{w \in V_N} \left ( P(u, w) - P(w, v) \right )^2 \right ) \geq 0 .
	\end{eqnarray}
	If the inequality in $(\ref{Lm_ineq_Markov_chain_ineq})$ holds for $k$, then for $k+1$
	\begin{eqnarray}
	&&P^{k+1}(u, v) = \sum\limits_{w \in V_N} P^k (u, w) P(w, v) \nonumber \\ \nonumber
	&&\leq P^2 (u, u) \sum\limits_{w \in V_N}  P(w, v)= P^2 (u, u) .
	\end{eqnarray}
	Thus the inequality in $(\ref{Lm_ineq_Markov_chain_ineq})$ has been proved. 
	
	To prove the last equality  in $(\ref{Lm_ineq_Markov_chain_ineq})$ consider
	\begin{eqnarray*}\label{TJ57}
	P^{2} \left (u, u \right ) && = \sum\limits_{v \in V_N} P(u, v)P(v,u) = \dfrac{1}{Z_N^2} \sum\limits_{\begin{smallmatrix} v \in V_N \\ v\neq u \end{smallmatrix}} (p(u, v))^2  \\ \nonumber
                                   && = \dfrac{1}{Z_N^2} \sum\limits_{r = 1}^{N} p_r^2N_r
                                      = \dfrac{4 c^2 \log{N}}{N^2} + O(1 / N^2).
	\end{eqnarray*}
which is due to (\ref{ZN}) and Proposition \ref{Prop1}. This ends the proof. 
\end{proof}

Proposition \ref{Lm_ineq_Markov_chain} together with (\ref{t1}) gives us another useful bound:
      	\begin{equation}\label{tt1}
\sum\limits_{\begin{smallmatrix} x \in V_N \\ x \neq u, v \end{smallmatrix}}
p(u, x) p(x, v) = Z_N^2  P^2(u, v)  < 5c^2\frac{\log N}{N^2}
\end{equation}
for all large $N$.

\subsection{Random walk on tree.}\label{SectionRWT}

Although the introduced Markov chain ${\bf X}$ clearly resembles construction of connected components in $T_k$, the major difference is that trajectories of ${\bf X}$ on $V_N$ may have self-intersections. 
Therefore we shall modify ${\bf X}$ into another 
random walk on $V_N$ to approximated closely components of $T_k$.

Given a sequence of increasing subsets ${\bf A}=\{A_i\}_{i\geq 1}$ in $V_N$:
\[A_1 \subseteq  A_2 \subseteq \ldots A_k \subseteq \ldots\subset V_N,  \]
define with a help of matrix $P_A$ introduced above
 a random walk without cycles and which avoids consecutively elements of ${\bf A}$ as follows.

\begin{Def}\label{Xtil}
  For a sequence ${\bf A}$ let ${\bf\tilde{X}}_{\bf A}=\{ \tilde{X}_{{\bf A},i}\}_{i \geq 1}$ be a random walk (``with memory'') on $V_N$ with initial state $\tilde{X}_{{\bf A},1}=a_1$, whose transition probabilities are defined  conditionally on its own past trajectory
as follows
\[\mathbb{P}\{\tilde{X}_{{\bf A}, i+1}=a_{i+1} |\tilde{X}_{{\bf A},1}=a_1, \ldots \tilde{X}_{{\bf A},i}=a_i\}= P_{A_i\cup \{a_1,...,a_i\}}(a_i,a_{i+1}),\]
\[a_{i+1}\in V_N \setminus (A_i\cup \{a_1,...,a_i\}), \ \ \ i \geq 1.\]
  \end{Def}

This construction allows us to approximate 
distribution of trees 
in the representation (\ref{Tree}) for $T_k$.
Consider for $i<j<k$, $L>0$  an event that $v_i$ is directly connected  to $v_j$ in $T_k$, i.e., there is a path $(v_i\rightarrow v_j)_{T_k}$. Assuming further that $v_i$ is connected  by $L$ edges to $v_j$ for some $L\geq 1 $, and $I_{i-1}=A_1\subset V_N$ we get for any $b\not \in A_{L}$ and $a \in A_1$
\[\min_{{\bf A} : A_L \not\owns b}
\mathbb{P}\{\tilde{X}_{{\bf A}, L+1}=b \big|\tilde{X}_{{\bf A},1}=a\}\]
\begin{equation}\label{TJ1}
\leq
\mathbb{P}\{v_j=b \big| v_i=a, \ |v_i, v_j|_{T_k}=L, \  |I_k|=K\} 
\end{equation}

\[
\leq \max_{{\bf A} : A_L \not \owns b}
\mathbb{P}\{\tilde{X}_{{\bf A}, L+1}=b \big|\tilde{X}_{{\bf A},1}=a\},
\]
where the minimum and the maximum are taken over
all sequences $${\bf A}=(A_1, \ldots, A_{L+1})$$ such that
\begin{equation}\label{seqA1}
 A_1 \subseteq A_2\subseteq \ldots \subseteq A_{L+1}, \ \left| A_{L+1} \right|\leq K . 
\end{equation}

By its construction a trajectory $\{\tilde{X}_{{\bf A},i}, \ i=1, \ldots ,k\}$ passes through different vertices in $V_N$, and moreover avoids consecutively sets $A_i, i\leq k.$ However, for sufficiently small $k$ it  is still
close in distribution to ${\bf {X}}$ as we show now.

\begin{Lemma}\label{LMarkov1}
	Let 
	\begin{equation*}\label{C0}
	C_0=-\dfrac{6}{ \log \left(1 - \frac{c_{cr}}{2}\right)}.
	\end{equation*}
	 For any constants $C, C_1$ such that 
	\begin{equation*}\label{CC}
	C_0<C<C_1,
	\end{equation*}
	 all
\begin{equation}\label{klog}	
	C \log N \leq k \leq C_1 \log N , 
	\end{equation}
        and any sequence ${\bf A}$ such that
\begin{equation}\label{seqA}
{\bf A}=(A_1, \ldots, A_k): A_1=A \subseteq A_2\subseteq \ldots \subseteq A_k
	\end{equation}
        with 
	 \begin{equation}\label{K}
\left| A_k \right|	  =O\left(
	 |V_N|^{2/3}
	 \right),	 
	 \end{equation}
	  one has
	\begin{equation*}\label{TJ4}
	\left | \mathbb{P}(\tilde{X}_{{\bf A},k+1} = b \big | \tilde{X}_{{\bf A},1} = a) - \dfrac{1}{N^2} \right | = O\left ( \dfrac{\log ^2 N}{N^{7/3}} \right )  
	\end{equation*}
	for all $a\in V_N$ and $ b \in V_N \setminus A_k \setminus \{a\}$.
\end{Lemma}
\begin{proof}
	
	For any $A\subseteq V_N$ consider the $|V_N|\times|V_N|$ matrix $P_A$ defined in (\ref{XA}). Note that matrix $P_A$ is obtained from $P$ by replacing all
	the elements in the rows and columns corresponding to the vertices $A$ by zero, and then normalizing each row. 
	Hence, for all $v\not \in A$
	\begin{equation}\label{TJ2}
	P(u,v)\leq P_A(u,v)\leq P(u,v)\frac{1}{1-
		\sum\limits_{v\in A}p(u,v)
	}
	,
	\end{equation}
	which by (\ref{Pr1a1}) yields
	\begin{equation}\label{TJ3}
	P(u,v)\leq P_A(u,v) \leq P(u,v)\left(1+O \left({\frac{\sqrt{A}}{N}}\right)\right).
	\end{equation}
	
	By the definition (\ref{Xtil}) we have
	\begin{eqnarray}
	&& \mathbb{P} \left\{ \tilde{X}_{{\bf A}, k+1} = b \big | \tilde{X}_{{\bf A},1} = a \right\} \nonumber \\ \nonumber
	&& = \sum\limits_{\small
		\begin{smallmatrix} 
		u_2, \ldots, u_{k} \in V_N:\\
	 u_{i+1} \not \in A_i\cup\{u_1=a, u_2, \ldots, u_{i}\}
 \end{smallmatrix} 
} P_{A_1\cup\{a\}}(a, u_2) P_{A_2\cup\{a, u_2\}}(u_2, u_3) \ldots   P_{A_{k}\cup\{a, u_2, \ldots, u_{k}\}}(u_{k}, b),
	\end{eqnarray}
	which together with bounds (\ref{TJ3})
	yields
	\begin{eqnarray}\label{vk_v1}
	&& \mathbb{P} \left\{\tilde{X}_{{\bf A}, k+1} = b \big | \tilde{X}_{{\bf A},1} = a\right\} 
	 \\ \nonumber
	&&  =\sum\limits_{\small
		\begin{smallmatrix} 
		u_2, \ldots, u_{k} \in V_N:\\
		u_{i+1} \not \in A_i\cup\{u_1 = a, u_2, \ldots, u_{i}\}
		\end{smallmatrix} 
		} P(a, u_2) P(u_2, u_3) \ldots  P(u_{k}, b)  
	\\ \nonumber
	&& \times \prod_{i=1}^k
	\left(1+O \left({\frac{\sqrt{|A_i|+i}}{N}}\right)\right) ,
	\end{eqnarray}
	Under assumptions (\ref{K}) and (\ref{klog}) here we have 
	\begin{equation*}
	\prod_{i=1}^k
	\left(1+O \left({\frac{\sqrt{|A_i|+i}}{N}}\right)\right)=1+O \left(\frac{\log N}{N^{1/3}}\right)
	\end{equation*}
	uniformly for all sequences ${\bf A}$ satisfying conditions of the lemma. 
	Substituting this into (\ref{vk_v1}) we derive
	\begin{equation}\label{TJ6}
	\mathbb{P} (\tilde{X}_{{\bf A}, k+1} = b \big | \tilde{X}_{{\bf A},1} = a) 
	= \left(1+O \left(\frac{\log N}{N^{1/3}}\right)\right)
	\end{equation}
	\[
	\times \left(P^{k}(a, b)
	-\sum\limits_{\small
		\begin{smallmatrix} 
	\exists	2\leq i \leq k	:\\
	u_i \in A_{i-1}\cup\{u_1=a, u_2, \ldots, u_{i-1}\}
		\end{smallmatrix} 
	} P(a, u_2) P(u_2, u_3)   \ldots  P(u_{k}, b)  
	\right).
	\]
	Making use of  Proposition \ref{Lm_ineq_Markov_chain} and bounds (\ref{Pr1a1}) first we get
	\begin{equation}\label{TJ7}
	\sum\limits_{\small
		\begin{smallmatrix} 
		\exists	2\leq i \leq k	:\\
		u_i \in A_{i-1}\cup\{u_1=a, u_2, \ldots, u_{i-1}\}
		\end{smallmatrix} 
	} 
P(a, u_2) P(u_2, u_3)   \ldots  P(u_{k}, b)  
	\end{equation}
	\[
	\leq 
	\sum_{i=2}^{k} \sum_{u_i \in A_{i-1}\cup\{u_1=a, u_2, \ldots, u_{i-1}\}}
	P(a, u_2) P(u_2, u_3)   \ldots  P(u_{k}, b)  
      \]
\[
	\leq \max_{u\in V_N}P^2(a, u) 
	\sum_{i=3}^{k}
        \sum_{u_i \in A_{i-1}\cup\{u_1=a, u_2, \ldots, u_{i-1}\}}
	P(u_3, u_4)   \ldots  P(u_{k}, b)  
      \]
      
	\[\leq O\left( \dfrac{ \log N}{N^2}\right) \sum_{i=3}^{k} \frac{\sqrt{|A_i|}}{N}
	=O\left( \dfrac{ \log N}{N^2} \frac{\log N}{N^{1/3}}\right). \]
	
	Hence, by (\ref{TJ7}) and (\ref{TJ6})
	\begin{equation*}\label{TJ8}
	\mathbb{P} (\tilde{X}_{{\bf A}, k+1} = b \big | \tilde{X}_{{\bf A},1} = a) 
	-P^{k}(a, b)=O\left( \dfrac{ (\log N)^2}{N^{7/3}}\right).
	\end{equation*}
Then taking also into account Lemma \ref{LmLevin_Peres} with $k\geq C_0 \log N$, we derive
	\begin{eqnarray}
	&& \left | \mathbb{P}(\tilde{X}_{{\bf A}, k+1} = b \big | \tilde{X}_{{\bf A},1} = a) - \dfrac{1}{N^2} \right | \nonumber \\ \nonumber
	&& \leq  \left | \mathbb{P}(\tilde{X}_{{\bf A},k+1} = b \big | \tilde{X}_{{\bf A},1} = a) - P^{k}(a, b) \right | + \left | P^{k}(a, b) - \dfrac{1}{N^2} \right | \\ \nonumber
	&& \leq  O\left ( \dfrac{\log^2 N}{N^{7/3}} \right )  +O\left ( \dfrac{1}{N^{3}} \right ),
	\end{eqnarray}
	which yields the statement of the Lemma.
	\end{proof}

        We shall now extend the statement of the last lemma for larger values of $k$. First we
        establish convergence.

\begin{Lemma}\label{LMarkov2}
		 For all $T>0$ all
	\begin{equation*}\label{k}	
	C_0 \log N <j\leq k \leq T|V_N|^{2/3} , 
	\end{equation*}
	and any sequence ${\bf A}$ satisfying (\ref{seqA}) and (\ref{K}),
	one has 
	\begin{equation*}\label{TJ9}
	 \mathbb{P}\left\{\tilde{X}_{{\bf A},j+1} = b \big | \tilde{X}_{{\bf A},1} = a\right\} - 
	 \mathbb{P}\left\{
	 \tilde{X}_{{\bf A},j+1} = b \big | \tilde{X}_{{\bf A},1} = a'\right\}
	 = O\left ( \dfrac{\log ^2 N}{N^{7/3}} \right )  
	\end{equation*}
	for all $a,a',b\in V_N$.
      \end{Lemma}
      Notice that constant $C_0$ here as defined in Lemma \ref{LMarkov1}.
\begin{proof}
	Making use of the definition (\ref{Xtil}) write
		\begin{eqnarray}\label{TJ10}
	&&\mathbb{P} \left\{ \tilde{X}_{{\bf A}, j+1} = b \big | \tilde{X}_{{\bf A},1} = a \right\} 
	-\mathbb{P} \left\{ \tilde{X}_{{\bf A}, j+1} = b \big | \tilde{X}_{{\bf A},1} = a' \right\} 
	\\ \nonumber
	&& = \sum\limits_{
		\begin{smallmatrix} 
		u_2, \ldots, u_{j} \in V_N:\\
		u_i \not \in A_{i-1}\cup\{a,a',b, u_2, \ldots, u_{i-1}\}
		\end{smallmatrix} 
	} 
\left(
P_{A_1\cup\{a\}}(a, u_2) 
P_{A_2\cup\{a, u_2\}}(u_2, u_3) 
\ldots   \right.
\\ \nonumber
&&
\hspace{4cm}\times  
P_{A_j\cup\{a, u_2, \ldots, u_{j}\}}(u_{j}, b)\\ \nonumber
&& 
\left.
-
P_{A_1\cup\{a'\}}(a', u_2) P_{A_2\cup\{a', u_2\}}(u_2, u_3) \ldots P_{A_j\cup\{a', u_2, \ldots, u_{j}\}}(u_{j}, b)\right)\\  \nonumber \\  \nonumber
&& 
+\sum_{l=2}^{j} \left(
\sum\limits_{
	\begin{smallmatrix} 
	u_2, \ldots, u_{j} \in V_N:
	u_l=a'
	\\ \nonumber
	u_i \not \in A_{i-1}\cup\{a,a',b, u_2, \ldots, u_{i-1}\}
	\end{smallmatrix} 
} 
P_{A_1\cup\{a\}}(a, u_2) 
P_{A_2\cup\{a, u_2\}}(u_2, u_3) 
\ldots   \right.
\\ \nonumber
&&
\hspace{4cm}\times  
P_{A_j\cup\{a, u_2, \ldots, u_{j}\}}(u_{j}, b)\\ \nonumber
&& 
\left.
-\sum\limits_{
	\begin{smallmatrix} 
	u_2, \ldots, u_{j} \in V_N:
	u_l=a
	\\ \nonumber
	u_i \not \in A_{i-1}\cup\{a,a',b, u_2, \ldots, u_{i-1}\}
	\end{smallmatrix} 
} 
P_{A_1\cup\{a'\}}(a', u_2)  \ldots P_{A_j\cup\{a', u_2, \ldots, u_{j}\}}(u_{j}, b)\right)\\
\nonumber\\
\nonumber
&&=:\Sigma^{'} + \Sigma^{''}.
	\end{eqnarray}

	Consider the first difference, i.e., $\Sigma'$ in (\ref{TJ10}). 
	Let 
	\begin{equation*}\label{u}
	\bar{u} = \{u_2, \ldots, u_{j}\} \in V_N \setminus \{a, a', b\}
	\end{equation*}
	 be a set of distinct vertices. 
	Then for any pair $u_i,u_{i+1}$ 
	from this set
	and any $B \subset V_N \setminus \{a, a', b\}$  we have by definition (\ref{XA})
	\begin{eqnarray}
	&& P_{B\cup a'}(u_i, u_{i+1}) = \dfrac{p(u_i, u_{i+1})}{\sum\limits_{v \in V_N \setminus B \setminus \{a'\}} p(u_i, v)} \nonumber \\ \nonumber
	&& = \dfrac{p(u_i, u_{i+1})}{\sum\limits_{v \in V_N \setminus B \setminus \{a\}} p(u_i, v)} \left ( 1 - \dfrac{p(u_i, a) - p(u_i, a')}{\sum\limits_{v \in V_N \setminus B \setminus \{a'\}} p(u_i, v)} \right )  \\ \nonumber \\ \nonumber
	&& = P_{B\cup a}(u_i, u_{i+1})   
	\left ( 1 - \dfrac{p(u_i, a) - p(u_i, a')}{\sum\limits_{v \in V_N \setminus B \setminus \{a'\}} p(u_i, v)} \right ).\\ \nonumber 
	\end{eqnarray}
	Therefore we can rewrite the terms in $\Sigma'$ in (\ref{TJ10}) as follows
	\begin{equation}\label{TJ12}
	P_{A_1\cup\{a'\}}(a', u_2) P_{A_2\cup\{a', u_2\}}(u_2, u_3) \ldots P_{A_j\cup\{a', u_2, \ldots, u_{j}\}}(u_{j}, b)
	\end{equation}
	\[=P_{A_1\cup\{a\}}(a', u_2) P_{A_2\cup\{a, u_2\}}(u_2, u_3) \ldots P_{A_j\cup\{a, u_2, \ldots, u_{j}\}}(u_{j}, b)\]
	\[\prod_{i=2}^{j}\left( 1 - \dfrac{p(u_i, a) - p(u_i, a')}{\sum\limits_{v \in V_N \setminus 
			\{ 
			A_{i}\cup\{a', u_2, \ldots, u_{i}
			\}\}
		} p(u_i, v)} \right). \]
	Note that for any $B$ with $|B|=o(|V_N|)$ by (\ref{Pr1a1}) and (\ref{ZNA}) we have
	\begin{equation}\label{TJ11}
	\prod_{i=2}^{j}\left ( 1 - \dfrac{p(u_i, a) - p(u_i, a')}{\sum\limits_{v \in V_N \setminus B } p(u_i, v)} \right)
	\end{equation}
	\[=1+O\left( \sum_{i=2}^{j} (p(u_i, a) +p(u_i, a'))\right)=1+ O\left(\sqrt{\frac{j}{|V_N|}}\right)\]
	uniformly in $B$ and $u_2, \ldots, u_{j-1}.$ Hence, we derive from 
	 (\ref{TJ12}) with a help of  
	 (\ref{TJ11}) that
	 \begin{equation*}\label{TJ13}
	 P_{A_1\cup\{a'\}}(a', u_2) P_{A_2\cup\{a', u_2\}}(u_2, u_3) \ldots P_{A_j\cup\{a', u_2, \ldots, u_{j}\}}(u_{j}, b)
	 \end{equation*}
	 \[=P_{A_1\cup\{a\}}(a', u_2) P_{A_2\cup\{a, u_2\}}(u_2, u_3) \ldots P_{A_j\cup\{a, u_2, \ldots, u_{j}\}}(u_{j}, b)
	 \left( 
	 1+ O\left({\frac{\sqrt{j}}{N}}\right)
	 \right).
	 \]
	 Now we can rewrite the first sum in (\ref{TJ10}) as follows
\begin{equation}\label{TJ14}
\Sigma'= \sum\limits_{
		\begin{smallmatrix} 
		u_2, \ldots, u_{j} \in V_N:\\
		u_i \not \in A_{i-1}\cup\{a,a',b, u_2, \ldots, u_{i-1}\}
		\end{smallmatrix} 
	} 
\left(
P_{A_1\cup\{a\}}(a, u_2) 
-  P_{A_1\cup\{a'\}}(a', u_2)  \right)
\end{equation}
\[
\hspace{4cm}\times  P_{A_2\cup\{a, u_2\}}(u_2, u_3) \ldots 
P_{A_j\cup\{a, u_2, \ldots, u_{j}\}}(u_{j}, b)\]

\[
 + O\left({\frac{\sqrt{j}}{N}}\right)  \sum\limits_{
		\begin{smallmatrix} 
		u_2, \ldots, u_{j} \in V_N:\\
		u_i \not \in A_{i-1}\cup\{a,a',b, u_2, \ldots, u_{i-1}\}
		\end{smallmatrix} 
	} 
P_{A_1\cup\{a \}}(a', u_2) \ldots P_{A_j\cup\{a , u_2, \ldots, u_{j}\}}(u_{j}, b).  \]

Taking into account that
\[p(u, v) \geq \dfrac{c}{N^2}=\dfrac{c}{|V_N|},\]
we have
for any $B \subset V_N$ with $ |B| =O\left(|V_N|^{2/3}\right)$ and any $u, v \in V_N \setminus B$ the following decomposition
	\begin{eqnarray}\label{TJ16}
	&& P_B(u, v) = \dfrac{c}{|V_N|} + \tilde{P}_B (u, v), 
	\end{eqnarray}
	where $\tilde{P}_B(u, v) \geq 0$. Since $\sum\limits_{v \in V_N \setminus B} P_B(u, v) = 1$, this together with (\ref{ZNA})  implies
	\begin{equation} \label{sumPtilde}
          \sum\limits_{v \in V_N \setminus B} \tilde{P}_B(u, v) = 1 - c + O\left(
            |V_N|^{-1/3 }\right) < 1 - \frac{c}{2}<1
	\end{equation}
        for all large $N.$

        Consider now  the first sum in (\ref{TJ14})     on the right. 
        Note, that if all the involved probabilities would be uniform (as, e.g., term $\dfrac{c}{|V_N|}$ in  (\ref{TJ16})) the difference in the first sum in
        (\ref{TJ14}) would be simply zero. Here with a help of (\ref{TJ16})
we derive from (\ref{TJ14}):
        \begin{equation}\label{TJ15}          
\Sigma'= \sum\limits_{
		\begin{smallmatrix} 
		u_2, \ldots, u_{j} \in V_N:\\
		u_i \not \in A_{i-1}\cup\{a,a',b, u_2, \ldots, u_{i-1}\}
		\end{smallmatrix} 
	} 
\left(
P_{A_1\cup\{a\}}(a, u_2) 
-  P_{A_1\cup\{a'\}}(a', u_2)  \right)
\end{equation}
\[
\times  P_{A_2\cup\{a, u_2\}}(u_2, u_3) \ldots 
P_{A_{j-1}\cup\{a, u_2, \ldots, u_{j-1}\}}(u_{j-1}, u_j)
\left(\dfrac{c}{|V_N|} + \tilde{P}_B (u_{j}, b) \right) 
\]

\[
  + O\left({\frac{\sqrt{j}}{N}}\right)
  O\left(\max_{u \in V_N} \sum_{u_2}
P(a', u_2) P(u_2, u) 
\right)O\left(\max_{u \in V_N} \sum_{u_j}
P(u, u_j) P(u_j, b) 
\right)
\]

\[\times  \sum\limits_{
		\begin{smallmatrix} 
		u_3, \ldots, u_{j-1} \in V_N:\\
		u_i \not \in A_{i-1}\cup\{a,a',b, u_2, \ldots, u_{i-1}\}
		\end{smallmatrix} 
              }
P_{A_3\cup\{a \}}(u_3, u_4) \ldots P_{A_{j-2}\cup\{a , u_2, \ldots, u_{j-2}\}}(u_{j-2},u_{j-1}  ).  \]

         Observe that making use of (\ref{sumPtilde})  we get 
 \begin{equation*}\label{TJ170}          
\sum\limits_{
		\begin{smallmatrix} 
		u_2, \ldots, u_{j} \in V_N:\\
		u_i \not \in A_{i-1}\cup\{a,a',b, u_2, \ldots, u_{i-1}\}
		\end{smallmatrix} 
	} 
\left(
P_{A_1\cup\{a\}}(a, u_2) 
-  P_{A_1\cup\{a'\}}(a', u_2)  \right)
\end{equation*}
\[
\times  P_{A_2\cup\{a, u_2\}}(u_2, u_3) \ldots 
P_{A_{j-1}\cup\{a, u_2, \ldots, u_{j-1}\}}(u_{j-1}, u_j)\left(\dfrac{c}{|V_N|} + \tilde{P}_B (u_{j}, b) \right) 
\]
\[= \sum\limits_{
		\begin{smallmatrix} 
		u_2\not \in A_{1}\cup\{a,a',b\}
		\end{smallmatrix} 
	} \left(
P_{A_1\cup\{a\}}(a, u_2) 
-  P_{A_1\cup\{a'\}}(a', u_2)  \right) 
        \left(1-O(1/N)\right)^{j-2}
        \dfrac{c}{|V_N|} \]
      \[+ \sum\limits_{
		\begin{smallmatrix} 
		u_2, \ldots, u_{j} \in V_N:\\
		u_i \not \in A_{i-1}\cup\{a,a',b, u_2, \ldots, u_{i-1}\}
		\end{smallmatrix} 
	} \sum_{i=0}^{j-3}
\left(
P_{A_1\cup\{a\}}(a, u_2) 
-  P_{A_1\cup\{a'\}}(a', u_2)  \right) \left(1-O(1/N)\right)^{j-i-3}
\]
\[\times\dfrac{c}{|V_N|}
\tilde{P}_B (u_{j-i}, u_{j-i+1})
\ldots  \tilde{P}_B (u_{j}, b) \]

\[+ \sum\limits_{
	\begin{smallmatrix} 
		u_2, \ldots, u_{j} \in V_N:\\
		u_i \not \in A_{i-1}\cup\{a,a',b, u_2, \ldots, u_{i-1}\}
	\end{smallmatrix} 
} \left(
P_{A_1\cup\{a\}}(a, u_2) 
-  P_{A_1\cup\{a'\}}(a', u_2)  \right) \tilde{P}_B (u_{2}, u_{3})
\ldots  \tilde{P}_B (u_{j}, b)
\]

\[= O\left(
\frac{1}{|V_N| N^{2/3}}
        \right)\]
      as long as $j=O(|V_N|^{2/3})$ (notice that $(1-O(1/N)) \leq 1$ in the expression above). Substituting now this bound and making use of Proposition \ref{Lm_ineq_Markov_chain} combined with  (\ref{TJ3}) we get from
(\ref{TJ15})         
         \begin{equation}\label{TJ17}          
\Sigma'= O\left(
\frac{1}{|V_N| N^{2/3}}
\right) + O\left({\frac{\sqrt{j}}{N}}\right)O\left(
\frac{(\log N)^2}{|V_N|^2}
\right)|V_N|= O\left(
\frac{(\log N)^2}{N^{7/3}}
\right).
\end{equation}

Consider now the last sum $\Sigma''$ in (\ref{TJ10}). We first derive a bound for the
following sum (taking into account  bounds (\ref{TJ3}) and Proposition
\ref{Lm_ineq_Markov_chain}):
            \begin{equation*}\label{TJ18}
\sum_{l=2}^j
              \sum\limits_{\small
    u_2, \ldots, u_{j} \in V_N:
	u_l=a' 
} 
P_{A_1\cup\{a\}}(a, u_2) 
P_{A_2\cup\{a, u_2\}}(u_2, u_3) 
\ldots   P_{A_j\cup\{a, u_2, \ldots, u_{j}\}}(u_{j}, b)
\end{equation*}

\[\leq 
  P_{A_1\cup\{a\}}(a, a')
    \max_{v}      \left(    \sum\limits_{\small
    u_{j} \in V_N
} 
P_{A_{j-1}\cup\{a, u_2, \ldots, u_{j-1}\}}(v,u_{j})
P_{A_j\cup\{a, u_2, \ldots, u_{j}\}}(u_{j}, b)\right)
\]
\[+
  \max_{v}    \left(       \sum\limits_{\small
    u_{2} \in V_N
} P_{A_{1}\cup\{a\}}(a,u_{2})
P_{A_{2}\cup\{a,u_2\}}(u_{2},v)\right)
\max_{u_2, \ldots, u_{j}}     
P_{A_j\cup\{a, u_2, \ldots, u_{j}\}}(a', b)
\]
\[+\sum_{l=3}^{j-1}
 \max_{v}    \left(       \sum\limits_{\small
    u_{2} \in V_N
} P_{A_{1}\cup\{a\}}(a,u_{2})
P_{A_{2}\cup\{a,u_2\}}(u_{2},v)\right)\]
\[\times
 \max_{v}      \left(    \sum\limits_{\small
    u_{j} \in V_N
} 
P_{A_{j-1}\cup\{a, u_2, \ldots, u_{j-1}\}}(v,u_{j})
P_{A_j\cup\{a, u_2, \ldots, u_{j}\}}(u_{j}, b)\right)
\]
\[= O\left(
\frac{\log N}{N^{3}}
\right) + O\left(k
\frac{(\log N)^2}{N^{4}}
\right)= O\left(
\frac{\log N}{N^{3}}
\right) .\]
The same bound holds for the remaining sum in $\Sigma''$ in (\ref{TJ10}),
which allows us to conclude that
\[\Sigma''= O\left(
\frac{\log N}{N^{3}}
\right). \]
The latter combined  with (\ref{TJ17}) and  (\ref{TJ10}) yields the statement of the Lemma.
\end{proof}

Now as a corollary of the last two lemmas we can extend the result of
Lemma \ref{LMarkov1} up to higher values of $k$, establishing therefore the 
 uniform  distribution approximation for $ \tilde{X}_{{\bf A},k} $.

\begin{Cor}\label{LMarkov3}
	Under assumptions of Lemma \ref{LMarkov2}, i.e.,
for all $T>0$ all
\[	
C_0 \log N <j\leq k \leq T|V_N|^{2/3} , 
\]
and any sequence ${\bf A}$ satisfying (\ref{seqA}) and (\ref{K}), 
one has 	
\begin{equation}\label{TJ19}
	 \mathbb{P}\left\{\tilde{X}_{{\bf A},j+1} = b \big | \tilde{X}_{{\bf A},1} = a\right\} =
	 \mathbb{P}\left\{
	 \tilde{X}_{{\bf A},j+1} = b  \right\}\left ( 1 + o(1) \right ) 
	 =\dfrac{1}{N^2} \left ( 1 + o(1) \right ) 
	\end{equation}
	for all $a\in V_N$, $b \in V_N \setminus A_{j}$ uniformly in ${\bf A}$.
	
	If, in addition, $\tilde{X}_{{\bf A}, 1}$ is uniformly distributed on $V_N \setminus A_1$, then for all 
	\begin{equation*}
		1 \leq j \leq k \leq T |V_N|^{2/3}, 
	\end{equation*}
	one has 	
	\begin{equation}\label{EqUniformX_tilda}
		\mathbb{P} \left \{ \tilde{X}_{{\bf A}, j+1} = b \right \} = \dfrac{1}{N^2} \left ( 1 + o (1) \right ) .
	\end{equation}
	for all $\; b \in V_N \setminus A_j$ uniformly in ${\bf A}$.
\end{Cor}
\begin{proof}
  Let $l = \left [ C_0 \log N \right ] + 2$.
By  Lemma \ref{LMarkov1} we have  for any $a \in V_N$, and $b \in V_N\setminus A_l \setminus \{a\}$
	\begin{equation*}
          \mathbb{P}
\left\{\tilde{X}_{{\bf A},l+1} = b \big | \tilde{X}_{{\bf A},1} = a\right\} 
       = \dfrac{1}{N^2} + o \left ( \dfrac{1}{N^2} \right ).
	\end{equation*}
	Therefore for any $l<j\leq k $ we derive
	\begin{eqnarray}
          &&
             \mathbb{P}
             \left\{\tilde{X}_{{\bf A},j+1} = b \big | \tilde{X}_{{\bf A},1} = a\right\} 
  \nonumber \\ \nonumber
          && = \sum\limits_{u \in V_N \setminus A_{l-1} \setminus \{a, b\}}
             \mathbb{P}
\left\{\tilde{X}_{{\bf A},j+1} = b \big |  \tilde{X}_{{\bf A},l} = u, \tilde{X}_{{\bf A},1} = a\right\} 
             \mathbb{P}\left\{  \tilde{X}_{{\bf A},l} = u \big |  \tilde{X}_{{\bf A},1} = a\right\} \\ \nonumber
          && = \sum\limits_{u \in V_N \setminus A_{l-1} \setminus \{a, b\}}
             \mathbb{P}
             \left\{\tilde{X}_{{\bf A},j+1} = b \big |  \tilde{X}_{{\bf A},l} = u, \tilde{X}_{{\bf A},1} = a\right\}
            \mathbb{P}\left\{\tilde{X}_{{\bf A},l} = u\right\}
             \left ( 1 + o(1) \right )
          \\ \nonumber
          &&
             = \sum\limits_{u \in V_N \setminus A_{l-1} \setminus \{a, b\}} \mathbb{P}
\left\{\tilde{X}_{{\bf A},j+1} = b \big |  \tilde{X}_{{\bf A},l} = u \right\}
             \mathbb{P}\left\{\tilde{X}_{{\bf A},l} = u\right\}
             \left ( 1 + o(1) \right ) \\ \nonumber
          && + \sum\limits_{u \in V_N \setminus A_{l-1} \setminus \{a, b\}} \left  (
\mathbb{P}
\left\{\tilde{X}_{{\bf A}\cup \{a\},j+1} = b \big |  \tilde{X}_{{\bf A}\cup \{a\},l} = u \right\}-
\mathbb{P}
             \left\{\tilde{X}_{{\bf A},j+1} = b \big |  \tilde{X}_{{\bf A},l} = u \right\} \right )
            \\ \nonumber
          && \times  \dfrac{1}{N^2} \left (1 + o(1) \right ).
	\end{eqnarray}
        Employing now Lemma \ref{LMarkov2}, we continue as follows
\begin{equation*}\label{TJ20}
  \mathbb{P}\left\{\tilde{X}_{{\bf A},j+1} = b \big | \tilde{X}_{{\bf A},1} =
    a\right\} = \mathbb{P}
\left\{\tilde{X}_{{\bf A},j+1} = b \right\}
\left ( 1 + o(1) \right )
\end{equation*}
\[
+ \sum\limits_{u \in V_N \setminus A \setminus \{a, b\}} O \left (\dfrac{\log N}{N^{7/3}} \right ) \dfrac{1}{N^2} \left ( 1 + o(1) \right ) = \mathbb{P}\left\{\tilde{X}_{{\bf A},j+1} = b \right\}
\left ( 1 + o(1) \right ) ,
\]
from which we conclude that $\mathbb{P}\left\{\tilde{X}_{{\bf A},j+1} = b \right \}$ is the stationary probability for $b$. Equation $(\ref{PropEstsum2})$ implies that the stationary distribution has the form 
\begin{equation*}
	\mathbb{P}\left\{\tilde{X}_{{\bf A},j+1} = b \right \} = \dfrac{1}{N^2} \left ( 1 + O \left ( 1 / N^{1/3} \right ) \right ) .
\end{equation*}
It confirms the statement $(\ref{TJ19})$ of the corollary.

 		If, in addition, $\tilde{X}_{{\bf A}, 1}$ is uniformly distributed on $V_N \setminus A_1$, we have for any $a_1 \in V_N \setminus A_1$
		\begin{equation*}
			\mathbb{P}(\tilde{X}_{{\bf A}, 1} = a_1) = \dfrac{1}{N^2 - |A_1|} = \dfrac{1}{N^2} \left ( 1 + \dfrac{|A_1|}{N^2 - |A_1|} \right ) = \dfrac{1}{N^2} + o \left ( \dfrac{1}{N^2} \right ).
		\end{equation*}		
		For $j >  C_0 \log N$, $(\ref{EqUniformX_tilda})$ follows from $(\ref{TJ19})$.	
		
		If $j \leq  C_0 \log N$, then taking into account $(\ref{PropEstsum2})$, we derive
		\begin{eqnarray}
			&& \mathbb{P}(\tilde{X}_{{\bf A}, j+1} = b) = \sum\limits_{a_1, \ldots, a_{j} }\mathbb{P}(\tilde{X}_{{\bf A}, 1} = a_1) P_{A_{1} \cup \{a_1\}}(a_{1}, a_{2}) \nonumber \\ \nonumber
			&& \cdot \ldots \cdot  P_{A_{j-1}\cup \{a_1,...,a_{j-1}\}}(a_{j-1}, a_{j}) \cdot   P_{A_{j}\cup \{a_1,...,a_{j}\}}(a_{j}, b)   \\ \nonumber
			&& = \dfrac{1}{N^2} \left ( 1 + o(1) \right ) \left ( 1 + O(1 / N^{1/3}) \right )^{j} = \dfrac{1}{N^2} + o \left ( \dfrac{1}{N^2} \right ) ,
		\end{eqnarray}
		which confirms the statement $(\ref{EqUniformX_tilda})$ of the corollary.
	\end{proof}

Finally we get one more useful bound.

\begin{Lemma}\label{LEp}
For all $T>0$, all
	\begin{equation*}\label{k2}	
1<j\leq k \leq T|V_N|^{2/3} , 
	\end{equation*}
	and any sequence ${\bf A}$ satisfying (\ref{seqA}) and (\ref{K}),
one has 	
\begin{equation}\label{TJ36}
  \mathbb{P}\left\{\tilde{X}_{{\bf A},j} = b \big | \tilde{X}_{{\bf A},1} = a\right\}
\leq \max \left ( \dfrac{4 c^2 \log N}{N^2}, p(a, b) \right ) \left ( 1 + o(1) \right ) .
	\end{equation}
	for all $a,b\in V_N$ uniformly in ${\bf A}$.
        
	\begin{proof}
		If $j = 2$ then
		\begin{eqnarray}
                  && \mathbb{P}\left\{\tilde{X}_{{\bf A},2} = b \big | \tilde{X}_{{\bf A},1} = a\right\} \nonumber \\ \nonumber
                   &&  = p(a, b) \left ( 1 + \sum\limits_{u \in A_2} \dfrac{p(a, u)}{1 - \sum\limits_{u \in A_2} p(a, u)} \right )  \\ \nonumber
		&&  = p(a, b) \left ( 1 + o(1) \right ) .
		\end{eqnarray}
		By (\ref{TJ2}) and by Proposition \ref{Lm_ineq_Markov_chain} for any $A_{i-1} \subset A_i \subset A_{i+1}$ with $u \in V_N \setminus A_{i-1} $
under assumptions of the lemma
                we have
		\begin{equation*}
               \sum\limits_{w \in V_N \setminus A_i \setminus \{u\}} P_{A_i \cup \{u\}}(u, w) P_{A_{i+1}\cup \{u, w\}} (w, b) \leq  \dfrac{4 c^2 \log N}{N^2} \left ( 1 + o(1) \right ) 
             \end{equation*}
             uniformly in $a,b$ and sequence ${\bf A}$. Hence, for all $j\geq 2$
             \[
               \mathbb{P} \left\{\tilde{X}_{{\bf A},j} = b \big | \tilde{X}_{{\bf A},1}
                 = a\right\}
\leq  \dfrac{4 c^2 \log N}{N^2} \]
	and the statement (\ref{TJ36}) follows. 
	\end{proof}
\end{Lemma}

 \subsection{Proof of Theorem \ref{ThIndep}.}\label{SectionProof29}
 First, we need the following auxiliary lemma. 
 \begin{Lemma}\label{Cor_b_not_in_I}
 	Let $k = |V_N|^{2/3} s, \; s \in [0, T]$. Then uniformly for all $a \neq b$ and $1 \leq i, j, l \leq k\leq T|V_N|^{2/3}$
 	\begin{equation*}
 		\mathbb{P} (b \notin I_{l} \big| v_i = a, |v_i,v_j|_{T_k} < \infty, |I_k| \leq |V_N|^{2/3} C ) = 1 + O \left ( \dfrac{1}{N^{1/3}} \right ),
 	\end{equation*}
 	\begin{equation*}
 		\mathbb{P} (b \notin I_{l} \big| v_i = a, |v_i,v_j|_{T_k} = \infty, |I_k| \leq |V_N|^{2/3} C ) = 1 + O \left ( \dfrac{1}{N^{1/3}} \right ),
 	\end{equation*}
 	for all large $N$.
 	
 	\begin{proof}
 		We will prove the first equality. The second can be proved in exactly the same way.
 		\begin{eqnarray}
 			&& \mathbb{P} (b \notin I_{l} \big| v_i = a, |v_i,v_j|_{T_k} < \infty, |I_k| \leq |V_N|^{2/3} C  ) \nonumber \\ \nonumber
 			&& \geq \mathbb{P} (b \notin I_k \big| v_i = a, |v_i,v_j|_{T_k} < \infty, |I_k| \leq |V_N|^{2/3} C  ) \\ \nonumber
 			&& = \mathbb{P} \left ( b \notin \mathcal{R}_k, b \notin \mathcal{N} (v_1), \ldots, b \notin \mathcal{N} (v_k) \big| v_i = a, |v_i,v_j|_{T_k} < \infty, |I_k| \leq |V_N|^{2/3} C  \right ) \\ \nonumber
 			&& \geq 1 - \sum\limits_{l = 1}^{k} \mathbb{P} (b \in \mathcal{N} (v_l) \big| v_i = a, |v_i,v_j|_{T_k} < \infty, |I_k| \leq |V_N|^{2/3} C ) \\ \nonumber
 			&& - \mathbb{P} (b \in \mathcal{R}_k \big| v_i = a, |v_i,v_j|_{T_k} < \infty, |I_k| \leq |V_N|^{2/3} C ) \\ \nonumber
 			&& \geq 1 - \max\limits_{A \subseteq V_N : |A| = k + |V_N|^{2/3} C} \sum\limits_{u \in A} p(u, b) - \dfrac{k}{N^2 - k} = 1 + O \left ( \dfrac{1}{N^{1/3}} \right ),
 		\end{eqnarray}
 		where the last equality holds by $(\ref{Pr1a1})$.
 		
 	\end{proof}
 \end{Lemma}

 \subsubsection{Proof of Lemma \ref{LT1}}
 \begin{proof}
 		Consider 
 	\begin{equation}\label{LT1eq1}
 		\mathbb{P} \left \{ v_j = b \big| v_i = a,    |v_i,v_j|_{T_k}=\infty ,
 		|I_k| \leq |V_N|^{2/3} C 
 		\right \}.
 	\end{equation}
 	Firstly, we need the following definition.
 	\begin{Def}\label{Def_parent}
 		We call $u$ a \textbf{parent} of $v$ if $\left\|(u \rightarrow v)_{T_k}\right\| = 1$. In what follows, we will always denote the parent of $v_j$ by $v_{\hat{j}}$. If $v_j$ is a root of a component then we define $\hat{j} : = j-1$.
 	\end{Def}
 	
 	If the probability in  $(\ref{LT1eq1})$ is further conditioned on the event $ \left \{b \in I_{\hat{j} - 1} \right \}$, it becomes $0$. As a result, making use of Lemma \ref{Cor_b_not_in_I}, we have
 	\begin{eqnarray}\label{LT1eq2}
 		&&\mathbb{P} \left \{ v_j = b \big| v_i = a, |v_i,v_j|_{T_k}=\infty ,
 		|I_k| \leq |V_N|^{2/3} C
 		\right \}  \\ \nonumber
 		&& = \mathbb{P} \left \{  v_j = b \big| v_i = a, |v_i,v_j|_{T_k}=\infty, |I_k| \leq |V_N|^{2/3} C, b \notin I_{\hat{j} - 1} \right \} \\ \nonumber
 		&& \times \mathbb{P} \left \{  b \notin I_{\hat{j} - 1} \big| v_i = a, |v_i,v_j|_{T_k}=\infty, |I_k| \leq |V_N|^{2/3} C \right \} \\ \nonumber
 		&& =  \mathbb{P} \left \{  v_j = b \big| v_i = a, |v_i,v_j|_{T_k}=\infty, |I_k| \leq |V_N|^{2/3} C, b \notin I_{\hat{j} - 1} \right \} \left ( 1 + O\left ( 1/N^{1/3} \right ) \right ) .
 	\end{eqnarray}
 	
 	Since $|v_i,v_j|_{T_k}=\infty$, there exists $v_l = R(v_j)$ (the root of the tree, containing $v_j$), for some $i < l \leq j$. Denote $$L : = \left\|(v_l \rightarrow v_j)_{T_k}\right\| \geq 0,$$ 
 	where $L=0$ means that $v_l = v_j$, and consider a process ${\bf\tilde{X}}_{\bf A}$ which was defined in Section \ref{SectionRWT} for any suitable
 	 ${\bf A}$
 	 with $A_1 = I_{l-1}$.  	
 	 	By the property (\ref{TJ1}) we can estimate
 	
 	\[\min_{{\bf A}}
 	\mathbb{P}\{\tilde{X}_{{\bf A}, L+1}=b \big| b \notin A_L\}\]
 	\begin{equation}\label{LT1eq3}
 		\leq
 		\mathbb{P} \left \{  v_j = b \big| v_i = a, |v_i,v_j|_{T_k}=\infty, |I_k| \leq |V_N|^{2/3} C, b \notin I_{\hat{j} - 1} \right \}
 	\end{equation}
 	\[
 	\leq \max_{{\bf A}}
 	\mathbb{P}\{\tilde{X}_{{\bf A}, L+1}=b \big| b \notin A_L \},
 	\]
 	where the minimum and the maximum are taken over
 	all sequences $${\bf A}=(A_1, \ldots, A_{L+1})$$ such 
 	that
 	\[
 	I_{l-1} = A_1 \subseteq A_2\subseteq \ldots \subseteq A_{L+1}\subset V_N, \ \left| A_{L+1} \right|\leq |V_N|^{2/3} C 
 	\]
 	(see (\ref{seqA1})).
 	
 	Since $v_l$ is uniformly distributed on $V_N \setminus I_{l-1}$ we can use $(\ref{EqUniformX_tilda})$ from Corollary \ref{LMarkov3} to get that both lower and upper bound in $(\ref{LT1eq3})$ have the same limit, resulting in 
 	
 	\begin{eqnarray}
 		&& \mathbb{P} \left \{  v_j = b \big| v_i = a, |v_i,v_j|_{T_k}=\infty, |I_k| \leq |V_N|^{2/3} C, b \notin I_{\hat{j} - 1} \right \} \nonumber \\ \nonumber
 		&& = \dfrac{1}{N^2} + o \left ( \dfrac{1}{N^2} \right ),
 	\end{eqnarray}
 	which together with $(\ref{LT1eq2})$ proves Lemma \ref{LT1}.
 \end{proof}
  
 \subsubsection{Proof of Lemma \ref{LT3}.}

  Fix $C_1>0$ and
  $L\leq C_1\log N$
   arbitrarily. 
   
Recall that by our construction vertices $v_i$ when belonging to  the same component, 
form consecutive generations. In other words, the tree-distance from $v_i$ to the root of component is non-decreasing in $i$. Hence, 
if  $v_i,v_j$ are in the same tree-component of $T_k$ with $i<j$, then
all  $j-i+1$ vertices $v_i, v_{i+1}, \ldots, v_j$, 
must belong to the generations which include vertex $v_i$, vertex $v_j$, and all the generations in  between. The number of these generations  is  $ |v_i,v_j|_{T_k} + 1$, if
\[|v_i,v_j|_{T_k}= \|(v_i \rightarrow v_j)_{T_k}\|;\]
otherwise, if
\[  |v_i,v_j|_{T_k}= \|(v_r \rightarrow v_j)_{T_k}\| + \|(v_r \rightarrow v_i)_{T_k}\| \]
for some $ v_r$, this number 
 is
\begin{equation}\label{TJ33}
  \|(v_r \rightarrow v_j)_{T_k}\| - \|(v_r \rightarrow v_i)_{T_k}\| + 1.
  \end{equation} 
 Hence, there are at most $ |v_i,v_j|_{T_k} + 1=L$ generations of a tree (to which both $v_i,v_j$ belong to) which carry at least $j-i+1$ vertices.
This yields that there is a generation with at least
\[\frac{j-i+1}{L}\]
vertices. This gives us the following bound
\begin{equation}\label{TJ34}
  \mathbb{P} \left \{  |v_i,v_j|_{T_k}=L \big| v_i = a,  \ |I_k|\leq C|V_N|^{2/3}\right \}
\end{equation} 
\[
  \leq \mathbb{P} \left \{
\mbox{ there is a generation in } T_k  \mbox{ with at least } \frac{j-i+1}{L}
\mbox{ vertices }
\right \} .
\]

Note that the  tree-component  in $T_k$
is stochastically dominated by a tree (i.e., the number of vertices in a generation of one tree is stochastically dominated by the number of vertices in a corresponding generation of another tree) generated by a  branching process  with offspring distribution 
\begin{equation}\label{CBr}
\sum_r Bin\left(N_r, p_r \right),
\end{equation}
where $Bin\left(N_r, p_r \right), r\geq 1, $ denote  independent binomial random variables.
Let  $\zeta _0=1$, and let $ \zeta _k$ denote the number of offspring in $k$-th generations  of the latter process.
This allows us to derive a bound for the probability in (\ref{TJ34}):
\[\mathbb{P} \left \{  |v_i,v_j|_{T_k}=L \big| v_i = a,  \ |I_k|\leq C|V_N|^{2/3}\right \}\]
\begin{equation}\label{TJ35}
  \leq \mathbb{P} \left \{
\max \limits_{k \geq 1} \zeta_k  \geq  \frac{j-i+1}{L}
\right \} .
\end{equation}

Observe  that branching process  $\{\zeta_k\}_{k\geq 0}$ with offspring distribution
(\ref{CBr}) is critical by our assumption. It is known
\cite{Lindvall}
that under mild conditions (see for the details \cite{Lindvall}), which are fulfilled in our case,  a critical branching processes has the following property:
	\begin{equation}\label{TJ31}
		n \cdot \mathbb{P} \left ( \max\limits_{k \geq 1} \zeta_k > n \right ) \rightarrow 1 \text{ as } n\rightarrow \infty .
	\end{equation} 
Now taking into account
 assumption that 
 $$j-i \geq |V_N|^{1/6}
 $$ for any $L\leq C_1\log N$
we derive from (\ref{TJ35}) with a help of  (\ref{TJ31}):
\[\mathbb{P} \left \{  |v_i,v_j|_{T_k}=L \big| v_i = a,  \ |I_k|\leq C|V_N|^{2/3}\right \}\]
\begin{equation*}\label{TJ322}
  \leq \mathbb{P} \left \{
\max \limits_{k \geq 1} \zeta_k  \geq  \frac{j-i+1}{L}
\right \} \leq 2   \dfrac{C_1\log N}{  N^{1/3} }
\end{equation*}
for all large $N$, which proves Lemma \ref{LT3}.
\hfill$\Box$

\subsubsection{Proof of Lemma \ref{LT2}.}
Fix $C_1>2C_0$ where $C_0$ is a constant as in 
Corollary \ref{LMarkov3}.

Let us split the condition event in
\[
\mathbb{P} \left \{ v_j = b \big| v_i = a, |v_i,v_j|_{T_k}=L, \ |I_k|\leq C|V_N|^{2/3} \right \}
\]
  into two following cases: 
  when
	\begin{equation}\label{Case1a}
          |v_i,v_j|_{T_k}=\|(v_i \rightarrow v_j)_{T_k}\|=L,
	\end{equation}
	i.e., there is a directed path of length $L$ from $v_i$ to $v_j$,
	and when for some $v_r$ with $r<i<j$
	\begin{equation}\label{Case1b}
	|v_i,v_j|_{T_k}=\|(v_r \rightarrow v_i)_{T_k}\|+
	\|(v_r \rightarrow v_j)_{T_k}\|=L,
	\end{equation}
	i.e., the shortest distance between $v_i$ and $v_j$ is through vertex $v_r$. 
	
	By the property (\ref{TJ1})
	we have in the case (\ref{Case1a})
	\[\min_{{\bf A} : A_L \not\owns b}
	\mathbb{P}\{\tilde{X}_{{\bf A}, L+1}=b \big|\tilde{X}_{{\bf A},1}=a\}\]
	\begin{equation}\label{TJ26}
	\leq
	\mathbb{P}\{v_j=b \big| v_i=a, \ |v_i, v_j|_{T_k}=\|(v_i \rightarrow v_j)_{T_k}\|=L, \  |I_k|=K\}
	\end{equation}
	\[
	\leq \max_{{\bf A}: A_L \not\owns b}
	\mathbb{P}\{\tilde{X}_{{\bf A}, L+1}=b \big|\tilde{X}_{{\bf A},1}=a\},
	\]
	where the minimum and the maximum are taken over
	all sequences $${\bf A}=(A_1, \ldots, A_{L+1})$$ such 
	that
	\[
	A_1 \subseteq A_2\subseteq \ldots \subseteq A_{L+1} \subset V_N, \ \left| A_{L+1} \right|\leq K 
	\]
	(see (\ref{seqA1})).

        By the assumption $|I_k|=K\leq C|V_N|^{2/3}$,
	and $L> C_1 \log N >2C_0 \log N$, as we have set. Hence, we are in the conditions of Lemma \ref{LEp}, which together with  (\ref{TJ26}) gives us the upper bound
	\[
	\mathbb{P}\{v_j=b \big| v_i=a, \ |v_i, v_j|_{T_k}=\|(v_i \rightarrow v_j)_{T_k}\|=L, \  |I_k|=K\}
	\]
		\begin{equation}\label{TJ25*}
                  \leq \max \left( \frac{4c^2 \log N}{N^2}, p(a,b)
\right)
                    (1+o(1)).
	\end{equation}
Also under extra assumption (\ref{cMk}) 
we can apply 
	Corollary \ref{LMarkov3} to prove convergence of the upper and the lower bounds in (\ref{TJ26}) to the same limit, resulting in 
	\[
	\mathbb{P}\{v_j=b \big| v_i=a, \ |v_i, v_j|_{T_k}=\|(v_i \rightarrow v_j)_{T_k}\|=L, \  |I_k|=K\}
	\]
		\begin{equation}\label{TJ25}
	=\frac{1}{|V_N|}(1+o(1)).
	\end{equation}

	Next under condition (\ref{Case1b}) we have by (\ref{TJ33}) 
 $$
	\|(v_r \rightarrow v_j)_{T_k}\|\geq L/2>C_0 \log N.$$
	Loosely speaking along this  long path there is a mixing on a graph.
For any $u\in V_N$ we have by construction (\ref{TJ1}) the following bounds 
	\[
	\min_{{\bf A} : A_{L_1} \not\owns b}
	\mathbb{P}\{\tilde{X}_{{\bf A}, L_1+1}=b \big|\tilde{X}_{{\bf A},1}=u\},
	\]
	\begin{equation*}\label{TJ24}
	\begin{array}{ll}
	\mathbb{P}\{v_j=b \big| & v_i=a, \ |v_i, v_j|_{T_k}=\|(v_r \rightarrow v_j)_{T_k}\|
	+\|(v_r \rightarrow v_i)_{T_k}\|
	=L, \\ \\
	&v_r=u, \ \|(v_r \rightarrow v_j)_{T_k}\|= L_1, \ 
	 |I_k|=K\}
	\end{array}
		\end{equation*}
		\[
		\leq \max_{{\bf A} : A_{L_1} \not\owns b}
		\mathbb{P}\{\tilde{X}_{{\bf A}, L_1+1}=b \big|\tilde{X}_{{\bf A},1}=u\},
		\]
		where the minimum and the maximum are taken over
		all sequences $${\bf A}=(A_1, \ldots, A_{L_1+1})$$ such 
		that
		\[
		A_1 \subseteq A_2\subseteq \ldots \subseteq A_{L_1+1}\subset V_N, \ \left| A_{L_1+1} \right|\leq K . 
		\]
		(see (\ref{seqA1})).
		Here again we are in the conditions of Lemma  \ref{LEp},
                and under extra condition (\ref{cMk}) also
                in the conditions of Corollary \ref{LMarkov3}. Hence, application of the latter mentioned results yield again 
                asymptotically uniform distribution of $v_j$ as in 
		(\ref{TJ25}) under condition (\ref{cMk}), and otherwise, a bound as in 
(\ref{TJ25*}). Hence, statements (\ref{TM30}) and (\ref{TM30*}) follow, which finishes the proof.
                 \hfill$\Box$

		\bigskip
		
		Finally, the statement of Theorem \ref{ThIndep} follows
		by (\ref{TM28n}) after application of the results of Lemma \ref{LT1}, Lemma \ref{LT2},
		and Lemma \ref{LT3} to the respective terms on the right in (\ref{TM28n}).
		\hfill$\Box$
  
 \subsection{Asymptotic properties of the drift $\mathcal{D}(k)$.}\label{SectionAsymPropDrift}
 
 \subsubsection{$\mathcal{D}_1(k)$.}
 As defined in (\ref{TSe10}) consider 
 \[
 \mathcal{D}_1(k+1) =
 \sum\limits_{j=2}^{k}  \sum_{r=1}^N p_r\sum_{i=1}^{j-1} \mathbb{E} \left\{ 
 \sum_{s=1}^N\mathbb{E} \left\{ \left|{N_r}(v_j) \cap \mathcal{N}_s(v_i)  \right| \big| \mathcal{F}_{j-1}  \right\} \big| \mathcal{F}_{j}^z  \right\}
\]
\[
 =
\sum\limits_{j=2}^{k}  \sum_{r=1}^N p_r\sum_{i=1}^{j-1} \mathbb{E} \left\{ \left|{N_r}(v_j) \cap \mathcal{N}(v_i)  \right|  \big| \mathcal{F}_{j}^z  \right\}
\]
\begin{equation}\label{YN}
 =\sum\limits_{j=2}^{k} \sum_{i=1}^{j-1} \mathbb{E} \left\{ 
 \left| \mathcal{N}(v_j) \cap \mathcal{N}(v_i)  \right| \big| \mathcal{F}_{j}^z  \right\}.
 \end{equation}
 First we study $\mathbb{E} \mathcal{D}_1(k+1) $. 
 
 \begin{Lemma}\label{ThExpectationY_N}
 	Let $T>0$ be fixed arbitrarily and let $k = s |V_N|^{2/3} = s N^{4/3}$, $s\in [0,T]$. Then
 	\[
 	\mathbb{E} \mathcal{D}_1(k+1) = \dfrac{k^2}{2|V_N| } \left ( 1 + o(1) \right )  +o\left(\dfrac{1}{|V_N|^{1/2}} \right)\]
 	\[= \dfrac{s^2}{2}|V_N|^{1/3 } \left ( 1 + o(1) \right )  +o\left(\dfrac{1}{|V_N|^{1/2}} \right),
 	\]
 	uniformly in $s\in [0,T]$.
 \end{Lemma}
 
 \begin{proof}
   From (\ref{YN}) we obtain first
   \begin{eqnarray}\label{YN1}
 \mathbb{E} \mathcal{D}_1(k+1)  = \sum\limits_{j=2}^{k} \sum_{i=1}^{j-1} \mathbb{E} 
 \left| \mathcal{N}(v_j) \cap \mathcal{N}(v_i)  \right| .
 \end{eqnarray}
   Observe that
 	\[
        \mathbb{E} \left \{ \left | \mathcal{N}(v_j)  \cap \mathcal{N}(v_i) \right | \right \} \leq \max \limits_{u\neq v, u,v \in V_N}
             \mathbb{E} \left \{ \left | \mathcal{N}(v) \cap \mathcal{N}(u) \right | \right \} \]\[
	 = \max \limits_{u\neq v, u,v \in V_N}\sum_{x \in  V_N}p(u,x)p(x,v).
       	\]
        Therefore applying here bound (\ref{tt1}) we get
        	\begin{equation}\label{EstExpectationNv_iNv_j}
          \mathbb{E} \left \{ \left | \mathcal{N}(v_j) \cap \mathcal{N}(v_i) \right | \right \} \leq  \dfrac{5 c^2 \log N}{N^2}
 	\end{equation}
      uniformly in $1\leq i<j<k$.  
 	For any   $M(k) <k-2 $ let us split the sum in (\ref{YN1}) as follows
 	\begin{equation}\label{Expectation_Y_N}
 	 \mathbb{E}\mathcal{D}_1(k+1) = \sum\limits_{j=2}^{k} \sum_{i=1}^{j-1}  
 	\mathbb{E} \left\{ \left|{\mathcal{N}}(v_j) \cap \mathcal{N}(v_i)  \right| \right\}
 	\end{equation}
 	 \[= \sum\limits_{j=1}^{M(k)+1} \sum_{i=1}^{j - 1}  
 	\mathbb{E} \left\{ \left|\mathcal{N}(v_j) \cap \mathcal{N}(v_i)  \right| \right\}+ \sum\limits_{j=M(k)+2}^{k} \sum_{i=j-M(k)}^{j - 1}  
 	\mathbb{E} \left\{ \left|\mathcal{N}(v_j) \cap \mathcal{N}(v_i)  \right| \right\}
 	\]
 	\[ +\sum\limits_{j=M(k)+2}^{k} \sum_{i=1}^{j-M(k) - 1}  
 	\mathbb{E} \left\{ \left|\mathcal{N}(v_j) \cap \mathcal{N}(v_i)  \right| \right\}.
 	\]
        The first two sums on the right we can bound with a help of
        $(\ref{EstExpectationNv_iNv_j})$:
 	\[
        \sum\limits_{j=1}^{M(k)+1} \sum_{i=1}^{j - 1}  
        \mathbb{E} \left\{ \left|\mathcal{N}(v_j) \cap \mathcal{N}(v_i)  \right| \right\}+   \sum\limits_{j=M(k)+2}^{k} \sum_{i=j-M(k)}^{j - 1}	\mathbb{E} \left\{ \left|\mathcal{N}(v_j) \cap \mathcal{N}(v_i)  \right| \right\}
 	\]
 	\begin{equation}\label{TJ40}
 	\leq 5 c^2 \dfrac{(k M(k)+M^2(k)) \log N}{N^2}.
 	\end{equation}
        Setting 
        \begin{equation*}\label{Mk}
        M(k) = \max\left\{ \sqrt{k}, |V_N|^{1/6}
        \right\},
        \end{equation*}
 	we get from (\ref{TJ40}) 
 	\[
 	\sum\limits_{j=1}^{M(k)+1} \sum_{i=1}^{j - 1}  
 	\mathbb{E} \left\{ \left|\mathcal{N}(v_j) \cap \mathcal{N}(v_i)  \right| \right\}+   \sum\limits_{j=M(k)+2}^{k} \sum_{i=j-M(k)}^{j - 1}	\mathbb{E} \left\{ \left|\mathcal{N}(v_j) \cap \mathcal{N}(v_i)  \right| \right\}
 	\]
 	\begin{equation}\label{TJ42}
 	\leq 5 c^2 \dfrac{(k^{3/2} + k N^{1/3} + k + N^{2/3}) \log N}{N^2}= o\left(\dfrac{k^{2}}{|V_N|} \right)+o\left(\dfrac{1}{|V_N|^{1/2}} \right).
 	\end{equation}

 	Consider the last sum  in $(\ref{Expectation_Y_N})$ in the case when $k>|V_N|^{1/6}$ (otherwise, it is zero)
 	\begin{eqnarray}\label{TJ43}
 	&& \sum\limits_{j=M(k)+2}^{k} \sum_{i=1}^{j-M(k) - 1}  
 	\mathbb{E} \left\{ \left|\mathcal{N}(v_j) \cap \mathcal{N}(v_i)  \right| \right\} \\ \nonumber
 	&& = \sum\limits_{j=M(k)+2}^{k} \sum_{i=1}^{j-M(k) - 1}  
           \sum\limits_{a \in V_N} \sum\limits_{b \in V_N} \mathbb{P} (v_j = b, v_i = a) \sum\limits_{u \in V_N} p(a, u) p(u, b).
	\end{eqnarray}
	By Theorem \ref{ThIndep} (which conditions are fulfilled here) for $j - i \geq |V_N|^{1/6}$ and all $a \neq b$
 	\[
 	\mathbb{P} (v_j = b \big| v_i = a) = \dfrac{1}{N^2} \left ( 1 + o(1) \right ),
      \]
      uniformly in $a, b \in V_N$. Hence, taking also into
      account mentioned previously property that
		\begin{equation*}
 		\mathbb{P}(v_i = a) = \dfrac{1}{N^2},
        \end{equation*}
              we derive from  (\ref{TJ43}) 
	\begin{eqnarray*}\label{TJ44}
 	&& \sum\limits_{j=M(k)+2}^{k} \sum_{i=1}^{j-M(k) - 1}  
 	\mathbb{E} \left\{ \left|\mathcal{N}(v_j) \cap \mathcal{N}(v_i)  \right| \right\}  \\ \nonumber
 	&& = \sum\limits_{j=M(k)+2}^{k} \sum_{i=1}^{j-M(k) - 1}  
           \sum\limits_{a \in V_N} \sum\limits_{b \in V_N}
\dfrac{1}{N^4} \left ( 1 + o(1) \right )
           \sum\limits_{u \in V_N} p(a, u) p(u, b)
\\ \nonumber
&& = \left ( 1 + o(1) \right )
\sum\limits_{j=M(k)+2}^{k} \sum_{i=1}^{j-M(k) - 1}  
\sum\limits_{u \in V_N}\dfrac{1}{N^4}
\left (\sum\limits_{a \in V_N}p(a, u)\right )
 \left (\sum\limits_{b \in V_N}
  p(u, b)\right ).
  \end{eqnarray*}
  Applying now (\ref{Ecr}) to the last expression we get
  \begin{eqnarray}\label{TJ45}
  && \sum\limits_{j=M(k)+2}^{k} \sum_{i=1}^{j-M(k) - 1}  
  \mathbb{E} \left\{ \left|\mathcal{N}(v_j) \cap \mathcal{N}(v_i)  \right| \right\} 
  \\ \nonumber
  && = \sum\limits_{j=M(k)+2}^{k} \sum_{i=1}^{j-M(k) - 1}  
  \left ( 1 + o(1) \right )\dfrac{1}{N^2}\left (1 - \dfrac{2 c}{N} + O(1 / N^2)\right )^2 \\ \nonumber
  && 
   = \dfrac{k^2}{2 N^2}\left ( 1 + o(1) \right ),
  \end{eqnarray}
  where the last $o(1)$ is uniform in $k$. 
  
  Substituting (\ref{TJ45}) and (\ref{TJ42}) into (\ref{Expectation_Y_N}) gives us the statement of the Lemma.
\end{proof}

 \begin{Lemma}\label{Lm_Convergence_Y_N}
 	The following convergence takes place when $N\rightarrow \infty$
 	\begin{equation}\label{Lm_Convergence_Y_N_eq}
 	\sup\limits_{0 \leq s \leq T} \left | \dfrac{\mathcal{D}_1 \left (1 + [|V_N|^{2/3} s] \right ) - \mathbb{E} \left \{ \mathcal{D}_1\left (1 + [|V_N|^{2/3} s] \right ) \right \} }{|V_N|^{1/3}} \right | \overset{{P}}{\longrightarrow} 0.
 	\end{equation}
 \end{Lemma}
 \begin{proof}
 	Recall that $|V_N|=N^2$.
 	Let  us denote 
 	\begin{equation}\label{Y}
 	Y_j = N^{2/3}\mathbb{E} \left\{ \left. \sum\limits_{r = 1}^{N} p_r \sum\limits_{i = 1}^{j-1} 
 	\left|N_r(v_j) \cap \mathcal{N}(v_i)  \right| \right| \mathcal{F}_{j}^z  \right\}.
 \end{equation}
 	In this notation (\ref{YN}) becomes
 	\[\mathcal{D}_1(k+1)=\frac{1}{N^{2/3}} \sum_{j=2}^k Y_j,\]
 	and (\ref{Lm_Convergence_Y_N_eq}) will follow if both
 	\begin{equation}\label{TJ46*}
 	\sup\limits_{0 \leq s \leq |V_N|^{-1/3}} \left | \dfrac{1}{N^{4/3}} \sum\limits_{j = 2}^{1 + [N^{4/3} s]} \left( Y_j - \mathbb{E} Y_j \right)  \right | \overset{{P}}{\longrightarrow} 0,
 	\end{equation}
 	and
\begin{equation}\label{TJ46}
\sup\limits_{|V_N|^{-1/3} <s \leq T} \left | \dfrac{1}{N^{4/3}} \sum\limits_{j = N^{2/3}}^{1 + [N^{4/3} s]} \left( Y_j - \mathbb{E} Y_j \right)  \right | \overset{{P}}{\longrightarrow} 0.
\end{equation}

To establish  (\ref{TJ46*}) we derive first
for an arbitrarily fixed $\varepsilon >0$ 

\begin{equation}\label{TJ47*}
\mathbb{P} \left \{\sup\limits_{0 \leq s\leq  |V_N|^{-1/3}} \left | \dfrac{1}{N^{4/3}} \sum\limits_{j = 2}^{1 + [N^{4/3} s]} \left( Y_j - \mathbb{E} Y_j \right)  \right | > \varepsilon \right \}
\end{equation}
\[
\leq \dfrac{1}{\varepsilon} \mathbb{E} \left ( \sup\limits_{0 \leq s \leq|N|^{-2/3}}  \dfrac{1}{N^{4/3}} \sum\limits_{j = 2}^{1 + [N^{4/3} s]} |Y_j - \mathbb{E} Y_j|   \right ) 
\leq 
\dfrac{2}{\varepsilon N^{2/3}} \max\limits_{j \leq N^{2/3}+1} {\mathbb{E} |Y_j - \mathbb{E} Y_j|}.
\]
Since $Y_j$ is non-negative, here we have
\[\mathbb{E} |Y_j - \mathbb{E} Y_j|\leq 2\mathbb{E} Y_j,
	\]
	where by the definition (\ref{Y}) and a uniform bound	(\ref{EstExpectationNv_iNv_j})
		\begin{equation}\label{Y1}
\mathbb{E}	Y_j \leq  N^{2/3}  \sum\limits_{i = 1}^{j-1} \mathbb{E}  
	\left|\mathcal{N}(v_j) \cap \mathcal{N}(v_i)  \right|
	\leq N^{2/3} j \dfrac{5 c^2 \log N}{N^2}
	\end{equation}
	for all $j\leq 1 + N^{2/3} $. Substituting this bound into (\ref{TJ47*}) we obtain
	\begin{equation}\label{Tau1}
	\mathbb{P} \left \{\sup\limits_{0 \leq s\leq  |V_N|^{-1/3}} \left | \dfrac{1}{N^{4/3}} \sum\limits_{j = 2}^{1 + [N^{4/3} s]} \left( Y_j - \mathbb{E} Y_j \right)  \right | > \varepsilon \right \} 
	\end{equation}
	\[
	\leq \frac{c_1}{\varepsilon N^{2/3}} \frac{N^{4/3}\log N}{N^2},
	\]
	which goes to zero as $N\rightarrow \infty$. This proves  (\ref{TJ46*}).

       We are left to prove (\ref{TJ46}). Consider again as in (\ref{TJ47*})
 	for an arbitrarily fixed $\varepsilon >0$ 
 	\begin{equation}\label{TJ47}
 	\mathbb{P} \left \{\sup\limits_{|V_N|^{-1/3} \leq s \leq T} \left | \dfrac{1}{N^{4/3}} \sum\limits_{j = N^{2/3}}^{1 + [N^{4/3} s]} \left( Y_j - \mathbb{E} Y_j \right)  \right | > \varepsilon \right \}
 	 	\end{equation}
 	\[
 	\leq \dfrac{1}{\varepsilon} \mathbb{E} \left ( \sup\limits_{N^{-2/3} \leq s \leq T}  \dfrac{1}{N^{4/3}} \sum\limits_{j = N^{2/3}}^{1 + [N^{4/3} s]} |Y_j - \mathbb{E} Y_j|   \right ) 
 	\leq \dfrac{T}{\varepsilon} \max\limits_{N^{2/3}\leq j \leq T N^{4/3}} \sqrt{\mathrm{Var} (Y_j) }.
 	\]
 	By the property of the conditional expectation we have
 	\begin{equation*}
 	\mathrm{Var} (Y_j) \leq \mathrm{Var} \left (N^{2/3} \sum\limits_{r = 1}^{N} p_r \sum\limits_{i = 1}^{j-1} \left| N_r (v_j) \cap \mathcal{N}(v_i)  \right| \right )
 .	\end{equation*}
 	Therefore by proving
 	\begin{equation}\label{Y_N}
 	\max\limits_{N^{2/3} \leq j \leq T N^{4/3}} \mathrm{Var} \left (N^{2/3} \sum\limits_{r = 1}^{N} p_r \sum\limits_{i = 1}^{j-1} \left|N_r (v_j) \cap \mathcal{N}(v_i)  \right| \right ) \rightarrow 0
 	\end{equation}
 	we establish (\ref{TJ46}) via (\ref{TJ47}), and hence, the statement of lemma follows.
 	
 	To prove (\ref{Y_N}) we notice first that  random variables $\sum\limits_{i = 1}^{j-1} \left|N_r (v_j) \cap \mathcal{N}(v_i)  \right|$ are independent for different $r$. Hence,
 	\begin{equation}\label{TJ48}
 	\mathrm{Var} \left (N^{2/3} \sum\limits_{r = 1}^{N} p_r \sum\limits_{i = 1}^{j-1} 
 	\left| N_r (v_j) \cap \mathcal{N}(v_i)  \right| \right ) 
 	\end{equation}
 	\begin{eqnarray}
 	&&
 	=
 	N^{4/3} \sum\limits_{r = 1}^{N} p_r^2 \mathrm{Var} \left ( \sum\limits_{i = 1}^{j-1} \left| N_r (v_j) \cap \mathcal{N}(v_i)  \right| \right ) \nonumber \\ \nonumber
 	&& = N^{4/3} \sum\limits_{r = 1}^{N} p_r^2 \left ( \sum\limits_{i = 1}^{j-1} \mathrm{Var} \left| N_r (v_j) \cap \mathcal{N}(v_i) \right| + \right . \\ \nonumber
 	&& + \left . 2 \sum\limits_{i = 1}^{j-1} \sum\limits_{l = 1}^{i-1} \mathrm{Cov} \left ( \left| N_r (v_j) \cap \mathcal{N}(v_i)  \right|, \left| N_r (v_j) \cap \mathcal{N}(v_l)  \right| \right ) \right ) .
 	\end{eqnarray}
 	
 	Consider 
 	\begin{equation}\label{TJ51}
 	\sum\limits_{i = 1}^{j-1} \mathrm{Var} \left| N_r (v_j) \cap \mathcal{N}(v_i) \right|.
 	\end{equation}
 	Making use of independent Bernoulli random variables $\eta_{uv} \sim Be(p(u,v))$, introduced in (\ref{eta}), we derive
 	\begin{equation}\label{TJ49}
          \mathrm{Var} \left| N_r (v_j) \cap \mathcal{N}(v_i) \right| = \mathrm{Var} \left ( \sum\limits_{u \in N_r(v_j)} \eta_{uv_i} \right )
	\end{equation}
        
        \[
          = \mathbb{E} \ \mathrm{Var} \left \{ \sum\limits_{u \in N_r (v_j)} \eta_{uv_i} \big| v_j, v_i \right \} + \mathrm{Var}\ \mathbb{E} \left \{ \sum\limits_{u \in N_r (v_j)} \eta_{uv_i} \big| v_j, v_i \right \}  \]

 	\[= \mathbb{E} \sum\limits_{u \in N_r (v_j)} p(u, v_i) \left ( 1 - p(u, v_i) \right ) +  \mathrm{Var} \sum\limits_{u \in N_r (v_j)} p(u, v_i) \]
 	\[ = \mathbb{E} \sum\limits_{u \in N_r (v_j)} p(u, v_i)\left ( 1 - p(u, v_i) \right )\]
 	\[  + \mathbb{E} \left ( \sum\limits_{u \in N_r (v_j)} p(u, v_i) \right )^2 - \left ( \mathbb{E}  \sum\limits_{u \in N_r (v_j)} p(u, v_i)  \right )^2 
 \]
 	\[ \leq  \mathbb{E} \sum\limits_{u \in N_r (v_j)} p(u, v_i)  + \mathbb{E} \sum\limits_{u, w \in N_r (v_j):w\neq u} p(u, v_i) p(w, v_i)\]
 	\[ = \sum_{a,b \in V_N: a\neq b} \mathbb{P} \left ( v_j = b, v_i = a \right )\sum\limits_{u \in N_r (b)} \left (
 	 p(u, a) +  \sum\limits_{w \in N_r(b):w\neq u} p(w, a) p(a,u)\right ). 
 	\]
 	
 	Return to the sum in (\ref{TJ51}). We shall split it as follows
 	\[
 	\sum\limits_{i = 1}^{j-1} \mathrm{Var} \left| N_r (v_j) \cap \mathcal{N}(v_i) \right| 
 	\]
 	\begin{equation}\label{TJ52}
 	=\sum\limits_{i = 1}^{j- |V_N|^{1/6}} \mathrm{Var} \left| N_r (v_j) \cap \mathcal{N}(v_i) \right|+
 	\sum\limits_{i = j-|V_N|^{1/6}+1}^{j-1} \mathrm{Var} \left| N_r (v_j) \cap \mathcal{N}(v_i) \right| .
 	\end{equation}
 In the first sum, where
 	$j - i \geq |V_N|^{1/6}$  we have by (\ref{TM28}) from Theorem \ref{ThIndep} 
 	for all $a\neq b$
 	\begin{equation}\label{EstVjVi}
 	\mathbb{P} \left ( v_j = b, v_i = a \right ) = \dfrac{1}{N^4} \left ( 1 + o(1) \right ),
 	\end{equation}
 	in which case by (\ref{TJ49}) 
 	\begin{equation}\label{TJ50}
 	\sum\limits_{i = 1}^{j-|V_N|^{1/6}} \mathrm{Var} \left| N_r (v_j) \cap \mathcal{N}(v_i) \right| \leq j  \dfrac{1}{N^4} \left ( 1 + o(1) \right ) 
 	\end{equation}
 	\[\times 
 	\sum_{b \in V_N}
 	\sum\limits_{u \in N_r (b)} \left (
 	\sum_{a \in V_N}
 	p(u, a) +  
 	\sum\limits_{w \in N_r(b):w\neq u} 
 	\sum_{a \in V_N} p(w, a) p(a,u)\right ).
 	\]
 	Applying now results (\ref{Ecr}) and (\ref{tt1}) to  the two sums in the last brackets, correspondingly, we derive from (\ref{TJ50}) for all large $N$
 	\begin{equation}\label{TJ55}
 	\sum\limits_{i = 1}^{j- |V_N|^{1/6}} \mathrm{Var} \left| N_r (v_j) \cap \mathcal{N}(v_i) \right| \leq 2
 	\dfrac{j}{N^2} N_r \left ( 1 + N_r\dfrac{5 c^2 \log N}{N^2} \right ) 
 \end{equation}
 	\[ \leq 2\dfrac{T N^{4/3}N_r}{N^2} \left ( 1 + o(1) \right ) <3 \dfrac{T N_r}{N^{2/3}} 
 	\]
 	for all $j\leq T N^{4/3}$.
 	
 	We treat the remaining sum in (\ref{TJ52})
 	in a similar way, 
 	only instead of (\ref{EstVjVi}) we use  bound (\ref{TM28*}) from  Theorem \ref{ThIndep}, by	 which 
 		\begin{equation*}
 	\mathbb{P} \left ( v_j = b, v_i = a \right ) \leq \dfrac{c}{N^3},
 	\end{equation*}
 	and this  gives us
 	\begin{equation}\label{TJ54}
 	 \sum\limits_{i = j - |V_N|^{1/6} + 1}^{j - 1} \mathrm{Var} \left| N_r (v_j) \cap \mathcal{N}(v_i) \right| 
 	\leq 2c
 	\dfrac{N^{1/3}}{N} N_r \left ( 1 + N_r\dfrac{5 c^2 \log N}{N^2} \right ) 
 	\end{equation}
 	\[\leq 3 \dfrac{ c  N_r}{ N^{2/3}} \]
 	for all $j\leq k \leq T N^{4/3}$.
 	Combining now (\ref{TJ54}) and (\ref{TJ55}) in (\ref{TJ52}), we obtain
 	\begin{equation}\label{EstVar}
 	\sum\limits_{i = 1}^{j-1} \mathrm{Var} \left| N_r (v_j) \cap \mathcal{N}(v_i) \right| \leq 3 \dfrac{(T + c)  N_r}{N^{2/3}} 
 	\end{equation}
 	for all large $N$.
 	
 	Consider now terms with
 	covariances in (\ref{TJ48}). Using again variables $\eta_{u v}$
        we obtain first the following representation
        \[\left| N_r (v_j) \cap \mathcal{N}(v_i)  \right| =
 \sum\limits_{u \in N_r (v_j)} \eta_{u v_i} .
\]
Then
	\begin{equation}\label{Cov}
\mathrm{Cov} \left ( \left| N_r (v_j) \cap \mathcal{N}(v_i)  \right|, \left| N_r (v_j) \cap \mathcal{N}(v_l)  \right| \right ) 
	\end{equation}
        \[
          =  \mathbb{E}
          \left ( \sum\limits_{u \in N_r (v_j)} \eta_{u v_i} \sum\limits_{u' \in N_r (v_j)} \eta_{u' v_l}
          \right )\]
     \[
            - \left (\mathbb{E} \sum\limits_{u \in N_r (v_j)} \eta_{u v_i}\right ) \left (\mathbb{E}\sum\limits_{u \in N_r (v_j)} \eta_{u v_l}\right )
          \]

\[= \sum_{a,a', b \in V_N}
          \mathbb{P}\{v_l=a, v_i=a', v_j=b\}\sum\limits_{u \in N_r (b)} p (u, a')
          \sum\limits_{u' \in N_r (b)} p (u', a)\]
\[-   \left ( \sum_{a',b\in V_N}
    \mathbb{P}\{v_i=a', v_j=b\}\sum\limits_{u \in N_r (b)} p (u, a')\right )\]
\[\times
  \left ( \sum_{a,b \in V_N}
    \mathbb{P}\{v_l=a, v_j=b\}\sum\limits_{u \in N_r (b)} p (u, a)\right ).
\]

Consider separately different cases.
When both
\begin{equation}\label{condit}
  j-i \geq |V_N|^{1/6} \mbox{ and }  i-l \geq |V_N|^{1/6},
\end{equation}
by
Theorem \ref{ThIndep} we have for any three different values $a,a',b$
\begin{equation*}\label{mix3}
 \mathbb{P}\{v_l=a, v_i=a', v_j=b\}=\frac{1}{|V_N|^3}(1+o(1)),
	\end{equation*}
        in which case  covariance in (\ref{Cov}) becomes
	\begin{equation}\label{Cov2}
\mathrm{Cov} \left ( \left| N_r (v_j) \cap \mathcal{N}(v_i)  \right|, \left| N_r (v_j) \cap \mathcal{N}(v_l)  \right| \right ) = \frac{N_r^2}{|V_N|^2} (1 + o(1)).
	\end{equation}

        In all remaning cases of indices $1\leq l<i<j<k$  when (\ref{condit}) is not held, we derive with a help of Theorem \ref{ThIndep}
	\[
          \mathbb{P}\{v_l=a, v_i=a', v_j=b\}\]
        \[=\mathbb{P}\{v_j=b \mid
       v_i=a',    v_l=a\}\mathbb{P}\{v_i=a' \mid
       v_l=a\}\mathbb{P}\{
       v_l=a\}\]
        \begin{equation}\label{TJab}
        \leq
       B \frac{\log ^2N}{N^2}p(a,a')p(b,a'),
        \end{equation}
        where $B$ is some positive constant.
 Therefore we can bound 
 \[\sum_{a,a', b \in V_N}
 \mathbb{P}\{v_l=a, v_i=a', v_j=b\}\sum\limits_{u \in N_r (b)} p (u, a')
 \sum\limits_{u' \in N_r (b)} p (u', a)\]
 \[\leq B \frac{\log ^2N}{N^2}
 \sum_{b \in V_N}
 \sum\limits_{u,u' \in N_r (b)}\sum_{a' \in V_N} p (u, a')p(a',b)\sum_{a \in V_N}
 p (u', a)p(a,a').\]
 Making use of the bound (\ref{tt1}) in the last formula, we derive
 \[\sum_{a,a', b \in V_N}
 \mathbb{P}\{v_l=a, v_i=a', v_j=b\}\sum\limits_{u \in N_r (b)} p (u, a')
 \sum\limits_{u' \in N_r (b)} p (u', a)\]
 \begin{equation}\label{Cov3}
 \leq B \frac{\log ^2N}{N^2}
 \sum_{b \in V_N}
 \sum\limits_{u,u' \in N_r (b)}
 \left(5c^2 \frac{\log N}{N^2}\right)^2\leq B_1 \frac{\log ^4 N}{N^4}
 N_r^2,
 \end{equation}
 where $B_1$ is some positive constant.
 
        In the same manner we derive
        	\begin{equation}\label{TJ60}
        \mathbb{E} \sum\limits_{u \in N_r (v_j)} \eta_{u v_i}
        \leq  \sum_{b,b' \in V_N}
          \mathbb{P}\{v_i=b, v_j=b'\}\sum\limits_{u \in N_r (b')} p (u, b),
        \end{equation}
        where by the inequality (\ref{TM28*})  from Theorem \ref{ThIndep} 
     \begin{equation*}\label{TJ61*}
        \mathbb{P}\{v_i=b, v_j=b' \}=\mathbb{P}\{v_j=b' \mid
       v_i=b\}\mathbb{P}\{v_i=b\}
      \leq 5c 
       \frac{\log N}{N^2}p(b,b')
     \end{equation*}
     for all $1\leq i<j\leq T|V_N|^{2/3}$. This bound together with  (\ref{TJ60}) give us
    	\begin{equation*}\label{TJ61}
        \mathbb{E} \sum\limits_{u \in N_r (v_j)} \eta_{u v_i} 
        \leq 5c 
       \frac{\log N}{N^2}
        \sum_{c \in V_N}
          \sum\limits_{u \in N_r (c)}  \sum_{b \in V_N} p (u, b)p(b,c).
        \end{equation*}
Making use of bound (\ref{tt1}) in the last formula we obtain
	\begin{equation}\label{TJ62}
        \mathbb{E} \sum\limits_{u \in N_r (v_j)} \eta_{u v_i}
        \leq B
       \frac{\log N}{N^2}
        \sum_{c \in V_N}
          \sum\limits_{u \in N_r (c)}  \frac{\log N}{N^2}  \leq B_2
       \frac{\log^2 N}{N^2}N_r
        \end{equation}
uniformly in $1\leq i<j\leq T|V_N|^{2/3}$, where $B_2$ is  some positive constant.
Bounds (\ref{TJ62}) and (\ref{Cov3}) 
yield the following bound for the 
covariance in (\ref{Cov}):
\begin{equation}\label{Cov4}
\mathrm{Cov} \left ( \left| N_r (v_j) \cap \mathcal{N}(v_i)  \right|, \left| N_r (v_j) \cap \mathcal{N}(v_l)  \right| \right ) \leq B_3 \frac{N_r^2}{|V_N|^2} \log^4 N,
\end{equation}
where $B_3$ is  some positive constant.

Now we can bound variance in (\ref{TJ48}) with a help of bounds (\ref{EstVar}), (\ref{Cov2}), and (\ref{Cov4}):
	\[
 	\mathrm{Var} \left (N^{2/3} \sum\limits_{r = 1}^{N} p_r \sum\limits_{i = 1}^{j-1} 
 	\left| N_r (v_j) \cap \mathcal{N}(v_i)  \right| \right ) 
    \]
    \[\leq N^{4/3} \sum\limits_{r = 1}^{N} p_r^2 B_4 
    \left(
    \dfrac{(T + c) N_r}{N^{2/3}} 
    + 2 j^2
   \frac{N_r^2}{|V_N|^2} \log^4 N
   \right ), \] 
   where $B_4$ is  some positive constant. Hence, after straightforward computation we derive from here
   	\begin{equation*}\label{TJ56}
   \mathrm{Var} \left (N^{2/3} \sum\limits_{r = 1}^{N} p_r \sum\limits_{i = 1}^{j-1} 
   \left| N_r (v_j) \cap \mathcal{N}(v_i)  \right| \right ) 
   = O\left(\frac{\log^4 N}{N} \right).
   \end{equation*}
   for all $j\leq T N^{4/3}$,     which converges to $0$ as $N \rightarrow \infty$. This confirms (\ref{Y_N}) and completes the proof.	
 \end{proof}
 
 As a straightforward corollary of the last lemma and Lemma \ref{ThExpectationY_N} we derive the following result.
 
 \begin{Cor}\label{D1} 	For any fixed $T>0$
 \[
 \sup\limits_{0 \leq s \leq T} 
 \left|
  \dfrac{\mathcal{D}_1 \left (1 + [N^{4/3} s] \right )  }{N^{2/3}} -\frac{s^2}{2} \right| \overset{{P}}{\longrightarrow} 0.
 \]
\end{Cor}
\hfill$\Box$

 \subsubsection{$\mathcal{D}_2(k)$.}
 Consider as defined in (\ref{TSe10})
 \begin{equation*}\label{TJ63}
 \mathcal{D}_2(k):= \sum\limits_{j=2}^{k}  \sum_{i=1}^{j-1} {\bf I}_{{\cal A}_{i-1}=\emptyset} \sum_{r=1}^N p_r
 \mathbb{P} \left\{ v_i \in {N_r}(v_j) \big| \mathcal{F}_{j}^z  \right\}. 
 \end{equation*}
 The expectation of this term is
 \begin{equation}\label{TJ76}
 \mathbb{E} \mathcal{D}_2(k) = \sum\limits_{j=2}^{k} \sum_{r=1}^N p_r\sum_{i=1}^{j-1} \mathbb{E}  {\bf I}_{{\cal A}_{i-1}=\emptyset}
 	{\bf I}_{
 		v_i \in {N_r}(v_j)
 	}
 \end{equation}
 \[=\sum\limits_{j=2}^{k} \sum_{i=1}^{j-1} \mathbb{E}  {\bf I}_{{\cal A}_{i-1}=\emptyset}
  \sum_{r=1}^N p_r\mathbb{P} \left\{ v_i \in {N_r}(v_j) \big| {\cal A}_{i-1}=\emptyset  \right\}.
 \]
 With a help of results of Theorem \ref{ThIndep} we derive first
 \[
  \sum_{r=1}^N p_r
 \mathbb{P} \left\{ v_i \in {N_r}(v_j) \big|  {\cal A}_{i-1}=\emptyset   \right\}
 \]
 \[= \sum_{r=1}^N p_r
 \sum\limits_{b\in V_N}
 \sum\limits_{a\in N_r(b)}
 \mathbb{P} \left\{ v_i=a, v_j=b \big|  {\cal A}_{i-1}=\emptyset   \right\}\]
 \[
 \leq 
 \sum\limits_{b\in V_N} \sum_{r=1}^N p_r
 \sum\limits_{a\in N_r(b)} 
 p(a,b)\frac{O\left(\log N\right)}{N^2}
 = \frac{O\left(\log N\right)}{N^2}\]
 uniformly in $0<i<j<k$. 
 Substituting this bound into (\ref{TJ76}) we get
 \[\mathbb{E} \mathcal{D}_2(k) \leq \frac{O\left(\log N\right)}{N^2}\sum\limits_{j=2}^{k} \mathbb{E}\sum_{i=1}^{j-1}  {\bf I}_{{\cal A}_{i-1}=\emptyset}
 \]
\[ \leq 
O\left(\frac{\log N}{N^2} \right)
 \sum\limits_{j=2}^{k}\mathbb{E}
\left ( 1 + \max\limits_{i < j-1} |z(i)| \right )
\]
\begin{equation}\label{TWA}
\leq O\left
(\frac{\log N}{N^2} \right)
k \mathbb{E}\left ( 1 + \max\limits_{j < k} |z(j)| \right ).
\end{equation}

 \begin{Lemma}\label{Lbz}
 	For any fixed $T$ and for any $K >1$ 
 	\begin{equation*}\label{Bz}
 	\mathbb{P} \left ( \max\limits_{k \leq N^{4/3} T} | z(k) | > K N^{2/3} \right ) \leq 2 \frac{\sqrt{T}+T^2}{K}. 
 	\end{equation*}
      \end{Lemma}

 	\begin{proof}
 		We follow the idea of the proof from \cite{Aldous}. Define
 		\begin{eqnarray*}
 		T_{N}^{*} = \min \{ s : |z(s)| > K N^{2/3}\}, \\
 		T_N = \min \{ T_{N}^{*}, [ N^{4/3} T ] \}.
 		\end{eqnarray*}
 		Then
 		\begin{eqnarray}\label{Prop_Z_stoch_bound_Eq1}
 		&&\mathbb{P} \left ( \max\limits_{k \leq N^{4/3} T} | z(k) | > K N^{2/3} \right ) = \mathbb{P} \left ( | z(T_N) | > K N^{2/3} \right )  \\ \nonumber
 		&&\leq \mathbb{P} \left ( | \mathcal{M}(T_N) + \mathcal{D}(T_N) | > K N^{2/3} \right ) \leq \dfrac{\mathbb{E} |\mathcal{M}(T_N)| + \mathbb{E} |\mathcal{D}(T_N)|}{K N^{2/3}}  \\ \nonumber
 		&&\leq \dfrac{(\mathbb{E} (\mathcal{M}(T_N))^2)^{1/2} + \mathbb{E} |\mathcal{D}(T_N)|}{K N^{2/3}}.
 		\end{eqnarray}
 		Firstly, by optional sampling theorem
 		\begin{eqnarray}\label{TJ65}
 		\mathbb{E} (\mathcal{M}(T_N))^2  \leq \sum\limits_{j = 1}^{[ N^{4/3} T ] - 1} \mathbb{E} (\Delta \mathcal{M}(j))^2,
 		\end{eqnarray}
 		where by definition (\ref{Ms})
 		\begin{eqnarray}\label{Var_est}
 		&& \mathbb{E} (\Delta \mathcal{M}(j))^2 = \mathbb{E} \mathrm{Var} \left ( \Delta z(j) \big| \mathcal{F}_{j}^{z} \right ) \leq \mathrm{Var} \left ( \Delta z(j) \right )  \\ \nonumber
 		&& = \mathbb{E} \left ( \Delta z(j) - \mathbb{E} \Delta z(j) \right )^2   \\ \nonumber
 		&& = \mathbb{E} |\mathcal{N} (v_j, j)|^2 - \left ( \mathbb{E} |\mathcal{N} (v_j, j)| \right  )^2 .
 		\end{eqnarray}
 		Application of Lemma \ref{Lm_Est_ENvjj} to (\ref{Var_est}) gives us
 		\begin{eqnarray}\label{Var_est2}
 		\mathbb{E} (\Delta \mathcal{M}(j))^2 \leq 1 + O(1/N)
 		\end{eqnarray}
 		uniformly in $j<k$, implying by (\ref{TJ65})
 		\begin{equation}\label{TJ64}
 		\mathbb{E} (\mathcal{M}(T_N))^2 \leq N^{4/3} T (1 + o(1)) .
 		\end{equation}
 		
 		Now we consider taking into account definition (\ref{TSe10})
 		\begin{equation}\label{TJ66}
 		\mathbb{E} |\mathcal{D}(T_N)| \leq T_N \dfrac{3c}{N} 
 		+ \mathbb{E} \mathcal{D}_1(T_N) +\mathbb{E}  \mathcal{D}_2(T_N) +
 		\mathbb{E} \mathcal{D}_3(T_N) +\mathbb{E} \mathcal{D}_4(T_N) .
 		\end{equation}
 		By the result of Lemma
 		\ref{ThExpectationY_N} here we have
 			\begin{equation}\label{TJ67}
 			\mathbb{E} \mathcal{D}_1(T_N+1) \leq \dfrac{T_N^2}{|V_N| } .
 		\end{equation}
We shall bound each of  $\mathbb{E} \mathcal{D}_i$ separately.
By definition (\ref{TSe10})  and observation in (\ref{TWA}) we have
\begin{equation*}\label{TJ68}
  \mathbb{E} \mathcal{D}_2(T_N) = \sum\limits_{j=2}^{T_N-1} \sum_{i=1}^{j-1}
 \sum_{r=1}^N p_r \mathbb{P} \left\{{\cal A}_{i-1}=\emptyset,  v_i \in {N}_r(v_j) \right\} \end{equation*}
\[= \frac{O(\log N)}{N^2} T_N \left ( 1 + \max\limits_{j < T_N} |z(j)| \right )
\leq \frac{O(\log N)}{N^2} T N^{4/3}  \left ( 1 + K N^{2/3} \right ), 
\]
where the last inequality is due to the definition of $T_N$. Hence,
\begin{equation}\label{TJ69}
  \mathbb{E} \mathcal{D}_2(T_N)=O(\log N).
\end{equation}

To bound $\mathbb{E} \mathcal{D}_3$ first we derive with a help of (\ref{TJab})
\[
 		\sum_{s=1}^N \sum_{r=1}^N p_r \mathbb{E} \left|{N_r}(v_j) \cap \mathcal{N}_s(v_i) \cap \mathcal{N}(v_l,l) 
 		\right| =
                \mathbb{E} \left\{ \left| \mathcal{N}(v_j) \cap \mathcal{N}(v_i) \cap \mathcal{N}(v_l,l) \right| \right\}
\]\[
  \leq \sum\limits_{a, a', b \in V_N} \sum\limits_{u \in V_N} p(u, a') p(u, b) p(u, a)  \mathbb{P} \left (  v_l = a , v_i = a', v_j =  b \right )\]
\[
  \leq \sum\limits_{a, a', b \in V_N} \sum\limits_{u \in V_N} p(u, a') p(u, b) p(u, a)
p(a,a')p(a',b)
O \left ( \frac{\log^2 N}{N^2} \right )\]
\[
  =   \sum\limits_{u,b \in V_N}p(u, b)
  \sum\limits_{a'\in V_N}p(u, a') p(a',b)  \left( \sum\limits_{a \in V_N}
    p(u, a)
p(a,a')\right) 
O \left ( \frac{\log^2 N}{N^2} \right ).\]
Applying here  two times result (\ref{tt1}) we derive
\[
 		\sum_{s=1}^N \sum_{r=1}^N p_r \mathbb{E} \left|{N_r}(v_j) \cap \mathcal{N}_s(v_i) \cap \mathcal{N}(v_l,l) 
 		\right| \leq O \left ( \frac{\log^4 N}{N^4} \right ),
              \]
which yields (by definition of $\mathcal{D}_3$ in (\ref{TSe10})) that 
\[\mathbb{E} \mathcal{D}_3(T_N) \leq \mathbb{E} \mathcal{D}_3([N^{4/3}T]
)\]
\begin{equation}\label{TJ70}
  = \sum\limits_{j=2}^{[N^{4/3}T]-1}
  \sum_{i=1}^{j-1} \sum_{l=1}^{i-1} \sum_{s=1}^N \sum_{r=1}^N p_r \mathbb{E} \left|{N_r}(v_j) \cap \mathcal{N}_s(v_i) \cap \mathcal{N}(v_l,l) 
 		\right|
\end{equation}
\[\leq [N^{4/3}T]^3 O \left ( \frac{\log^4 N}{N^{4}} \right ) = O \left ( \log^4 N \right ). \]

For the last term $ \mathbb{E} \mathcal{D}_4$ first we derive,  making again use of (\ref{TJab})
\[ \mathbb{P} \left\{v_l \in \mathcal{N}(v_j) \cap \mathcal{N}(v_i) \right\}
= \sum\limits_{a, a', b \in V_N}  p(a, b) p(a, a') \mathbb{P} \left\{v_l = a, v_i = a', v_j = b \right\}
\]\[
  \leq \sum\limits_{a, a', b \in V_N}p(a, b) p(a, a') 
p(a,a')p(a',b)
O \left ( \frac{\log^2 N}{N^2} \right ).
\]
Summing over $b$ we use bound (\ref{tt1}), which gives us 
\[ \mathbb{P} \left\{v_l \in \mathcal{N}(v_j) \cap \mathcal{N}(v_i) \right\}
  \leq \sum\limits_{a, a' \in V_N} p^2(a, a') 
O \left ( \frac{\log^3 N}{N^4} \right )
\]
uniformly in $l<i<j<T_N$.
Finally, applying here Proposition \ref{Prop1} we get
\[ \mathbb{P} \left\{v_l \in \mathcal{N}(v_j) \cap \mathcal{N}(v_i) \right\}
  \leq 
O \left ( \frac{\log^4 N}{N^4} \right ),
\]
which yields
\[\mathbb{E} \mathcal{D}_4(T_N)\leq\mathbb{E} \mathcal{D}_4([N^{4/3}T])\]
\begin{equation}\label{TJ71}
 = \sum\limits_{j=2}^{[N^{4/3}T]-1}  \sum_{r=1}^N p_r\sum_{i=1}^{j-1} 
  \sum_{l=1}^{i-1}\mathbb{P} \left\{
    {\cal A}_{l-1}= \emptyset ,  v_l \in {N_r}(v_j) \cap \mathcal{N}(v_i) 
 	 \right\}
 \end{equation}		
\[=
  \sum\limits_{j=2}^{[N^{4/3}T]-1}  \sum_{i=1}^{j-1} 
  \sum_{l=1}^{i-1}\mathbb{P} \left\{
    {\cal A}_{l-1}= \emptyset ,  v_l \in { \mathcal{N}}(v_j) \cap \mathcal{N}(v_i) 
  \right\}\]
\[=O \left ( (N^{4/3}T)^3 \frac{\log^4 N}{N^4} \right ) =O \left ( \log^4 N \right ).
 		\]

                Combining bounds (\ref{TJ71}),  (\ref{TJ70}), (\ref{TJ69}), and (\ref{TJ67})
                in 
(\ref{TJ66}), we obtain
\begin{equation*}\label{TJ662}
 		\mathbb{E} |\mathcal{D}(T_N)| \leq 2T^2N^{2/3}
\end{equation*}
for all large $N$.

This together with bound (\ref{TJ64}) implies
\[ \dfrac{(\mathbb{E} (\mathcal{M}(T_N))^2)^{1/2} + \mathbb{E} |\mathcal{D}(T_N)|}{K N^{2/3}} \leq 2\frac{\sqrt{T}+T^2}{K}. \]
Hence, 
by $(\ref{Prop_Z_stoch_bound_Eq1})$ 
\[
  \mathbb{P} \left \{
    \max\limits_{k \leq N^{4/3} T} | z(k) | > K N^{2/3}
  \right \}
  < 2\frac{\sqrt{T}+T^2}{K}. \]
                \end{proof}
            
     Note that $| z(k) |\leq N^2$ for all $k<N^2$. Therefore the result the last Lemma \ref{Lbz} allows us to derive
    \[
   \mathbb{E} \left(
   \frac{\max\limits_{k \leq N^{4/3} T} | z(k) |}{N^{2/3}}\right)=
   O\left( \sum_{K=1}^{N^{4/3}}
   \mathbb{P} \left \{
   \max\limits_{k \leq N^{4/3} T} | z(k) | > K N^{2/3}
   \right \}
   \right)\]
   \begin{equation*}\label{TJ77}
   =
   O\left(\log N\right).
      \end{equation*}
    Making use of this bound in (\ref{TWA})
we obtain also a bound for the expectation
\begin{equation}\label{TA1}
\mathbb{E}
\mathcal{D}_2 
	\left (1 + k \right ) 
\leq O\left(\frac{\log N}{N^2} \right)k
O\left
(N^{2/3} \log N  \right)= O\left( \log^2 N  \right)
\end{equation}
uniformly in $k\leq TN^{4/3}$.

\subsubsection{Proof of Lemma \ref{dr}.}
Now we can prove (\ref{M5}), which states that 
 	\[
 	{\widetilde {\cal D}} (s)
 	= \frac{ {\cal D}(1+ [ N^{4/3} s ])}{N^{2/3}}
 	 \stackrel{P}{\rightarrow} -
 	\frac{1}{2}
 	s^2
 	\]
 	uniformly in $s\in [0,T]$
 	as $N\rightarrow \infty$.
 	
 By  (\ref{TSe10}) we have
 	\begin{eqnarray}\label{TSeA10}
 	&& {\widetilde {\cal D}} (s)
 	= \frac{ {\cal D}\left(1+ [N^{4/3} s ]
 	\right)}{N^{2/3}}
 	 =-\frac{  N^{4/3} s }{N^{2/3}}
 	 \left(\dfrac{2c}{N} +O\left(\dfrac{1}{N^2}\right) \right) \\ \nonumber
 	&& -  	\dfrac{\mathcal{D}_1 \left (1 + [N^{4/3} s] \right )  }{N^{2/3}} 
 	-\dfrac{\mathcal{D}_2 \left (1 + [N^{4/3} s] \right )  }{N^{2/3}}
 	\\ \nonumber
 	&&+\dfrac{\mathcal{D}_3 \left (1 + [N^{4/3} s] \right )  }{N^{2/3}}
 	+\dfrac{\mathcal{D}_4 \left (1 + [N^{4/3} s] \right )  }{N^{2/3}}.
 	\end{eqnarray}
 	
 	By (\ref{TA1}) we have
 	\begin{equation}\label{D2}
 	\dfrac{\mathcal{D}_2 \left (1 + [N^{4/3} s] \right )  }{N^{2/3}}
 	\stackrel{P}{\rightarrow} 0.
 	\end{equation}
 Also, by (\ref{TJ71})
 \begin{equation}\label{D3}
 \dfrac{\mathcal{D}_3 \left (1 + [N^{4/3} s] \right )  }{N^{2/3}}
 \stackrel{P}{\rightarrow} 0
 \end{equation}
 and by (\ref{TJ70})
 \begin{equation}\label{D4}
 \dfrac{\mathcal{D}_4 \left (1 + [N^{4/3} s] \right )  }{N^{2/3}}
 \stackrel{P}{\rightarrow} 0,
 \end{equation}
 while by Corollary \ref{D1}
 \begin{equation}\label{D}
 \dfrac{\mathcal{D}_1 \left (1 + [N^{4/3} s] \right )  }{N^{2/3}}
 \stackrel{P}{\rightarrow} \frac{s^2}{2}.
 \end{equation}
 The last four assertions (\ref{D2})-(\ref{D}) together with 
 (\ref{TSeA10}) yield statement (\ref{M5}).
 \hfill$\Box$
 
 \subsection{Proof of Lemma \ref{ma}}\label{SectionProof27}
 We shall prove $(\ref{M4})$, which states that
 \begin{equation*}\label{M4proof}
 	{\widetilde {\cal M}} (s)  = \dfrac{{\cal M}(1+ [ |V_N|^{2/3} s ] )}{|V_N|^{1/3}}	
 	\stackrel{d}{\rightarrow}
 	\ W(s),
 \end{equation*}
 uniformly in $s \in [0, T]$ as $N \rightarrow \infty$.
 Recall that $|V_N|=N^2$ and we also use notation $n=|V_N|.$
 Here we follow closely approach of \cite{Aldous} (as it was adapted for inhomogeneous case in \cite{Turova})
 utilizing the results of the previous sections.
 
   Let us define a quadratic variation process
 \begin{equation*}
 	A(k) = \sum\limits_{j = 1}^{k-1} \mathbb{E} \left ( |\Delta \mathcal{M} (j)|^2 \big| \mathcal{F}_{j}^{z}\right ),
 \end{equation*}
 and its normalized version
 \begin{equation}\label{Atilda}
 	\widetilde{A}_n(s) = n^{-2/3} A(1 + [ n^{2/3} s]) = n^{-2/3} \sum\limits_{j = 1}^{[ n^{2/3} s]} \mathbb{E} \left ( |\Delta \mathcal{M} (j)|^2 \big| \mathcal{F}_{j}^{z}\right ) .
 \end{equation}
 Then 
 \begin{equation}\label{EqMar}
 	\mathcal{M}^{2}(k) - A(k), \; k \geq 1
 \end{equation}
 is a martingale. 

 		To prove $(\ref{M4})$ we apply the functional central limit theorem for martingales (from \cite{EthierKurtz} Chapter 7, Theorem 1.4 and Remark 1.5). Since $(\ref{EqMar})$ is a martingale, we need to prove the following three statements: for each $s > 0$
 		\begin{eqnarray}
 		&&\lim\limits_{n \rightarrow \infty} \mathbb{E} \left( \sup\limits_{t \leq s} \left | \widetilde{A}_n(t) - \widetilde{A}_n(t-) \right | \right) = 0, \label{Prop_Martingale_eq1} \\
 		&&\lim\limits_{n \rightarrow \infty} \mathbb{E} \left( \sup\limits_{t \leq s} \left | \widetilde{\mathcal{M}}_n(t) - \widetilde{\mathcal{M}}_n(t-) \right |^2 \right) = 0, \label{Prop_Martingale_eq2} \\
 		&&\widetilde{A}_n(s) \stackrel{P}{\rightarrow} s, \text{ as }n\rightarrow \infty . \label{Prop_Martingale_eq3}
 		\end{eqnarray}
 		
 		Consider the expectation in $(\ref{Prop_Martingale_eq2})$
 		\begin{eqnarray}\label{EqMar2}
 		\mathbb{E} \left( \sup\limits_{t \leq s} \left | \widetilde{\mathcal{M}}_n(t) - \widetilde{\mathcal{M}}_n(t-) \right |^2 \right) = \dfrac{1}{N^{4/3}} \mathbb{E} \sup\limits_{j \leq s N^{4/3}} \left ( \Delta \mathcal{M} (j) \right )^2 .
 		\end{eqnarray}
 		By $(\ref{Var_est})$ and $(\ref{Var_est2})$
 		\begin{eqnarray*}
 		&& \mathbb{E} \left ( \Delta \mathcal{M} (j) \right )^2 \leq 1 + O(1/N) .
 		\end{eqnarray*}
 		
 		Recall from  $(\ref{eta})$ and $(\ref{meanE})$ that $|\mathcal{N} (v_j, j)|$ is stochastically dominated by Poisson distributed variable with mean $1$. Denote such a Poisson variable as $\eta$ and consider, writing  $\varphi(k)=\sqrt{k}$ 
 		\begin{eqnarray}\label{EqMar1}
 		&& \mathbb{E} \sup\limits_{j \leq k} \left ( \Delta \mathcal{M} (j) \right )^2 \leq \mathbb{E} \sup\limits_{j \leq k} | \mathcal{N} (v_j, j) |^2 \\ \nonumber
 		&& \leq \varphi (k) \mathbb{P} \left ( \sup\limits_{j \leq k} | \mathcal{N} (v_j, j) |^2 \leq \varphi (k) \right ) + \varphi (k) \sum\limits_{l > \varphi(k)} \mathbb{P} \left ( \sup\limits_{j \leq k} | \mathcal{N} (v_j, j) |^2 = l \right ) \\ \nonumber
 		&& + \sum\limits_{l > \varphi(k)} \mathbb{P} \left ( \sup\limits_{j \leq k} | \mathcal{N} (v_j, j) |^2 \geq l \right ) \leq \varphi(k) + \sum\limits_{l > \varphi(k)} \mathbb{P} \left ( \sup\limits_{j \leq k} | \mathcal{N} (v_j, j) |^2 \geq l \right ) \\ \nonumber
 		&& \leq \varphi(k) + \sum\limits_{j = 1}^{k} \sum\limits_{l > \varphi(k)} \mathbb{P} \left ( | \mathcal{N} (v_j, j) |^2 \geq l \right ) \leq \varphi(k) + \sum\limits_{j = 1}^{k}\sum\limits_{l > \varphi(k)} \dfrac{\mathbb{E} \eta^3}{l^{3/2}} \\ \nonumber
 		&& \leq \varphi(k) + \dfrac{5 k}{2 \sqrt{\varphi(k)}} \left ( 1 + o(1) \right ) ,
 		\end{eqnarray}
 		since $\mathbb{E} \eta^{3} = 5$.
 		
 		Setting $k = s N^{4/3}$  in $(\ref{EqMar1})$ we derive 
 		\begin{equation*}
 		\dfrac{1}{N^{4/3}} \mathbb{E} \sup\limits_{j \leq s N^{4/3}} \left ( \Delta \mathcal{M} (j) \right )^2 \leq \dfrac{1}{N^{4/3}} \left ( \sqrt{s N^{4/3}} + \dfrac{5 s N^{4/3}}{2 (s N^{4/3})^{1/4}} \right ) \left ( 1 + o(1) \right ) \rightarrow 0 ,
 		\end{equation*}
 		which proves $(\ref{Prop_Martingale_eq2})$.
 		
 		Consider $\widetilde{A}_n(s)$ from $(\ref{Atilda})$
 		\begin{eqnarray}
 		&& \widetilde{A}_n(s) = n^{-2/3} \sum\limits_{j = 1}^{[ n^{2/3} s]} \mathbb{E} \left ( |\Delta \mathcal{M} (j)|^2 \big| \mathcal{F}_{j}^{z}\right ) \nonumber \\ \nonumber
 		&& = n^{-2/3} \sum\limits_{j = 1}^{[ n^{2/3} s]} \mathrm{Var} \left ( \Delta z(j) \big| \mathcal{F}_{j}^{z}\right ) .
 		\end{eqnarray}
 		To prove $(\ref{Prop_Martingale_eq3})$ it is enough to prove that 
 		\begin{eqnarray}
 		\mathbb{E} \widetilde{A}_n(s) & \rightarrow & s \label{EqMar14} , \\
 		\mathrm{Var} \widetilde{A}_n(s) & \rightarrow & 0 \label{EqMar15} .
 		\end{eqnarray}
 		Firstly,
 		\begin{eqnarray}\label{EqMar5}
 		&& \mathbb{E} \widetilde{A}_n(s) = \dfrac{1}{N^{4/3}} \sum\limits_{j = 1}^{[ N^{4/3} s]} \mathbb{E} \mathrm{Var} \left ( \Delta z(j) \big| \mathcal{F}_{j}^{z} \right ) \\ \nonumber
 		&& = \dfrac{1}{N^{4/3}} \sum\limits_{j = 1}^{[ N^{4/3} s]} \mathbb{E} \mathrm{Var} \left ( |\mathcal{N} (v_j, j)|\big| \mathcal{F}_{j}^{z} \right ) \\ \nonumber
 		&& = \dfrac{1}{N^{4/3}} \sum\limits_{j = 1}^{[ N^{4/3} s]} \left ( \mathbb{E} |\mathcal{N} (v_j, j)|^2 - \mathbb{E} \left ( \mathbb{E} \left ( |\mathcal{N} (v_j, j)| \big| \mathcal{F}_j^z \right ) \right )^2 \right ) .
 		\end{eqnarray}
 		Recall that by Lemma \ref{Lm_Est_ENvjj}
 		\begin{equation}\label{EqMar4}
 			\mathbb{E} |\mathcal{N} (v_j, j)|^2 = 2 + O(N^{-1/4}).
 		\end{equation}
 		and by Lemma \ref{Lm_Est_ENvjjFjz} 
 		\begin{eqnarray*}
 			\mathbb{E} \left ( |\mathcal{N} (v_j, j)| \big| \mathcal{F}_j^z \right ) = 1 + O(1/N) + O \left (\dfrac{\sqrt{|I_{j-1}|}}{N} \right ) .
 		\end{eqnarray*}
 		Taking also into account  $(\ref{tm12})$ we derive
 		\begin{eqnarray}\label{EqMar3}
 			\mathbb{E} \left ( \mathbb{E} \left ( |\mathcal{N} (v_j, j)| \big| \mathcal{F}_j^z \right ) \right )^2 = 1 + O(N^{-1/4}) .
 		\end{eqnarray}
 		Applying $(\ref{EqMar3})$ and $(\ref{EqMar4})$ to $(\ref{EqMar5})$ we get
 		\begin{equation*}
 		\mathbb{E} \widetilde{A}_n(s) = s + o(1) ,
 		\end{equation*}
 		which proves $(\ref{EqMar14})$.
 		
 		Next for the variance in $(\ref{EqMar15})$ we have
 		\begin{eqnarray}\label{EqMar6}
 		&& \mathrm{Var} \widetilde{A}_n(s) = \mathrm{Var} \left ( \dfrac{1}{N^{4/3}} \sum\limits_{j = 1}^{[ N^{4/3} s]} \mathrm{Var} \left ( |\mathcal{N} (v_j, j)| \big| \mathcal{F}_j^z \right ) \right ) \\ \nonumber
 		&& = \dfrac{1}{N^{8/3}} \left ( \sum\limits_{j = 1}^{[ N^{4/3} s]} \mathrm{Var} \mathrm{Var} \left ( |\mathcal{N} (v_j, j)| \big| \mathcal{F}_j^z \right ) \right . \\ \nonumber
 		&& \left . + 2 \sum\limits_{i = 1}^{[ N^{4/3} s]} \sum\limits_{j = i + 1}^{[ N^{4/3} s]} \mathrm{Cov} \left ( \mathrm{Var} \left ( |\mathcal{N} (v_i, i)| \big| \mathcal{F}_i^z \right ), \mathrm{Var} \left ( |\mathcal{N} (v_j, j)| \big| \mathcal{F}_j^z \right ) \right ) \right ) .
 		\end{eqnarray}
 		Consider a term in the first sum on the right in $(\ref{EqMar6})$:
 		\begin{eqnarray}\label{EqMar7}
 		&& \mathrm{Var} \mathrm{Var} \left ( |\mathcal{N} (v_j, j)| \big| \mathcal{F}_j^z \right ) \\ \nonumber
 		&& = \mathrm{Var} \left ( \mathbb{E} \left ( |\mathcal{N} (v_j, j)|^2 \big| \mathcal{F}_j^z \right ) - \left ( \mathbb{E} \left ( |\mathcal{N} (v_j, j)| \big| \mathcal{F}_j^z \right ) \right )^2 \right ) \\ \nonumber
 		&& = \mathbb{E} \left ( \mathbb{E} \left ( |\mathcal{N} (v_j, j)|^2 \big| \mathcal{F}_j^z \right ) \right )^2 - 2 \mathbb{E} \left ( \mathbb{E} \left ( |\mathcal{N} (v_j, j)|^2 \big| \mathcal{F}_j^z \right ) \left ( \mathbb{E} \left ( |\mathcal{N} (v_j, j)| \big| \mathcal{F}_j^z \right ) \right )^2 \right ) \\ \nonumber
 		&& + \mathbb{E} \left ( \mathbb{E} \left ( |\mathcal{N} (v_j, j)| \big| \mathcal{F}_j^z \right ) \right )^4 - \left ( \mathbb{E} |\mathcal{N} (v_j, j)|^2 \right )^2 \\ \nonumber
 		&& + 2 \mathbb{E} |\mathcal{N} (v_j, j)|^2 \mathbb{E} \left ( \mathbb{E} \left ( |\mathcal{N} (v_j, j)| \big| \mathcal{F}_j^z \right ) \right )^2 - \left ( \mathbb{E} \left ( \mathbb{E} \left ( |\mathcal{N} (v_j, j)| \big| \mathcal{F}_j^z \right ) \right )^2 \right )^2 .
 		\end{eqnarray}
 		We shall estimate all 6 terms in $(\ref{EqMar7})$ separately. With a help of Lemma \ref{Lm_Est_ENvjjFjz} we bound the first term
 		\begin{eqnarray*}
 			&& \mathbb{E} \left ( \mathbb{E} \left ( |\mathcal{N} (v_j, j)|^2 \big| \mathcal{F}_j^z \right ) \right )^2 \leq \mathbb{E} \left ( 2 + O(N^{-1}) \right )^2 =  \nonumber \\ 
 			&& = 4 + O (N^{-1}).
 		\end{eqnarray*}
 		With a help of Lemma \ref{Lm_Est_ENvjjFjz} and $(\ref{tm12})$ we bound the second term
 		\begin{eqnarray}
 			&& - 2 \mathbb{E} \left ( \mathbb{E} \left ( |\mathcal{N} (v_j, j)|^2 \big| \mathcal{F}_j^z \right ) \left ( \mathbb{E} \left ( |\mathcal{N} (v_j, j)| \big| \mathcal{F}_j^z \right ) \right )^2 \right )  \nonumber \\ \nonumber
 			&& \leq -2 \mathbb{E} \left ( \left ( 2 - 4 \dfrac{4 c \sqrt{|I_{j-1}|}}{N} + O \left (\dfrac{1}{N} \right ) \right ) \left ( 1 - \dfrac{4 c \sqrt{|I_{j-1}|}}{N} + O \left (\dfrac{1}{N} \right ) \right )^2 \right )  \\ \nonumber
 			&& = -4 + O (N^{-1/4}) .
 		\end{eqnarray}
 		For the third term on the right in $(\ref{EqMar7})$
 		we have by Lemma \ref{Lm_Est_ENvjjFjz},
 		\begin{eqnarray*}
 			\mathbb{E} \left ( \mathbb{E} \left ( |\mathcal{N} (v_j, j)| \big| \mathcal{F}_j^z \right ) \right )^4 \leq \mathbb{E} \left ( 1 + O(N^{-1}) \right )^4 = 1 + O(N^{-1}) .
 		\end{eqnarray*}
 		while by Lemma \ref{Lm_Est_ENvjj} for the forth one
 		\begin{eqnarray*}
 			 \left ( \mathbb{E} |\mathcal{N} (v_j, j)|^2 \right )^2 =  \mathbb{E} \left ( 2 + O(N^{-1/4}) \right )^2 = 4 + O(N^{-1/4}) .
 		\end{eqnarray*}
 		and by of Lemma \ref{Lm_Est_ENvjj} and Lemma \ref{Lm_Est_ENvjjFjz} for the fifth
 		\begin{eqnarray}
 			&& 2 \mathbb{E} |\mathcal{N} (v_j, j)|^2 \mathbb{E} \left ( \mathbb{E} \left ( |\mathcal{N} (v_j, j)| \big| \mathcal{F}_j^z \right ) \right )^2 \leq 2 \left ( 2 + O(N^{-1/4}) \right ) \mathbb{E} \left ( 1 + O(N^{-1}) \right )^2 \nonumber \\ \nonumber
 			&& = 4 + O(N^{-1/4}) .
 		\end{eqnarray}
 		Finally, Lemma \ref{Lm_Est_ENvjjFjz} and $(\ref{tm12})$ gives us a bound for the last term
 		in $(\ref{EqMar7})$
 		\begin{eqnarray}
 			&& - \left ( \mathbb{E} \left ( \mathbb{E} \left ( |\mathcal{N} (v_j, j)| \big| \mathcal{F}_j^z \right ) \right )^2 \right )^2 \leq - \left ( \mathbb{E} \left ( 1 - \dfrac{4 c \sqrt{|I_{j-1}|}}{N} + O \left (\dfrac{1}{N} \right ) \right )^2 \right )^2 \nonumber \\ \nonumber 
 			&& \leq -1 + O(N^{-1/4}) .
 		\end{eqnarray}
 		Combining all six previous assertions we derive for $(\ref{EqMar7})$  		
 		\begin{eqnarray}\label{EqMar13}
 			&&  \mathrm{Var} \mathrm{Var} \left ( |\mathcal{N} (v_j, j)| \big| \mathcal{F}_j^z \right ) \\ \nonumber
 			&& \leq 4 - 4 + 1 - 4 + 4 - 1 + O(N^{-1/4}) = O(N^{-1/4}) .
 		\end{eqnarray} 		
 		
 		Consider a term in the second sum on the right in $(\ref{EqMar6})$
 		\begin{eqnarray}\label{EqMar8}
 		&& \mathrm{Cov} \left ( \mathrm{Var} \left ( |\mathcal{N} (v_i, i)| \big| \mathcal{F}_i^z \right ), \mathrm{Var} \left ( |\mathcal{N} (v_j, j)| \big| \mathcal{F}_j^z \right ) \right ) \\ \nonumber
 		&& = \mathbb{E}  \left \{ \left ( \mathbb{E} \left ( |\mathcal{N} (v_i, i)|^2 \big| \mathcal{F}_i^z \right ) - \left ( \mathbb{E} \left ( |\mathcal{N} (v_i, i)| \mathcal{F}_i^z \right ) \right )^2 \right ) \right . \\ \nonumber
 		&& \left . \times \left ( \mathbb{E} \left ( |\mathcal{N} (v_j, j)|^2 \big| \mathcal{F}_j^z \right ) - \left ( \mathbb{E} \left ( |\mathcal{N} (v_j, j)| \mathcal{F}_j^z \right ) \right )^2 \right ) \right \} \\ \nonumber
 		&& - \mathbb{E} \left ( \mathrm{Var} \left ( |\mathcal{N} (v_i, i)| \big| \mathcal{F}_i^z \right ) \right ) \mathbb{E} \left ( \mathrm{Var} \left ( |\mathcal{N} (v_j, j)| \big| \mathcal{F}_j^z \right ) \right ) .
 	\end{eqnarray} 	
 	With a help of Lemma \ref{Lm_Est_ENvjjFjz} and the property $|I_{i-1}| \leq |I_{N^{4/3}s}|$ we  get
 	\begin{eqnarray}\label{EqMar9}
 		&& \mathbb{E} \left ( |\mathcal{N} (v_i, i)|^2 \big| \mathcal{F}_i^z \right ) - \left ( \mathbb{E} \left ( |\mathcal{N} (v_i, i)| \mathcal{F}_i^z \right ) \right )^2 \\ \nonumber
 		&& \leq 2 + O \left (\dfrac{1}{N} \right ) - \left ( 1 - \dfrac{4c \sqrt{|I_{N^{4/3}s}|}}{N} + O\left ( \dfrac{1}{N} \right ) \right )^2,
 	\end{eqnarray}
 	and in a similar way 
 	\begin{eqnarray}\label{EqMar10}
 		&& \mathbb{E} \left ( \mathrm{Var} \left ( |\mathcal{N} (v_i, i)| \big| \mathcal{F}_i^z \right ) \right ) \\ \nonumber
 		&& =  \mathbb{E} |\mathcal{N} (v_i, i)|^2 - \mathbb{E} \left ( \mathbb{E} \left ( |\mathcal{N} (v_i, i)| \big| \mathcal{F}_i^z \right ) \right )^2 \\ \nonumber
 		&&  =  1 + O (N^{-1/4} ).
 	\end{eqnarray}
 	Finally, again with a help of Lemma \ref{Lm_Est_ENvjj} and Lemma \ref{Lm_Est_ENvjjFjz}, and taking into account  $(\ref{tm12})$ we derive 
 	\begin{eqnarray}\label{EqMar11}
 		&& \mathbb{E} \left ( \mathrm{Var} \left ( |\mathcal{N} (v_j, j)| \big| \mathcal{F}_j^z \right ) \right ) \\ \nonumber
 		&& = \mathbb{E} |\mathcal{N} (v_j, j)|^2 + \mathbb{E} \left ( \mathbb{E} \left ( |\mathcal{N} (v_j, j)| \big| \mathcal{F}_j^z \right ) \right )^2 \\ \nonumber
 		&& \leq 2 + O \left ( \dfrac{1}{N^{1/4}} \right ) + \mathbb{E} \left ( 1 - \dfrac{4c \sqrt{|I_{j-1}|}}{N}+ O \left ( \dfrac{1}{N} \right ) \right )^2 = 1 + O \left ( \dfrac{1}{N^{1/4}} \right  ).
 	\end{eqnarray}
 		
 	Combining $(\ref{EqMar11})$, $(\ref{EqMar10})$, $(\ref{EqMar9})$, and taking into account $(\ref{tm12})$ we derive for the covariance in $(\ref{EqMar8})$
 	\begin{eqnarray}\label{EqMar12}
 		&& \mathrm{Cov} \left ( \mathrm{Var} \left ( |\mathcal{N} (v_i, i)| \big| \mathcal{F}_i^z \right ), \mathrm{Var} \left ( |\mathcal{N} (v_j, j)| \big| \mathcal{F}_j^z \right ) \right ) \\ \nonumber
 		&& \leq \mathbb{E} \left ( 2 + O \left (\dfrac{1}{N} \right ) - \left ( 1 - \dfrac{4c \sqrt{|I_{N^{4/3}s}|}}{N} + O\left ( \dfrac{1}{N} \right ) \right )^2 \right )^2 \\ \nonumber
 		&& - \left ( 1 +  O \left ( \dfrac{1}{N^{1/4}} \right  ) \right )^2 = O(N^{-1/4}) .
 		\end{eqnarray}
 		
 		Equipped with $(\ref{EqMar12})$ and $(\ref{EqMar13})$ we derive from $(\ref{EqMar6})$
 		\begin{eqnarray}
 		&& \mathrm{Var} \widetilde{A}_n(s) \leq \dfrac{1}{N^{8/3}} \left ( N^{4/3} s  \ O(N^{-1/4}) + N^{8/3} s^2  \ O(N^{-1/4}) \right ) \nonumber \\ \nonumber
 		&& = s^2 O(N^{-1/4}) \rightarrow 0,
 		\end{eqnarray}
 		which proves $(\ref{EqMar15})$. 		
 		Thus, 
 		$(\ref{Prop_Martingale_eq3})$ follows by
 		 $(\ref{EqMar14})$ and $(\ref{EqMar15})$.
 		
 		Consider now 
 		
 		Finally, $(\ref{Prop_Martingale_eq1})$
 		is derived with a help of Lemma \ref{Lm_Est_ENvjjFjz} as follows
 		\begin{eqnarray}
 		&& \mathbb{E} \left( \sup\limits_{t \leq s} \left | \widetilde{A}_n(t) - \widetilde{A}_n(t-) \right | \right) \nonumber \\ \nonumber
 		&& = \dfrac{1}{N^{4/3}}\mathbb{E} \left( \sup\limits_{j \leq N^{4/3} s} \mathrm{Var} \left ( |\mathcal{N} (v_j, j)|\big| \mathcal{F}_{j}^{z} \right ) \right) \\ \nonumber
 		&& \leq \dfrac{1}{N^{4/3}}\mathbb{E} \left( \sup\limits_{j \leq N^{4/3} s} \mathbb{E} \left ( |\mathcal{N} (v_j, j)|^2 \big| \mathcal{F}_{j}^{z} \right ) \right) \leq \dfrac{1}{N^{4/3}} \left ( 2 + O(N^{-1}) \right ) \rightarrow 0 .
 		\end{eqnarray}
 		This ends the proof of $(\ref{M4})$.
                \hfill$\Box$
\bigskip
                
                As it was stated earlier the proofs of Lemma \ref{dr} and Lemma \ref{ma} complete the proof of Theorem \ref{T}.
                \hfill$\Box$

\section*{Acknowledgements}
The authors thank Sreekar Vadlamani for the inspiring discussions at the beginning of the project.

\noindent
The work is supported by Swedish Research Council (VR).


\begin{thebibliography}{99}
 			
 			\bibitem{ANT}
 			Ajazi, F., Napolitano, G. M., and Turova, T. (2017)
 			\textit{Phase transition in random distance graphs on the torus}.
 			Journal of Applied Probability, \textbf{54}, 1278-1294
 			
 			\bibitem{Aldous}
 			Aldous, D. (1997)
 			\textit{Brownian Excursions, Critical Random Graphs and the Multiplicative Coalescent}.
 			The Annals of Probability, \textbf{25} (2), 812-854
 			
 		\bibitem{BHL}
 		Bhamidi, S., 
 		van der Hofstad, R., and van Leeuwaarden, J.S.H. (2010)	
 		\textit{ Scaling
 		limits for critical inhomogeneous random graphs with finite third moments.} Electron. J.
 		Probab. \textbf{15}, 1682-1703
 		
 		\bibitem{BHL2}
 		Bhamidi, S., 
 		van der Hofstad, R., and van Leeuwaarden, J.S.H. (2012)	
 		\textit{Novel
 		scaling limits for critical inhomogeneous random graphs}. Ann. Probab. \textbf{40}, 2299-2361	
 			
 				\bibitem{BolJanRio2}
 			Bollob{\'a}s, B., Janson, S., and Riordan, O. (2007) 
 			\textit{The phase transition in	inhomogeneous random graphs}.
 			Random Structures \& Algorithms, \textbf{31}, 3-122
 			
 			
 			\bibitem{BolJanRio1}
 			Bollob{\'a}s, B., Janson, S., and Riordan, O. (2007)
 			\textit{Spread-out percolation in $\mathbb{R}^d$}.
 			Random Structures \& Algorithms, \textbf{31}, 239-246
 			
 			
 		
 		
 		
 		\bibitem{CG}
 		Conchon-Kerjan, G., and  Goldschmidt, C. (2023) \textit{The stable graph: The metric space scaling limit of a critical random graph with i.i.d. power-law degrees.} {Ann. Probab.}  \textbf{51} (1), 1-69
 		
 		
 		\bibitem{CS}	
 		Chen, L.-C., and Sakai, A. (2015) \textit{Critical two-point functions for long-range statistical-mechanical models in high dimensions}. Ann. Probab. \textbf{43} (2), 639-681
 		
 			\bibitem{DHH}
 			Deijfen, M., van der Hofstad, R., and Hooghiemstra, G. (2013)
 			\textit{ Scale-free percolation}.
 				Ann. I. H. Poincare Probab. Statist., \textbf{49}, 817-838
 				
 			\bibitem{DLV}
 			Dembo, A., Levit, A., and Vadlamani, S. (2019) \textit{ Component sizes for large quantum 	Erd\H{o}s-R\'{e}nyi graph near criticality,}  Ann. Probab. \textbf{47} (2), 1185-1219
 			
 			\bibitem{DHL}
 			Dhara, S., van der Hofstad, R., and van Leeuwaarden, J.S.H. (2021) \textit{Critical  Percolation on Scale-Free Random Graphs: New Universality Class for the Configuration Model}. {Common. Math. Phys.} \textbf{382}, 123-171
 			
 			\bibitem{EthierKurtz}  
 			Ethier, S. N., and Kurtz, T. G. (1986) 
 			\textit{Markov processes: Characterization and convergence}.
 			John Wiley \& Sons.
 			
 			\bibitem{GS}
 			Goldschmidt, C., and Stephenson, R.
 			(2023)
 			\textit{The scaling limit of a critical random directed graph}.
 			Ann. Appl. Probab. \textbf{33} (3), 2024-2065
 			
 			
 			\bibitem{JK}
 			Janson, S., Kozma, R., Ruszink{\'o}, M., and  Sokolov, Y. (2019) 
 			\textit{A modified
 				bootstrap percolation on a random graph coupled with a lattice}.
 			Discrete Applied Mathematics, \textbf{258}, 152-165
 			
 			\bibitem{LevinPeres}
 			Levin, D. A., and Peres, Y. (2017)
 			\textit{Markov Chains and Mixing Times}.
 			American Mathematical Society
 			
 			\bibitem{Lindvall}
 			Lindvall, T. (1976)
 			\textit{On the Maximum of a Branching Process}.
 			Scandinavian Journal of Statistics, \textbf{3} (4), 209-214
 			
 			\bibitem{MartinLof}
 			Martin-L{\"o}f, A. (1998) 
 			\textit{The final size of a nearly critical epidemic, and the first passage time of a Wiener process to a parabolic barrier}.
 			Journal of Applied Probability, \textbf{35}, 671-682
 			
 			\bibitem{NT}
 			Napolitano, G.M., and Turova, T. (2019) 
 			\textit{Geometric Random Graphs vs Inhomogeneous Random Graphs}.
 			Markov Processes And Related Fields, \textbf{25} (4), 615-638
 			
 			
 			\bibitem{P} Penrose, M. (2003)
 			\textit{Random geometric graphs}.
 			Oxford Studies in Probability, \textbf{5}. Oxford University Press, Oxford. 
 			
 			\bibitem{Turova}
 			Turova, T. (2013) 
 			\textit{Diffusion Approximation for the Components in Critical Inhomogeneous Random Graphs of Rank 1}.
 			Random Structures \& Algorithms, \textbf{43}, 486-539
 			
 		\end{thebibliography}
\end{document}